\documentclass[11pt]{amsart}

\usepackage[top=3.75cm, bottom=3cm, left=3cm, right=3cm]{geometry}
\usepackage{amsmath}
\usepackage{amssymb}
\usepackage{amscd}
\usepackage{color}
\usepackage[colorlinks=true, citecolor=citation, urlcolor=citation, linkcolor=reference]{hyperref}
\usepackage{mathrsfs}

\usepackage{eucal}
\usepackage{upgreek}
\usepackage{bm}

\usepackage[normalem]{ulem}

\usepackage{extpfeil, extarrows } % for longarrows

\usepackage{xy}
\xyoption{all}

\newcommand{\nc}{\newcommand}

\definecolor{citation}{rgb}{0,.40,.80}
\definecolor{reference}{rgb}{.80,0,.40}

\renewcommand{\AA}{{\mathbb{A}}}
\nc{\CC}{{\mathbb{C}}}
\nc{\LL}{{\mathbb{L}}}
\nc{\RR}{{\mathbb{R}}}
\renewcommand{\P}{{\mathbb{P}}}
\nc{\OO}{{\mathbb{O}}}

\nc{\QQ}{{\mathbb{Q}}}
\nc{\ZZ}{{\mathbb{Z}}}

\nc{\cA}{{\mathcal{A}}}
\nc{\cB}{{\mathcal{B}}}
\nc{\cC}{{\mathcal{C}}}
\nc{\cD}{{\mathcal{D}}}
\nc{\tcD}{{\tilde{\mathcal{D}}}}
\nc{\scD}{{\mathscr{D}}}
\nc{\cE}{{\mathcal{E}}}
\nc{\cF}{{\mathcal{F}}}
\nc{\cG}{{\mathcal{G}}}
\nc{\cH}{{\mathcal{H}}}
\nc{\cI}{{\mathcal{I}}}
\nc{\cJ}{{\mathcal{J}}}
\nc{\cK}{{\mathcal{K}}}
\nc{\cL}{{\mathcal{L}}}
\nc{\cM}{{\mathcal{M}}}
\nc{\cN}{{\mathcal{N}}}
\nc{\cO}{{\mathcal{O}}}
\nc{\cP}{{\mathcal{P}}}
\nc{\cQ}{{\mathcal{Q}}}
\nc{\cR}{{\mathcal{R}}}
\nc{\cS}{{\mathcal{S}}}
\nc{\cT}{{\mathcal{T}}}
\nc{\cU}{{\mathcal{U}}}
\nc{\cV}{{\mathcal{V}}}
\nc{\cW}{{\mathcal{W}}}
\nc{\cX}{{\mathcal{X}}}
\nc{\cY}{{\mathcal{Y}}}
\nc{\cZ}{{\mathcal{Z}}}

\nc{\rc}{{\mathrm{c}}}
\nc{\rd}{{\mathrm{d}}}
\nc{\rw}{{\mathrm{w}}}

\nc{\rA}{{\mathrm{A}}}
\nc{\rB}{{\mathrm{B}}}
\nc{\rC}{{\mathrm{C}}}
\nc{\rD}{{\mathrm{D}}}
\nc{\rE}{{\mathrm{E}}}
\nc{\rF}{{\mathrm{F}}}
\nc{\rG}{{\mathrm{G}}}
\nc{\rK}{{\mathrm{K}}}
\nc{\rQ}{{\mathrm{Q}}}
\nc{\rP}{{P}}
\nc{\rR}{{\mathrm{R}}}
\nc{\rS}{{\mathrm{S}}}

\nc{\bA}{{\mathbf{A}}}
\nc{\bB}{{\mathbf{B}}}
\nc{\bC}{{\mathbf{C}}}
\nc{\bD}{{\mathbf{D}}}
\nc{\bE}{{\mathbf{E}}}
\nc{\bF}{{\mathbf{F}}}
\nc{\bG}{{\mathbf{G}}}
\nc{\bH}{{\mathbf{H}}}
\nc{\bI}{{\mathbf{I}}}
\nc{\bJ}{{\mathbf{J}}}
\nc{\bK}{{\mathbf{K}}}
\nc{\bL}{{\mathbf{L}}}
\nc{\bM}{{\mathbf{M}}}
\nc{\bN}{{\mathbf{N}}}
\nc{\bO}{{\mathbf{O}}}
\nc{\bP}{{\mathbf{P}}}
\nc{\bQ}{{\mathbf{Q}}}
\nc{\bR}{{\mathbf{R}}}
\nc{\bS}{{\mathbf{S}}}
\nc{\bT}{{\mathbf{T}}}
\nc{\bU}{{\mathbf{U}}}
\nc{\bV}{{\mathbf{V}}}
\nc{\bW}{{\mathbf{W}}}
\nc{\bX}{{\mathbf{X}}}
\nc{\bY}{{\mathbf{Y}}}
\nc{\bZ}{{\mathbf{Z}}}

\nc{\ba}{{\mathbf{a}}}
\nc{\bb}{{\mathbf{b}}}
\nc{\bc}{{\mathbf{c}}}
\nc{\bd}{{\mathbf{d}}}
\nc{\be}{{\mathbf{e}}}
\nc{\bg}{{\mathbf{g}}}
\nc{\bh}{{\mathbf{h}}}
\nc{\bi}{{\mathbf{i}}}
\nc{\bj}{{\mathbf{j}}}
\nc{\bk}{{\mathbf{k}}}
\nc{\bl}{{\mathbf{l}}}
\nc{\bn}{{\mathbf{n}}}
\nc{\bo}{{\mathbf{o}}}
\nc{\bp}{{\mathbf{p}}}
\nc{\bq}{{\mathbf{q}}}
\nc{\br}{{\mathbf{r}}}
\nc{\bs}{{\mathbf{s}}}
\nc{\bt}{{\mathbf{t}}}
\nc{\bu}{{\mathbf{u}}}
\nc{\bv}{{\mathbf{v}}}
\nc{\bfw}{{\mathbf{w}}}
\nc{\bx}{{\mathbf{x}}}
\nc{\by}{{\mathbf{y}}}
\nc{\bz}{{\mathbf{z}}}

\nc{\fA}{{\mathfrak{A}}}
\nc{\fB}{{\mathfrak{B}}}
\nc{\fC}{{\mathfrak{C}}}
\nc{\fD}{{\mathfrak{D}}}
\nc{\fE}{{\mathfrak{E}}}
\nc{\fF}{{\mathfrak{F}}}
\nc{\fG}{{\mathfrak{G}}}
\nc{\fH}{{\mathfrak{H}}}
\nc{\fI}{{\mathfrak{I}}}
\nc{\fJ}{{\mathfrak{J}}}
\nc{\fK}{{\mathfrak{K}}}
\nc{\fL}{{\mathfrak{L}}}
\nc{\fM}{{\mathfrak{M}}}
\nc{\fN}{{\mathfrak{N}}}
\nc{\fO}{{\mathfrak{O}}}
\nc{\fP}{{\mathfrak{P}}}
\nc{\fQ}{{\mathfrak{Q}}}
\nc{\fR}{{\mathfrak{R}}}
\nc{\fS}{{\mathfrak{S}}}
\nc{\fT}{{\mathfrak{T}}}
\nc{\fU}{{\mathfrak{U}}}
\nc{\fV}{{\mathfrak{V}}}
\nc{\fW}{{\mathfrak{W}}}
\nc{\fX}{{\mathfrak{X}}}
\nc{\fY}{{\mathfrak{Y}}}
\nc{\fZ}{{\mathfrak{Z}}}

\nc{\fa}{{\mathfrak{a}}}
\nc{\fb}{{\mathfrak{b}}}
\nc{\fc}{{\mathfrak{c}}}
\nc{\fd}{{\mathfrak{d}}}
\nc{\fe}{{\mathfrak{e}}}
\nc{\ff}{{\mathfrak{f}}}
\nc{\fg}{{\mathfrak{g}}}
\nc{\fh}{{\mathfrak{h}}}
\nc{\fj}{{\mathfrak{j}}}
\nc{\fk}{{\mathfrak{k}}}
\nc{\fl}{{\mathfrak{l}}}
\nc{\fm}{{\mathfrak{m}}}
\nc{\fn}{{\mathfrak{n}}}
\nc{\fo}{{\mathfrak{o}}}
\nc{\fp}{{\mathfrak{p}}}
\nc{\fq}{{\mathfrak{q}}}
\nc{\fr}{{\mathfrak{r}}}
\nc{\fs}{{\mathfrak{s}}}
\nc{\ft}{{\mathfrak{t}}}
\nc{\fu}{{\mathfrak{u}}}
\nc{\fv}{{\mathfrak{v}}}
\nc{\fw}{{\mathfrak{w}}}
\nc{\fx}{{\mathfrak{x}}}
\nc{\fy}{{\mathfrak{y}}}
\nc{\fz}{{\mathfrak{z}}}

\nc{\sA}{{\mathsf{A}}}
\nc{\sB}{{\mathsf{B}}}
\nc{\sC}{{\mathsf{C}}}
\nc{\sD}{{\mathsf{D}}}
\nc{\sE}{{\mathsf{E}}}
\nc{\sF}{{\mathsf{F}}}
\nc{\sG}{{\mathsf{G}}}
\nc{\sH}{{\mathsf{H}}}
\nc{\sI}{{\mathsf{I}}}
\nc{\sJ}{{\mathsf{J}}}
\nc{\sK}{{\mathsf{K}}}
\nc{\sL}{{\mathsf{L}}}
\nc{\sM}{{\mathsf{M}}}
\nc{\sN}{{\mathsf{N}}}
\nc{\sO}{{\mathsf{O}}}
\nc{\sP}{{\mathsf{P}}}
\nc{\sQ}{{\mathsf{Q}}}
\nc{\sR}{{\mathsf{R}}}
\nc{\sS}{{\mathsf{S}}}
\nc{\sT}{{\mathsf{T}}}
\nc{\sU}{{\mathsf{U}}}
\nc{\sV}{{\mathsf{V}}}
\nc{\sW}{{\mathsf{W}}}
\nc{\sX}{{\mathsf{X}}}
\nc{\sY}{{\mathsf{Y}}}
\nc{\sZ}{{\mathsf{Z}}}

\nc{\sa}{{\mathsf{a}}}
\nc{\sd}{{\mathsf{d}}}
\nc{\se}{{\mathsf{e}}}
\nc{\sg}{{\mathsf{g}}}
\nc{\sh}{{\mathsf{h}}}
\nc{\si}{{\mathsf{i}}}
\nc{\sj}{{\mathsf{j}}}
\nc{\sk}{{\mathsf{k}}}
\nc{\sm}{{\mathsf{m}}}
\nc{\sn}{{\mathsf{n}}}
\nc{\so}{{\mathsf{o}}}
\nc{\sq}{{\mathsf{q}}}
\nc{\sr}{{\mathsf{r}}}
\nc{\st}{{\mathsf{t}}}
\nc{\su}{{\mathsf{u}}}
\nc{\sv}{{\mathsf{v}}}
\nc{\sw}{{\mathsf{w}}}
\nc{\sx}{{\mathsf{x}}}
\nc{\sy}{{\mathsf{y}}}
\nc{\sz}{{\mathsf{z}}}

\nc{\oA}{{\overline{A}}}
\nc{\oB}{{\overline{B}}}
\nc{\oC}{{\overline{C}}}
\nc{\oD}{{\overline{D}}}
\nc{\oE}{{\overline{E}}}
\nc{\oF}{{\overline{F}}}
\nc{\oG}{{\overline{G}}}
\nc{\oH}{{\overline{H}}}
\nc{\oI}{{\overline{I}}}
\nc{\oJ}{{\overline{J}}}
\nc{\oK}{{\overline{K}}}
\nc{\oL}{{\overline{L}}}
\nc{\oM}{{\overline{M}}}
\nc{\oN}{{\overline{N}}}
\nc{\oO}{{\overline{O}}}
\nc{\oP}{{\overline{P}}}
\nc{\oQ}{{\overline{Q}}}
\nc{\oR}{{\overline{R}}}
\nc{\oS}{{\overline{S}}}
\nc{\oT}{{\overline{T}}}
\nc{\oU}{{\overline{U}}}
\nc{\oV}{{\overline{V}}}
\nc{\oW}{{\overline{W}}}
\nc{\oX}{{\overline{X}}}
\nc{\oY}{{\overline{Y}}}
\nc{\oZ}{{\overline{Z}}}

\nc{\oa}{{\overline{a}}}
\nc{\ob}{{\overline{b}}}
\nc{\oc}{{\overline{c}}}
\nc{\od}{{\overline{d}}}
\nc{\of}{{\overline{f}}}
\nc{\og}{{\overline{g}}}
\nc{\oh}{{\overline{h}}}
\nc{\oi}{{\overline{i}}}
\nc{\oj}{{\overline{j}}}
\nc{\ok}{{\overline{k}}}
\nc{\ol}{{\overline{l}}}
\nc{\om}{{\overline{m}}}
\nc{\on}{{\overline{n}}}
\nc{\oo}{{\overline{o}}}
\nc{\op}{{\mathrm{op}}}
\nc{\oq}{{\overline{q}}}
\nc{\os}{{\overline{s}}}
\nc{\ot}{{\overline{t}}}
\nc{\ou}{{\overline{u}}}
\nc{\ov}{{\overline{v}}}
\nc{\ow}{{\overline{w}}}
\nc{\ox}{{\overline{x}}}
\nc{\oy}{{\overline{y}}}
\nc{\oz}{{\overline{z}}}

\nc{\tA}{{\tilde{A}}}
\nc{\tB}{{\tilde{B}}}
\nc{\tC}{{\tilde{C}}}
\nc{\tD}{{\tilde{D}}}
\nc{\tE}{{\tilde{E}}}
\nc{\tF}{{\tilde{F}}}
\nc{\tG}{{\tilde{G}}}
\nc{\tH}{{\tilde{H}}}
\nc{\tI}{{\tilde{I}}}
\nc{\tJ}{{\tilde{J}}}
\nc{\tK}{{\tilde{K}}}
\nc{\tL}{{\tilde{L}}}
\nc{\tM}{{\tilde{M}}}
\nc{\tN}{{\tilde{N}}}
\nc{\tO}{{\tilde{O}}}
\nc{\tP}{{\tilde{P}}}
\nc{\tQ}{{\tilde{Q}}}
\nc{\tR}{{\tilde{R}}}
\nc{\tS}{{\tilde{S}}}
\nc{\tT}{{\tilde{T}}}
\nc{\tU}{{\tilde{U}}}
\nc{\tV}{{\tilde{V}}}
\nc{\tW}{{\tilde{W}}}
\nc{\tX}{{\tilde{X}}}
\nc{\tY}{{\tilde{Y}}}
\nc{\tZ}{{\tilde{Z}}}

\nc{\ta}{{\tilde{a}}}
\nc{\tb}{{\tilde{b}}}
\nc{\tc}{{\tilde{c}}}
\nc{\td}{{\tilde{d}}}
\nc{\te}{{\tilde{e}}}
\nc{\tf}{{\tilde{f}}}
\nc{\tg}{{\tilde{g}}}
\nc{\ti}{{\tilde{i}}}
\nc{\tj}{{\tilde{j}}}
\nc{\tk}{{\tilde{k}}}
\nc{\tl}{{\tilde{l}}}
\nc{\tm}{{\tilde{m}}}
\nc{\tn}{{\tilde{n}}}
\nc{\tp}{{\tilde{p}}}
\nc{\tq}{{\tilde{q}}}
\nc{\tr}{{\tilde{r}}}
\nc{\ts}{{\tilde{s}}}
\nc{\tu}{{\tilde{u}}}
\nc{\tv}{{\tilde{v}}}
\nc{\tw}{{\tilde{w}}}
\nc{\tx}{{\tilde{x}}}
\nc{\ty}{{\tilde{y}}}
\nc{\tz}{{\tilde{z}}}

\nc{\hA}{{\hat{A}}}
\nc{\hB}{{\hat{B}}}
\nc{\hC}{{\hat{C}}}
\nc{\hD}{{\hat{D}}}
\nc{\hE}{{\hat{E}}}
\nc{\hF}{{\hat{F}}}
\nc{\hG}{{\hat{G}}}
\nc{\hH}{{\hat{H}}}
\nc{\hI}{{\hat{I}}}
\nc{\hJ}{{\hat{J}}}
\nc{\hK}{{\hat{K}}}
\nc{\hL}{{\hat{L}}}
\nc{\hM}{{\hat{M}}}
\nc{\hN}{{\hat{N}}}
\nc{\hO}{{\hat{O}}}
\nc{\hP}{{\hat{P}}}
\nc{\hQ}{{\hat{Q}}}
\nc{\hR}{{\hat{R}}}
\nc{\hS}{{\hat{S}}}
\nc{\hT}{{\hat{T}}}
\nc{\hU}{{\hat{U}}}
\nc{\hV}{{\hat{V}}}
\nc{\hW}{{\hat{W}}}
\nc{\hX}{{\hat{X}}}
\nc{\hY}{{\hat{Y}}}
\nc{\hZ}{{\hat{Z}}}

\nc{\ha}{{\hat{a}}}
\nc{\hb}{{\hat{b}}}
\nc{\hc}{{\hat{c}}}
\nc{\hd}{{\hat{d}}}
\nc{\he}{{\hat{e}}}
\nc{\hf}{{\hat{f}}}
\nc{\hg}{{\hat{g}}}
\nc{\hh}{{\hat{h}}}
\nc{\hi}{{\hat{i}}}
\nc{\hj}{{\hat{j}}}
\nc{\hk}{{\hat{k}}}
\nc{\hn}{{\hat{n}}}
\nc{\ho}{{\hat{o}}}
\nc{\hp}{{\hat{p}}}
\nc{\hq}{{\hat{q}}}
\nc{\hr}{{\hat{r}}}
\nc{\hs}{{\hat{s}}}
\nc{\hu}{{\hat{u}}}
\nc{\hv}{{\hat{v}}}
\nc{\hw}{{\hat{w}}}
\nc{\hx}{{\hat{x}}}
\nc{\hy}{{\hat{y}}}
\nc{\hz}{{\hat{z}}}

\nc{\eps}{\varepsilon}
\nc{\lan}{\big\langle}
\nc{\ran}{\big\rangle}
\nc{\kk}{{\Bbbk}}

\DeclareMathOperator{\Aut}{\mathrm{Aut}}
\DeclareMathOperator{\Hom}{\mathrm{Hom}}
\DeclareMathOperator{\Ext}{\mathrm{Ext}}

\DeclareMathOperator{\Tor}{\mathrm{Tor}}

\DeclareMathOperator{\HOH}{\mathrm{HH}}

\DeclareMathOperator{\Spec}{\mathrm{Spec}}

\DeclareMathOperator{\Cone}{\mathrm{Cone}}

\DeclareMathOperator{\ev}{\mathrm{ev}}

\DeclareMathOperator{\Gr}{\mathrm{Gr}}
\DeclareMathOperator{\OGr}{\mathrm{OGr}}

\DeclareMathOperator{\IGr}{\mathrm{IGr}}

\DeclareMathOperator{\Fl}{\mathrm{Fl}}

\DeclareMathOperator{\Gm}{{\mathbb{G}_{\mathrm{m}}}}
\newcommand{\tGL}{\widetilde{\mathrm{GL}}}
\DeclareMathOperator{\GL}{\mathrm{GL}}

\DeclareMathOperator{\id}{\mathrm{id}}
\DeclareMathOperator{\rank}{\mathrm{rk}}

\DeclareMathOperator{\ind}{\mathrm{ind}}

\newcommand{\Db}{\mathrm{D}^{\mathrm{b}}}

\newcommand{\Dqc}{\mathrm{D}_{\mathrm{qc}}}
\newcommand{\pf}{{\mathrm{perf}}}
\newcommand{\Dp}{\mathrm{D}_{\pf}}

\newcommand{\bal}{\bm{\alpha}}

\newcommand{\bcB}{{\cB_{\cD}}}
\newcommand{\bcR}{{\cR_{\cD}}}

\DeclareMathOperator{\Cl}{{\mathcal{C}\!\ell}}

\DeclareMathOperator{\usdim}{{\overline{Sdim}}}
\DeclareMathOperator{\lsdim}{{\underline{Sdim}}}

\DeclareMathOperator{\hl}{{\mathrm{HL}}}

\newcommand{\uFdim}[1]{\overline{\dim}(#1)}
\newcommand{\lFdim}[1]{\underline{\dim}(#1)}

\newcommand{\Ind}{\mathrm{Ind}}

\newcommand{\Stab}{\mathrm{Stab}}
\newcommand{\PreStab}{\mathrm{PreStab}}

\theoremstyle{plain}

\newtheorem{theorem}{Theorem}[section]

\newtheorem{lemma}[theorem]{Lemma}
\newtheorem{proposition}[theorem]{Proposition}
\newtheorem{corollary}[theorem]{Corollary}

\theoremstyle{definition}

\newtheorem{definition}[theorem]{Definition}

\newtheorem{example}[theorem]{Example}
\newtheorem{setup}[theorem]{Setup}
\newtheorem{remark}[theorem]{Remark}

\numberwithin{equation}{section}

\title{Serre functors and dimensions of residual categories} 

\author{Alexander Kuznetsov}
\address{{\sloppy
\parbox{1.05\textwidth}{
Algebraic Geometry Section, Steklov Mathematical Institute of Russian Academy of Sciences,\\
8 Gubkin str., Moscow 119991 Russia
\\[5pt]
Laboratory of Algebraic Geometry, HSE University, Russia
}\bigskip}}
\email{akuznet@mi-ras.ru}
\date{}

\author{Alexander Perry}
\address{Department of Mathematics, University of Michigan, Ann Arbor, MI 48109 \smallskip}
\email{arper@umich.edu}

\thanks{A.K. was partially supported by the HSE University Basic Research Program. 
A.P. was partially supported by NSF grants DMS-2112747 and DMS-2052750, and a Sloan Research Fellowship.}

\begin{document}

\begin{abstract}
We describe in terms of spherical twists the Serre functors of many interesting semiorthogonal components, called residual categories, 
of the derived categories of projective varieties. 
In particular, we show the residual categories of Fano complete intersections 
are fractional Calabi--Yau up to a power of an explicit spherical twist.
As applications, we compute the Serre dimensions of residual categories of Fano complete intersections, 
thereby proving a corrected version of a conjecture of Katzarkov and Kontsevich, 
and deduce the nonexistence of Serre invariant stability conditions when the degrees of the 
complete intersection do not all coincide.
\end{abstract}

\maketitle

\section{Introduction}

Semiorthogonal decompositions of derived categories of varieties play an increasingly 
important role in algebraic geometry because of their connections to many areas, like 
birational geometry, projective geometry, moduli spaces of sheaves, and mirror symmetry; see~\cite{K14,K21sdf,BM21} for some surveys of the subject. 
A key insight is that the categories appearing in such decompositions should be thought of as \emph{noncommutative} varieties. 
Understanding the homological properties of these noncommutative varieties, 
especially their Serre functors, has been crucial for applications. 
The purpose of this paper is to give a precise formula for the Serre functors and Serre dimensions 
of an interesting class of semiorthogonal components. 

\subsection{Background} 
Recall that the Serre functor~$\bS_X$ of the bounded derived category~$\Db(X)$ of coherent sheaves 
on a smooth proper variety~$X$ is given by tensoring with~$\omega_X[\dim X]$. 
Thus, it combines two of the most important invariants of a variety --- the dimension and canonical bundle. 
Furthermore, the dimension of~$X$ can be abstractly recovered from the Serre functor, using Elagin and Lunts'~\cite{EL} notion
of \emph{upper}~$\usdim(\cC)$ and \emph{lower}~$\lsdim(\cC)$ \emph{Serre dimension} of a category~$\cC$ with Serre functor~$\bS_\cC$; 
roughly speaking, these dimensions measure how fast the maximal and minimal cohomological amplitudes of the functor~$\bS_\cC^{-k}$ 
grow as~\mbox{$k \to \infty$}  (see Definition~\ref{def:sdim} for the precise definition).
When~$\cC = \Db(X)$, 
\cite[Lemma~5.6]{EL} gives equalities
\begin{equation}
\label{SdimX}
\usdim(\Db(X)) = \lsdim(\Db(X)) = \dim(X).
\end{equation}
In particular, the upper and lower Serre dimensions are equal and take integral values.

When~$\cC$ is a semiorthogonal component of~$\Db(X)$, this is no longer true in general. 
For instance, in~\cite[Corollary~4.3]{K04} it was shown that for~$1 \leq d \leq n+1$, 
the derived category of a degree~$d$ hypersurface~$X \subset \P^n$ has a semiorthogonal decomposition
\begin{equation*}
\Db(X) = \langle \cR_X, \cO_X, \cO_X(1), \dots, \cO_X(n-d) \rangle  , 
\end{equation*} 
and if~$X$ is smooth then~$\cR_X$ is \emph{fractional Calabi--Yau} of dimension~$\frac{(n+1)(d-2)}{d}$, i.e.
the Serre functor of~$\cR_X$ satisfies 
\begin{equation*}
\bS_{\cR_X}^d \cong [(n+1)(d-2)].
\end{equation*}
As the terminology suggests, the Serre dimensions of such a category satisfy 
\begin{equation*}
\usdim(\cR_X) = \lsdim(\cR_X) = \frac{(n+1)(d-2)}{d}.
\end{equation*}
So in this example the upper and lower Serre dimensions are still equal, but no longer integral. 

The reason for the fractional Calabi--Yau property in the above example is 
that the derived category of projective space has a particularly nice semiorthogonal decomposition,
namely a \emph{rectangular Lefschetz decomposition}.
By~\cite{K19} a similar result holds more generally for a smooth projective variety~$M$ 
endowed with a rectangular Lefschetz decomposition 
\begin{equation}
\label{eq:dbm-rectangular}
\Db(M) = \langle \cB_M, \cB_M \otimes \cL_M, \dots, \cB_M \otimes \cL_M^{m-1} \rangle 
\end{equation}
with respect to a line bundle~$\cL_M$, and a \emph{spherical functor}
\begin{equation}
\label{eq:phi}
\Phi \colon \Db(X) \to \Db(M)
\end{equation}
from the derived category of another smooth projective variety~$X$. 
Geometrically interesting examples of such functors~$\Phi$ 
are pushforwards along divisorial embeddings and double covers, see~\S\ref{subsec:spherical-examples}.
In this situation and under appropriate assumptions,~\cite[Theorem~3.5]{K19} shows the derived category of~$X$ 
has a semiorthogonal decomposition
\begin{equation}
\label{eq:dbx}
\Db(X) = \langle \cR_X, \cB_X, \cB_X \otimes \cL_X, \dots, \cB_X \otimes \cL_{X}^{m-d-1} \rangle
\end{equation}
where $\cB_X \simeq \cB_M$ and the category~$\cR_X$ is fractional Calabi--Yau. 

Following~\cite{KS20}, we call $\cR_X$ the \emph{residual category} of~$X$. 
Intriguingly, its structure seems to govern much of the geometry of~$X$. 
For example, in the most famous case where~$X \subset \P^5$ is a cubic fourfold, the $2$-Calabi--Yau property 
of~$\cR_X$ leads to a conjectural criterion for rationality in terms of K3 surfaces~\cite[Conjecture~1.1]{Kuz-cubic4fold}, 
as well as a holomorphic symplectic structure on moduli spaces associated to~$X$ \cite{Kuz-symplectic, stability-families}. 

Unfortunately, the existence of a rectangular Lefschetz decomposition~\eqref{eq:dbm-rectangular} is very restrictive, 
so many natural examples are inaccessible by the results of~\cite{K19}. 
For instance, generalizing the case of Fano hypersurfaces, there is a residual category~$\cR_X$ 
associated to any Fano complete intersection~$X \subset \P^n$, 
but the nature of its Serre functor has remained mysterious. 
In this paper we greatly generalize the results of~\cite{K19} to a setting which in particular applies to complete intersections, 
and allows us to compute the Serre dimensions of their residual categories. 

\subsection{Serre functors in the non-rectangular case} 
We relax the assumption~\eqref{eq:dbm-rectangular} as follows.
Instead of asking for a rectangular Lefschetz decomposition, we assume that there is a semiorthogonal decomposition 
\begin{equation}
\label{eq:dbm}
\Db(M) = \langle \cR_M, \cB_M, \cB_M \otimes \cL_M, \dots, \cB_M \otimes \cL_M^{m-1} \rangle
\end{equation} 
consisting of a rectangular part generated by a subcategory~$\cB_M$, 
and a residual category~$\cR_M$. 
As in~\cite{K19} we assume that a spherical functor is given, 
but it turns out more convenient to replace the functor~\eqref{eq:phi} by its left adjoint
\begin{equation}
\label{eq:psi}
\Psi \colon \Db(M) \to \Db(X),
\end{equation}
which is also spherical. 
Recall that if~$\Psi^!$ denotes the right adjoint of~$\Psi$, 
then~$\Psi$ being spherical means the endofunctors~$\bT_{\Psi^!,\Psi}$ and~$\bT_{\Psi, \Psi^!}$ of~$\Db(M)$ and~$\Db(X)$ 
defined by the exact triangles 
\begin{equation*}
\bT_{\Psi^!,\Psi} \to  \id_{{\Db(M)}} \xrightarrow{ \mathrm{unit} } \Psi^! \circ \Psi  
\qquad \text{and} \qquad  
\Psi \circ \Psi^! \xrightarrow{ \mathrm{counit} } \id_{{\Db(X)}} \to \bT_{\Psi,\Psi^!} 
\end{equation*} 
are autoequivalences, called the {\sf spherical twists} of~$\Psi$. 
Our main result shows that under appropriate assumptions the semiorthogonal decomposition~\eqref{eq:dbx} still exists,
the functor~$\Psi$ induces a spherical functor between the residual categories of~$M$ and~$X$, 
and the Serre functors of~$\cR_M$ and~$\cR_X$ can be expressed in terms of the spherical twists of~$\Psi$.
More precisely, we make the following assumptions analogous to those in~\cite[Theorem~3.5]{K19}:  

\begin{itemize}
\item 
There is a line bundle~$\cL_X$ on~$X$ such that~$\Psi$ {\sf intertwines} between the tensor product functors~$- \otimes \cL_M^i$ and~$- \otimes \cL_X^i$,
i.e., there is a collection of functorial isomorphisms 
\begin{equation}
\label{eq:intertwining}
\lambda_i \colon \Psi(- \otimes \cL_M^i) \xrightarrow{\quad \simeq \quad }  \Psi(-) \otimes \cL_X^i,
\qquad 
i \in \ZZ,
\end{equation} 
such that $\lambda_i \circ \lambda_j = \lambda_{i+j}$ for all $i,j \in \ZZ$. 
\item 
The subcategory~$\cB_M$ of the rectangular part of~\eqref{eq:dbm}
is {\sf compatible with the Serre functor}~$\bS_M$, i.e.,
\begin{equation}
\label{eq:bsm-cb}
\bS_M(\cB_M) = \cB_M \otimes \cL_M^{-m} .
\end{equation} 
In other words, $\cB_M$ is invariant under the tensor product with~\mbox{$\omega_M \otimes \cL_M^m$}.  
Note that the parameter~$m$ in this compatibility condition must be equal to the length~$m$ of the rectangular part of~\eqref{eq:dbm}.
\item 
The subcategory~$\cB_M$ is {\sf compatible with the spherical twist~$\bT_{\Psi^!,\Psi}$}, i.e.,
\begin{equation}
\label{eq:btm-cb}
\bT_{\Psi^!,\Psi}(\cB_M) = \cB_M \otimes \cL_M^{-d}
\end{equation} 
for some $1 \le d \le m$. 
\end{itemize}

Under these assumptions we define the functors
\begin{equation*}
\bs_{\cR_M} \coloneqq (\bS_M \circ (-\otimes \cL_M^{m}))\vert_{\cR_M}
\qquad\text{and}\qquad 
\bt_{\cR_M} \coloneqq (\bT_{\Psi^!,\Psi} \circ ( - \otimes \cL_M^d))\vert_{\cR_M}, 
\end{equation*}
check that they are autoequivalences of~$\cR_M$ (Lemmas~\ref{lemma:serre-cr} and~\ref{lemma:twist-compatibility}), 
and prove the following generalization of~\cite[Theorem~3.5]{K19}. 

\begin{theorem}
\label{thm:main}
Assume given 
\begin{itemize}
\item 
a semiorthogonal decomposition~\eqref{eq:dbm} of a smooth projective variety~$M$ with a line bundle~$\cL_M$, and
\item 
a spherical functor~\eqref{eq:psi} to a smooth projective variety~$X$ with a line bundle~$\cL_X$,
\end{itemize}
such that~\eqref{eq:intertwining}, \eqref{eq:bsm-cb}, and~\eqref{eq:btm-cb} hold for some~$1 \le d \le m$. 
Then:
\begin{enumerate}\renewcommand{\theenumi}{\roman{enumi}}
\item 
There is a subcategory~$\cR_X \subset \Db(X)$ defined by the semiorthogonal decomposition
\begin{equation}
\label{eq:dbx-crx}
\Db(X) = 
\begin{cases}
\langle \cR_X, \cB_X, \cB_X \otimes \cL_X, \dots, \cB_X \otimes \cL_X^{m-d-1} \rangle & \text{if $d < m$} , \\
\hphantom{\langle}\cR_X & \text{if $d = m$}, 
\end{cases}
\end{equation} 
where in the first case~$\cB_X \coloneqq \Psi(\cB_M)$ and the functor~$\Psi\vert_{\cB_M}$ is fully faithful.
\item 
The Serre functor~$\bS_X$ of~$\Db(X)$ and the spherical twist~$\bT_{\Psi,\Psi^!}$  
are both compatible with~\eqref{eq:dbx-crx} in the sense that
the functors~$\bS_X \circ (- \otimes \cL_X^{m-d})$ and~$\bT_{\Psi,\Psi^!} \circ (- \otimes \cL^d_X)$ preserve~\eqref{eq:dbx-crx} 
and induce autoequivalences of the residual category~$\cR_X$
\begin{equation}
\label{eq:sigmax-rhox-geometric}
\bs_{\cR_X} \coloneqq (\bS_X \circ (-\otimes \cL_X^{m-d}))\vert_{\cR_X}
\qquad\text{and}\qquad 
\bt_{\cR_X} \coloneqq (\bT_{\Psi,\Psi^!} \circ ( - \otimes \cL_X^d))\vert_{\cR_X}. 
\end{equation} 
\item 
The functor~$\Psi$ takes~$\cR_M$ to~$\cR_X$ and its restriction
\begin{equation*}
\Psi_\cR \colon \cR_M \to \cR_X
\end{equation*}
is spherical. 
\item 
If~$\bT_{\Psi^!_\cR,\Psi_\cR} \colon \cR_M \to \cR_M$ and~$\bT_{\Psi_\cR,\Psi_\cR^!} \colon \cR_X \to \cR_X$
are the spherical twists with respect to~$\Psi_\cR$ and~$c = \gcd(d,m)$, 
then the Serre functors~$\bS_{\cR_M}$ and~$\bS_{\cR_X}$ satisfy
\begin{equation}
\label{eq:serre-crx-power}
\bS_{\cR_M}^{d/c} \cong \bT_{\Psi_\cR^!,\Psi_\cR}^{m/c} \circ \bt_{\cR_M}^{-m/c} \circ \bs_{\cR_M}^{d/c}
\qquad\text{and}\qquad
\bS_{\cR_X}^{d/c} \cong \bT_{\Psi_\cR,\Psi_\cR^!}^{(m-d)/c} \circ \bt_{\cR_X}^{(d-m)/c} \circ \bs_{\cR_X}^{d/c},
\end{equation}
and all the factors on the right hand sides of~\eqref{eq:serre-crx-power} commute.
\end{enumerate}
\end{theorem}

\begin{remark}
\label{rem:categorical-reformulation}
In the body of the paper we replace the derived categories~$\Db(M)$ and~$\Db(X)$ 
by suitably enhanced triangulated categories~$\cC$ and~$\cD$ admitting Serre functors.
Similarly, we replace the twists by the line bundles~$\cL_M$ and~$\cL_X$ 
by arbitrary autoequivalences~$\bal_\cC$ and~$\bal_\cD$, respectively, see Theorem~\ref{thm:main-categorical}.
In particular, this categorical version of the theorem applies equally well to the case where~$M$ and~$X$ are Gorenstein varieties
and one considers the categories~\mbox{$\cC = \Dp(M)$} and~\mbox{$\cD = \Dp(X)$} of perfect complexes, see Corollary~\ref{cor:gorenstein}. 
\end{remark}

\begin{remark}
\label{rem:cr-zero}
If the residual category~$\cR_M$ of~\eqref{eq:dbm} vanishes, so that we are in the setup of~\cite{K19},
then the spherical functor~$\bT_{\Psi_\cR,\Psi_\cR^!}$ is isomorphic to the identity, 
and the property~\eqref{eq:serre-crx-power} of the Serre functor~$\bS_{\cR_X}$ of the category~$\cR_X$ reduces to~\cite[Theorem~3.5]{K19};
indeed, it is easy to check that the autoequivalences~$\bs_{\cR_X}$ and~$\bt_{\cR_X}$ defined by~\eqref{eq:sigmax-rhox-geometric}
are related to the autoequivalences~$\sigma$ and~$\rho$ from~\cite{K19} 
by~$\sigma = \bs_{\cR_X} \circ \bt_{\cR_X}$ and~$\rho = \bt_{\cR_X}$. 
\end{remark}

Theorem~\ref{thm:main}, and its variants mentioned in Remark~\ref{rem:categorical-reformulation}, apply to a wide range of examples, some of which are detailed in~\S\ref{sec:examples}.
It turns out that in many examples,
the autoequivalences~$\bs_{\cR_X}$ and~$\bt_{\cR_X}$ are shifts (or compositions of a shift and an autoequivalence of finite order).
In these cases, the isomorphism~\eqref{eq:serre-crx-power} implies that a power of the Serre functor of the category~$\cR_X$ 
is isomorphic to a power of a spherical twist up to shift. 
For example, this happens when~$X$ is a suitable divisor in or double cover of~$M$, 
see Corollary~\ref{corollary-divisor-double-cover}. 
For concreteness, we spell out here the interesting case of Fano complete intersections. 

\begin{corollary}
\label{corollary-complete-intersection}
Let $M \subset \P^n$ be a smooth complete intersection of type~$(d_1,d_2,\dots,d_{k-1})$, 
i.e.~$M$ is a smooth $(n-k+1)$-dimensional intersection of hypersurfaces of degrees $d_1, \dots, d_{k-1}$. 
Let~$X \subset M$ be a smooth intersection of~$M$ with a hypersurface of degree~$d_k$.
Assume
\begin{equation*}
\sum_{i=1}^k d_i \le n , 
\end{equation*}
so that both~$X$ and~$M$ are Fano varieties of indices 
\begin{equation*} 
\ind(X) = n + 1 - \sum_{i=1}^k d_i \quad \text{and} \quad \ind(M) = n + 1 - \sum_{i=1}^{k-1} d_i . 
\end{equation*} 
Let~$\cR_M \subset \Db(M)$ and~$\cR_X \subset \Db(X)$ be the subcategories defined by the semiorthogonal decompositions
\begin{align*}
\Db(M) & = \langle \cR_M, \cO_M, \cO_M(1), \dots, \cO_M(\ind(M) -1) \rangle, \\ 
\Db(X) & = \langle \cR_X, \cO_X, \cO_X(1), \dots, \cO_X(\ind(X) - 1) \rangle.
\end{align*}
If $i \colon X \hookrightarrow M$ is the embedding, 
then the pullback functor~$\Psi \coloneqq i^* \colon \Db(M) \to \Db(X)$ takes~$\cR_M$ to~$\cR_X$ 
and induces a spherical functor~$\Psi_{\cR} \colon \cR_M \to \cR_X$. 
Moreover, if $\bT_{\Psi_{\cR}^!,\Psi_{\cR}} \colon \cR_M \to \cR_M$ 
and~$\bT_{\Psi_{\cR},\Psi_{\cR}^!} \colon \cR_X \to \cR_X$ are the corresponding spherical twists 
and~$c = \gcd(d_k, \ind(M)), $ then
\begin{equation*}
\bS_{\cR_{M}}^{d_k/c} \cong 
\bT_{\Psi_{\cR}^!,\Psi_{\cR}}^{\ind(M)/c} \circ \left[ \frac{d_k \dim(M)}{c} \right] 
\quad \text{and} \quad 
\bS_{\cR_X}^{d_k/c} \cong 
\bT_{\Psi_{\cR},\Psi_{\cR}^!}^{\ind(X)/c} \circ \left [ \frac{d_k \dim(X) - 2 \ind(X)}{c} \right ] .
\end{equation*}
\end{corollary}

The same holds for any pair~$X \subset M$ of complete intersections in weighted projective space
if we replace the bounded derived categories by the categories of perfect complexes
as in Remark~\ref{rem:categorical-reformulation}, see Corollary~\ref{cor:serre-ci-pn}.

\subsection{Serre dimensions and applications}

Corollary~\ref{corollary-complete-intersection} says that up to a specified power of a spherical twist, 
the residual category~$\cR_X$ is fractional Calabi--Yau of dimension  
\begin{equation*}
\dim(X) - 2\frac{\ind(X)}{d_k}. 
\end{equation*}
In their approach to the rationality problem for algebraic varieties, 
Katzarkov and Kontsevich~\cite{Kontsevich,Katzarkov} conjectured that the upper and lower Serre dimensions of~$\cR_X$ 
coincide and are equal to the above rational number if~$d_k$ 
is the largest of the degrees~$d_i$, i.e.~$d_k \geq d_i$ for all $1 \leq i \leq k-1$. 
Our second main result establishes a corrected version of Katzarkov--Kontsevich's conjecture, 
which instead says that typically the upper and lower Serre dimensions differ, 
but they are indeed given by the above formula 
for~$d_k$ being the \emph{largest} or \emph{smallest} of the~$d_i$. 
Note that without loss of generality, we may assume all~$d_i >1$, by possibly replacing~$\P^n$ with a smaller projective space. 
We will also use the following technical notion. 

\begin{definition}
\label{def:smoothly-attainable}
A complete intersection~$X \subset \P^n$ of type~$(d_1,\dots,d_k)$ is {\sf smoothly attainable}
if, assuming~$d_1 \ge \dots \ge d_k$, there is a chain of smooth subvarieties
\begin{equation*}
X = X_k \subset X_{k-1} \subset \dots \subset X_1 \subset X_0 = \P^n
\end{equation*}
where $X_l$ is a complete intersection of type~$(d_1,\dots,d_l)$ for each $0 \le l \le k$.
\end{definition}

Note that a smoothly attainable complete intersection is smooth by definition.
Conversely,
Bertini's Theorem implies that in characteristic~$0$ every smooth complete intersection 
is smoothly attainable (Lemma~\ref{lemma:m-smooth}).
However, in positive characteristic this is not true as the following example shows. 

\begin{example}[Sawin]
\label{ex:sawin}
Let~$X \subset \P^2$ be the complete intersection of the curves~$x^3 + yz^2 = 0$ and~$y^3 + xz^2 = 0$, 
where~$(x,y,z)$ are coordinates, over a base field of characteristic~$3$. 
Then~$X$ is smooth, but any cubic curve containing~$X$ is singular.
\end{example}

Under the assumption of smooth attainability we prove the following result. 

\begin{theorem}
\label{theorem-ci-sdim}
Let~$X \subset \P^n$ be a type~$(d_1, d_2, \dots, d_{k})$ smoothly attainable Fano complete intersection, with all~$d_i > 1$.
Let~$d_{\max}$ and~$d_{\min}$ be the maximum and minimum of the degrees~$d_i$. 
Then 
\begin{equation*}
\usdim(\cR_X) = \dim(X) - 2\frac{\ind(X)}{d_{\max}} 
\qquad \text{and} 
\qquad 
\lsdim(\cR_X) = \dim(X) - 2\frac{\ind(X)}{d_{\min}} . 
\end{equation*} 
\end{theorem}

\begin{remark}
If the base field has characteristic~$0$ and~$X$ is a smooth Fano complete intersection,
there is an interesting inequality between the upper Serre dimension of~$\cR_X$ 
and its {\sf Hochschild level}, defined for any category~$\cC$ as
\begin{equation*}
\hl(\cC) \coloneqq \max \{ i \mid \HOH_i(\cC) \neq 0 \},
\end{equation*}
where~$\HOH_\bullet$ stands for Hochschild homology. 
Combining the Hochschild--Kostant--Rosenberg isomorphism and~\cite[Proposition~1.15]{PS20} one computes
\begin{equation*}
\hl(\cR_X) = \hl(\Db(X)) = \dim(X) - 2 \left\lceil \frac{\ind(X)}{d_{\max}} \right\rceil,
\end{equation*} 
except when~$X$ is an odd-dimensional quadric, in which case~$\hl(\cR_X) = 0$.
Comparing this with the statement of Theorem~\ref{theorem-ci-sdim} we obtain the inequality
\begin{equation*}
\hl(\cR_X) \le \usdim(\cR_X).
\end{equation*}
It is an interesting question whether the analogous inequality~$\hl(\cC) \le \usdim(\cC)$ is true for any smooth and proper category~$\cC$.
\end{remark} 

The idea of the proof of Theorem~\ref{theorem-ci-sdim} is 
to use Corollary~\ref{corollary-complete-intersection} to inductively control the Serre functor of 
the residual category of a Fano complete intersection in terms of complete intersections of smaller codimension. 
The smoothness of intermediate complete intersections (and the condition of smooth attainability) 
is used to relate the dimensions of the spherical twist functors~$\bT_{\Psi_{\cR}^!,\Psi_{\cR}}$ and~$\bT_{\Psi_{\cR},\Psi_{\cR}^!}$,
see Theorem~\ref{theorem-sdim-bounds}\eqref{smoothproper}.

One simple consequence of the formula for Serre dimensions of the category~$\cR_X$ is its non-geometricity in most cases.
Indeed, the equalities in~\eqref{SdimX} show that for~$\cR_X$ to be geometric 
(i.e., to be equivalent to the derived category of a variety) 
it is necessary that the upper Serre dimension equals the lower one and is integral;
when~$X$ is a Fano complete intersection of type~$(d_1,\dots,d_k)$ this holds true for~$\cR_X$ 
if and only if~$d_1 = \dots = d_k$ and this integer divides~$2(n+1)$, see Corollary~\ref{corollary-geometricity}.
This, however, leaves some space for examples of geometric residual categories. 
It would be interesting to figure out if indeed such examples exist 
in addition to the known cases of some complete intersections of types~$(2)$, $(2,2)$, $(2,2,2)$, and~$(3)$. 

Theorem~\ref{theorem-ci-sdim} also has an interesting consequence for stability conditions on~$\cR_X$. 
Recent work has highlighted the importance of stability conditions on 
residual categories of Fano varieties, as their moduli spaces of stable objects exhibit rich structure and are often better behaved than classically studied moduli spaces~\cite{BLMS, stability-families, GM-stability, pertusi-yang, APR, bayer-theta, zhang}. 
The case of Fano threefolds suggests that stability conditions are particularly useful 
if they are \emph{Serre invariant}, see Definition~\ref{def:serre-invariant-stab}.
For many Fano threefolds, Serre invariant stability conditions are known to exist on the residual category~\cite{BLMS,pertusi-yang}, 
and have played a key role in the analysis of moduli spaces of stable objects. 
This raised the hope that such stability conditions might always exist on residual categories of Fano varieties, 
but we show this is far from true. 
Namely, in Proposition~\ref{proposition-Phi-invt} we prove that the upper and lower Serre dimensions of~$\cC$ 
must coincide in order for a Serre invariant stability condition --- 
or even a \emph{pre-}stability condition in the sense of Definition~\ref{definition-prestability} --- to exist, 
and thus deduce the following from Theorem~\ref{theorem-ci-sdim}. 

\begin{corollary}
\label{corollary-serre-invariant-stab}
Let~$X \subset \P^n$ be a type~$(d_1, d_2, \dots, d_{k})$ smoothly attainable Fano complete intersection, 
where all~$d_i > 1$ and not all of the~$d_i$ are equal.
Then there does not exist a Serre invariant pre-stability condition on~$\cR_X$. 
\end{corollary}

\subsection{Refined residual categories}
\label{subsection-intro-refined-categories}

In some cases, the semiorthogonal decomposition defining the residual category can be refined. 
For a simple yet interesting example, let~$X \subset \P^5$ be the intersection of a smooth quadric hypersurface and a cubic hypersurface
(i.e., a prime Fano threefold of genus~$4$)
over an algebraically closed field of characteristic not equal to~$2$.
Then if~$\cS$ denotes one of the spinor bundles on the quadric, there is a semiorthogonal decomposition 
\begin{equation*}
\Db(X) = \langle \cA_X, \cS\vert_X , \cO_X \rangle 
\end{equation*} 
refining the decomposition $\Db(X) = \langle \cR_X, \cO_X \rangle$. 
Our arguments can also be applied to refined residual components like $\cA_X$; in particular, 
in this case we show that 
\begin{equation*}
\usdim(\cA_X) = 3
\quad \text{and} \quad 
\lsdim(\cA_X) = 7/3 
\end{equation*}
(whereas~$\usdim(\cR_X) = 7/3$ and~$\lsdim(\cR_X) = 2$), and that~$\cA_X$ admits no Serre invariant stability conditions (Propositions~\ref{proposition-2-3-serredim} and~\ref{proposition-2-3-serreinvtstab}). 
This gives a partial negative answer to the following question left open by the results of~\cite{BLMS}: 
For~$X$ a prime Fano threefold of genus~$4$, do any stability conditions (Serre invariant or not) exist on~$\cA_X$? 
This question is of special interest, as in~\cite{BLMS} the existence of stability conditions is shown 
on residual categories of all other prime Fano threefolds. 

\subsection{Organization} 
The paper is organized as follows.
In~\S\ref{sec:spherical} we recall some properties of spherical functors and list a number of relevant examples.
In~\S\ref{sec:lefschetz} we review the notion of a Lefschetz decomposition, and define 
Serre and twist compatibility of such decompositions. 
In~\S\ref{sec:proof} we prove our main result, Theorem~\ref{thm:main-categorical}, 
and deduce from it Theorem~\ref{thm:main} and other consequences.
In~\S\ref{sec:examples} we list a number of examples where our theorems apply;
in particular, we treat complete intersections in \emph{weighted} projective spaces in 
Corollary~\ref{cor:serre-ci-pn}, which gives Corollary~\ref{corollary-complete-intersection} as a special case.
In~\S\ref{sec:serre-dims} we prove Theorem~\ref{theorem-ci-sdim} on Serre dimensions of complete intersections, 
as well as Corollary~\ref{corollary-serre-invariant-stab}. 
Finally, in Appendix~\ref{sec:more} we prove a generalization of results of Addington and Halpern-Leistner--Shipman
about spherical functors from categories with semiorthogonal decompositions, 
and in Appendix~\ref{appendix-ind} we gather some results about ind-completions of categories. 

\subsection{Conventions}
\label{section-conventions}
We work over an arbitrary field~$\kk$, but in some places locally in the text 
we impose some restrictions on the characteristic and algebraic closedness. 
We denote by~$\Db(X)$,~$\Dp(X)$, and~$\Dqc(X)$ the bounded derived category of coherent sheaves, 
the category of perfect complexes, and the unbounded derived category of quasi-coherent sheaves, respectively. 
All triangulated categories are assumed to be $\kk$-linear, idempotent complete, and enhanced,
and all functors between such categories are assumed to be $\kk$-linear, exact, and enhanced; for brevity, in the text we often omit these adjectives.
The type of the enhancement used is not important --- 
it can be an enhancement by differential graded categories, like in~\cite{KL15},
or an enhancement by $\infty$-categories, like in~\cite{KP20};
it is also possible to work with categories of geometric origin and Fourier--Mukai functors between them, like in~\cite{K19}. 
For a functor~$\Psi$ we usually denote by~$\Psi^*$ and~$\Psi^!$ its left and right adjoints, when they exist.

\subsection{Acknowledgements}
We would like to thank Arend Bayer and Emanuele Macr\`{i} 
for interesting conversations about the case of a~$(2,3)$ complete intersection in~$\P^5$ 
and for suggesting the application to nonexistence of Serre invariant stability conditions,  
Johan de Jong for a useful discussion about generators of triangulated categories, 
Will Sawin for providing Example~\ref{ex:sawin},
and Pieter Belmans for an interesting discussion about Corollary~\ref{corollary-geometricity}.
Finally, we are grateful to Ludmil Katzarkov and Maxim Kontsevich, whose 
ideas served as one of the motivations for this research. 

\section{Spherical functors}
\label{sec:spherical}

In this section we collect some preliminary results on spherical functors.
In~\S\ref{subsec:spherical} we recall the definition and some properties of spherical functors
and in~\S\ref{subsec:spherical-examples} we list some examples.

\subsection{Definition and basic properties}
\label{subsec:spherical}

Let~$\cC$ be a triangulated category such that for any objects~$C_1, C_2 \in \cC$ 
the space~$\Hom(C_1, C_2)$ is finite-dimensional.
A {\sf Serre functor}~\cite{BK} of such a category $\cC$ is an autoequivalence~\mbox{$\bS_\cC \colon \cC \to \cC$} 
such that there is a bifunctorial isomorphism
\begin{equation}
\label{eq:serre-def}
\Hom(C_1,C_2)^\vee \cong \Hom(C_2,\bS_\cC(C_1)).
\end{equation} 
One of the properties of the Serre functor that will be used below is commutativity with autoequivalences~\cite[Proposition~1.3]{BO01}:
\begin{equation}
\label{eq:serre-commutativity}
\bS_\cC \circ \bal \cong \bal \circ \bS_\cC,
\qquad 
\text{for any} \quad \bal \in \Aut(\cC).
\end{equation} 
Another useful property is the relation between left and right adjoint functors:
if~$\Psi \colon \cC \to \cD$ is a functor between triangulated categories which have Serre functors~$\bS_\cC$ and~$\bS_\cD$ 
and~$\Psi^* \colon \cD \to \cC$ is its left adjoint functor, then it is easy to see that the functor
\begin{equation}
\label{eq:right-left}
\Psi^! = \bS_\cC \circ \Psi^* \circ \bS_\cD^{-1}
\end{equation}
is the right adjoint of~$\Psi$,
\begin{equation}
\label{eq:double-right}
\Psi^{!!} = \bS_\cD \circ \Psi \circ \bS_\cC^{-1}
\end{equation}
is the right adjoint of~$\Psi^!$, and so on.
In particular, an adjoint pair of functors between categories with Serre functors 
extends to an infinite (in both directions) sequence of pairwise adjoint functors.

Let~$\Psi \colon \cC \to \cD$ be 
a triangulated functor between 
triangulated categories.
There are several ways to define what it means for~$\Psi$ to be spherical, see~\cite{Anno,AL,KSS}.
Perhaps the easiest definition is given in~\cite[Definition~2.8]{K19}.
To state it we assume~$\Psi$ has both left and right adjoint functors, $\Psi^*,\Psi^! \colon \cD \to \cC$, respectively; 
by our remarks above, if $\cC$ and $\cD$ admit Serre functors then it is enough to assume the existence of one adjoint, 
as the other then exists automatically.

\begin{definition}
\label{definition-spherical} 
A functor~$\Psi \colon \cC \to \cD$ with adjoints~$\Psi^*,\Psi^! \colon \cD \to \cC$ is {\sf spherical} if the morphisms
\begin{equation*}
\Psi^* \oplus \Psi^! \xrightarrow{\ \eta_{\Psi^!,\Psi} \oplus \eta_{\Psi,\Psi^*}\ } \Psi^! \circ \Psi \circ \Psi^*
\qquad\text{and}\qquad 
\Psi^* \circ \Psi \circ \Psi^! \xrightarrow{\ \epsilon_{\Psi,\Psi^!} \oplus \epsilon_{\Psi^*,\Psi}\ } \Psi^* \oplus \Psi^!
\end{equation*}
induced by the units~$\eta_{\Psi^!,\Psi},\eta_{\Psi,\Psi^*}$ and counits~$\epsilon_{\Psi,\Psi^!},\epsilon_{\Psi^*,\Psi}$ of the adjunctions, 
are isomorphisms.
\end{definition}

The importance of spherical functors is due to the fact that (assuming the functors and categories are enhanced)
they induce autoequivalences of the source and target categories.
More precisely, for any adjoint sequence~$(\Psi^*,\Psi,\Psi^!)$ 
one can define endofunctors~$\bT_{\Psi^*,\Psi},\bT_{\Psi^!,\Psi}$ of~$\cC$ 
and endofunctors~$\bT_{\Psi,\Psi^!},\bT_{\Psi,\Psi^*}$ of~$\cD$ by means of exact triangles
\begin{equation}
\label{eq:twists-triangles}
\begin{aligned}
\Psi^* \circ \Psi \xrightarrow{\ \epsilon_{\Psi^*,\Psi}\ } &\id_\cC \xrightarrow\qquad \bT_{\Psi^*,\Psi}, \qquad &
\bT_{\Psi^!,\Psi} \xrightarrow\qquad &\id_\cC \xrightarrow{\ \eta_{\Psi^!,\Psi}\ } \Psi^! \circ \Psi,
\\
\Psi \circ \Psi^! \xrightarrow{\ \epsilon_{\Psi,\Psi^!}\ } &\id_\cD \xrightarrow\qquad \bT_{\Psi,\Psi^!}, \qquad &
\bT_{\Psi,\Psi^*} \xrightarrow\qquad &\id_\cD \xrightarrow{\ \eta_{\Psi,\Psi^*}\ } \Psi \circ \Psi^*.
\end{aligned}
\end{equation} 
These endofunctors are called the {\sf spherical twists} associated to the functor~$\Psi$.
Moreover, for any~$\Psi$ they are pairwise adjoint
\begin{equation}
\label{eq:twists-adjunctions}
\bT_{\Psi^*,\Psi} \cong \bT_{\Psi^!,\Psi}^*,
\qquad
\bT_{\Psi,\Psi^!} \cong \bT_{\Psi,\Psi^*}^!. 
\end{equation}
Finally, if the functor~$\Psi$ is spherical, the spherical twists are autoequivalences and
\begin{equation}
\label{eq:twists-inversion}
\bT_{\Psi^*,\Psi} \cong \bT_{\Psi^!,\Psi}^{-1},
\qquad
\bT_{\Psi,\Psi^!} \cong \bT_{\Psi,\Psi^*}^{-1},
\end{equation}
see~\cite[Proposition~2.9]{K19}.
The converse is also true: 

\begin{proposition}[{\cite[Theorem~5.1]{AL}}]
\label{proposition:spherical-criterion}
Let $\Psi \colon \cC \to \cD$ be a functor with both adjoints. 
Then $\Psi$ is spherical if and only if any two of the following four conditions hold:
\begin{enumerate}
\item 
\label{item:twist-cc}
The twist functor $\bT_{\Psi^!,\Psi} \colon \cC \to \cC$ is an autoequivalence. 
\item 
\label{item:twist-cd}
The twist functor $\bT_{\Psi,\Psi^!} \colon \cD \to \cD$ is an autoequivalence. 
\item 
\label{item:iso-1}
The composition 
\begin{equation*}
\Psi^* \circ \bT_{\Psi, \Psi^!}[-1] \xrightarrow\qquad   \Psi^* \circ \Psi \circ \Psi^! \xrightarrow{\ \epsilon_{\Psi^*,\Psi} \ } \Psi^! 
\end{equation*} 
is an isomorphism of functors, where the first arrow is induced by the canonical morphism $\bT_{\Psi, \Psi^!}[-1] \to \Psi \circ \Psi^!$
from the rotation of the defining triangle~\eqref{eq:twists-triangles}. 
\item 
\label{item:iso-2}
The composition 
\begin{equation*}
\Psi^! \xrightarrow{\ \eta_{\Psi,\Psi^*} \ } \Psi^! \circ \Psi \circ \Psi^* \xrightarrow \qquad \bT_{\Psi^!, \Psi} \circ \Psi^* [1] 
\end{equation*} 
is an isomorphism of functors, where the second arrow is induced by the canonical morphism~$\Psi^! \circ \Psi \to \bT_{\Psi^!, \Psi}[1]$
from the rotation of the defining triangle~\eqref{eq:twists-triangles}. 
\end{enumerate}
\end{proposition}

Note that a combination of this with the adjunctions~\eqref{eq:twists-adjunctions}
implies the following criterion.

\begin{corollary}
\label{cor:spherical-adjoint}
Let $\Psi \colon \cC \to \cD$ be a functor that admits left and right adjoints $\Psi^*, \Psi^! \colon \cD \to \cC$. 
Assume that $\Psi^*$ and $\Psi^!$ also admit both adjoints 
\textup(e.g., this is true if $\cC$ and $\cD$ admit Serre functors\textup). 
Then $\Psi$ is spherical if and only if $\Psi^*$ is spherical, if and only if $\Psi^!$ is spherical.
\end{corollary}

We will use the following closely related property.

\begin{proposition}
\label{prop:spherical-intertwining}
For any functor~$\Psi \colon \cC \to \cD$ there is a natural isomorphism
\begin{equation}
\label{eq:spherical-intertwining}
\Psi \circ \bT_{\Psi^!,\Psi}[1] \cong \bT_{\Psi,\Psi^!}[-1] \circ \Psi.
\end{equation}
Moreover, if~$\Psi \colon \cC \to \cD$ is a spherical functor which has a double right adjoint functor~$\Psi^{!!}$
there are natural isomorphisms
\begin{equation}
\label{eq:psi-shriek-shriek}
\Psi^{!!} \cong \bT_{\Psi,\Psi^!}^{-1} \circ \Psi[1] \cong \Psi \circ \bT_{\Psi^!,\Psi}^{-1}[-1].
\end{equation} 
\end{proposition}

\begin{proof}
Consider the commutative diagram
\begin{equation*}
\xymatrix@R=2ex@C=5em{
\Psi \ar[rr]^-{{\id}} \ar[dr]_-{\Psi\eta_{\Psi^!,\Psi}} &&
\Psi 
\\
& \Psi \circ \Psi^! \circ \Psi \ar[dr] \ar[ur]_-{\eps_{\Psi,\Psi^!}\Psi} \ar@(ur,ul)[]^{{\iota}}
\\
\bT_{\Psi,\Psi^!} \circ \Psi [-1] \ar@{-->}[rr] \ar[ur] &&
\Psi \circ \bT_{\Psi^!,\Psi}[1],
}
\end{equation*}
where the diagonals are the defining triangles~\eqref{eq:twists-triangles} of~$\bT_{\Psi,\Psi^!}$ and~$\bT_{\Psi^!,\Psi}$
composed with~$\Psi$ on the right and left, respectively.
The upper horizontal arrow~$(\eps_{\Psi,\Psi^!}\Psi) \circ (\Psi\eta_{\Psi^!,\Psi})$ 
is the identity by definition of adjunction, hence the composition
\begin{equation*}
\iota = (\Psi\eta_{\Psi^!,\Psi}) \circ (\eps_{\Psi,\Psi^!}\Psi) \colon \Psi \circ \Psi^! \circ \Psi \to \Psi \circ \Psi^! \circ \Psi
\end{equation*}
is an idempotent endomorphism of~$\Psi \circ \Psi^! \circ \Psi$ with~$\Psi$ being the corresponding direct summand.
Now it follows from the diagram 
that the second summand, corresponding to the idempotent~$1 - \iota$, 
is isomorphic both to~$\bT_{\Psi,\Psi^!} \circ \Psi[-1]$
and to~$\Psi \circ \bT_{\Psi^!,\Psi}[1]$, 
and the dashed arrow in the diagram gives an isomorphism between these two functors.
This proves~\eqref{eq:spherical-intertwining}.

Now assume~$\Psi$ is spherical.
Consider the functor~$\Phi = \Psi^! \colon \cD \to \cC$, which is also spherical by Corollary~\ref{cor:spherical-adjoint}.
By Proposition~\ref{proposition:spherical-criterion}\eqref{item:iso-2} there is an isomorphism~$\Phi^! \cong \bT_{\Phi^!,\Phi} \circ \Phi^*[1]$; 
alternatively, such an isomorphism follows by a similar argument to the above applied to the direct 
sum decomposition of the functor~$\Phi^! \circ \Phi \circ \Phi^*$, see~\cite[proof of Proposition~2.9]{K19}.
Using~\eqref{eq:twists-inversion} and rewriting this isomorphism in terms of~$\Psi$ we obtain the first isomorphism in~\eqref{eq:psi-shriek-shriek}, 
and the second follows from~\eqref{eq:spherical-intertwining}.
\end{proof}

Later we will use yet another property of functors and their twists, 
in the presence of the the following intertwining condition.

\begin{definition}
\label{def:intertwining}
Let~$\bal_\cC \in \Aut(\cC)$, $\bal_\cD \in \Aut(\cD)$ be a pair of autoequivalences.
We say that a functor~$\Psi \colon \cC \to \cD$ {\sf intertwines} between~$\bal_\cC$ and~$\bal_\cD$
if a collection of isomorphisms
\begin{equation*}
\Psi \circ \bal_\cC^i \xrightarrow{\ a_i\ } \bal_\cD^i \circ \Psi,
\qquad i \in \ZZ,
\end{equation*}
is given such that for all $i,j \in \ZZ$ the composition
\begin{equation*}
\Psi \circ \bal_\cC^{i +j} \xrightarrow\quad
\Psi \circ \bal_\cC^i \circ \bal_\cC^j \xrightarrow{\ a_i\ } 
\bal_\cD^i \circ \Psi \circ \bal_\cC^j \xrightarrow{\ a_j\ } 
\bal_\cD^i \circ \bal_\cD^j \circ \Psi \xrightarrow\quad
 \bal_\cD^{i + j} \circ \Psi
\end{equation*}
where the first and last arrows are the natural isomorphisms, is equal to~$a_{i+j}$.
\end{definition}

\begin{lemma}
\label{lemma:twists-commute}
If a functor $\Psi \colon \cC \to \cD$ intertwines between autoequivalences~$\bal_\cC \in \Aut(\cC)$ and~\mbox{$\bal_\cD \in \Aut(\cD)$}, 
then the adjoint functors $\Psi^*,\Psi^! \colon \cD \to \cC$ intertwine between~$\bal_\cD$ and~$\bal_\cC$.
Moreover, in this case the twist functors~$\bT_{\Psi^*,\Psi}$, $\bT_{\Psi^!,\Psi}$ commute with~$\bal_\cC$ 
and the twist functors~$\bT_{\Psi,\Psi^!}$, $\bT_{\Psi,\Psi^*}$ commute with~$\bal_\cD$.
\end{lemma}

\begin{proof}
We start with a general remark. 
If~$F,G \colon \cC \to \cD$ are functors with left adjoints and~$a \colon F \to G$ is a morphism, 
then by the Yoneda lemma there is a uniquely determined morphism~$a^* \colon G^* \to F^*$ such that the diagram 
\begin{equation*}
\xymatrix{
\Hom(D, F(C)) \ar[r]^{a} \ar[d]_{\sim}& \Hom(D, G(C)) \ar[d]^{\sim} \\ 
\Hom(F^*(D), C) \ar[r]^{a^*} & \Hom(G^*(D), C) 
}
\end{equation*} 
commutes functorially in~$D \in \cD$, $C \in \cC$; 
we call~$a^*$ the left adjoint of~$a$. 
Similarly, if~$F,G$ have right adjoints, then we can consider the right adjoint~$a^! \colon G^! \to F^!$ of~$a$.

Now if~$a_i \colon \Psi \circ \bal_\cC^i \to \bal_\cD^i \circ \Psi$ are intertwining isomorphisms for~$\Psi$, then their adjoints
\begin{equation*}
\Psi^* \circ \bal_\cD^{-i} \xrightarrow{\ a_i^*\ } \bal_\cC^{-i} \circ \Psi^*,
\qquad 
\Psi^! \circ \bal_\cD^{-i} \xrightarrow{\ a_i^!\ } \bal_\cC^{-i} \circ \Psi^!,
\end{equation*}
are also isomorphisms, and it is easy to see that~$a_{-i}^*$ and~$a_{-i}^!$ 
provide the intertwining isomorphisms for~$\Psi^*$ and~$\Psi^!$, respectively. 
Furthermore, we have a commutative diagram
\begin{equation*}
\xymatrix@C=8em{
\bal_\cD \circ \bT_{\Psi,\Psi^*} \circ \bal_\cD^{-1} \ar[r] &
\bal_\cD \circ \bal_\cD^{-1} \ar[r]^{{\bal_\cD\eta_{\Psi,\Psi^*}\bal_\cD^{-1}}} \ar[d]_\cong &
\bal_\cD \circ \Psi \circ \Psi^* \circ \bal_\cD^{-1} \ar[d]^{a_1^{-1} \circ a_1^*}_\cong 
\\
\bT_{\Psi,\Psi^*} \ar[r] &
\id_{\cD} \ar[r]^{\eta_{\Psi,\Psi^*}} &
\Psi \circ \Psi^*
}
\end{equation*}
which induces an isomorphism~$\bal_\cD \circ \bT_{\Psi,\Psi^*} \circ \bal_\cD^{-1} \xrightarrow\quad \bT_{\Psi,\Psi^*}$,
and thus proves that the autoequivalences~$\bal_\cD$ and~$\bT_{\Psi,\Psi^*}$ commute.
Commutativity of the other pairs of functors is proved analogously.
\end{proof}

\subsection{Examples}
\label{subsec:spherical-examples}

Here we list some examples of spherical functors.

First consider the category~$\cC = \Db(\kk)^{\oplus r}$ generated by a completely orthogonal exceptional collection.
For $i \in \{1,\dots,r\}$ we denote by~$\kk_i$ the generator of the $i$-th factor of~$\cC$.
Obviously any autoequivalence of~$\cC$ can be written as
\begin{equation}
\label{eq:auteq-a1-r}
\kk_i \mapsto \kk_{\sigma(i)}[-n_i-1],
\end{equation} 
where $(n_1,\dots,n_r)$ is a tuple of integers and~$\sigma \in \fS_r$ is a permutation. 
(We have chosen this labelling for the shifts because it is convenient below.)

Recall that a triangulated functor is {\sf conservative} if its kernel is zero.
Since any object in~$\cC$ is a direct sum of shifts of~$\kk_i$, a functor~$\Psi$ from~$\cC$ 
is conservative if and only if~$\Psi(\kk_i) \ne 0$ for all~$i$.
We denote by~$\delta_{i,j}$ the Kronecker delta.
The next result can be found in an unpublished work of Sasha Efimov.

\begin{lemma}[Efimov]
\label{lemma:spherical-collection}
Let~$\cD$ be a category which admits a Serre functor~$\bS_{\cD}$.
A functor~$\Psi \colon \Db(\kk)^{\oplus r} \to \cD$ is conservative and spherical 
if and only if there is a permutation~$\sigma \in \fS_r$ and a tuple~$(n_1,\dots,n_r)$ of integers 
such that the objects~${\rP}_i = \Psi(\kk_i)$ 
satisfy
\begin{enumerate}
\item 
\label{item:spherical-sequences-ext}
$\Ext^\bullet({\rP}_i,{\rP}_j) \cong \kk^{\delta_{i,j}} \oplus \kk^{\delta_{i,\sigma(j)}}[-n_j]$;
\item 
\label{item:spherical-sequences-serre}
$\bS_\cD({\rP}_i) \cong {\rP}_{\sigma^{-1}(i)}[n_{\sigma^{-1}(i)}]$.
\end{enumerate}
Moreover, in this case the twist~$\bT_{\Psi^!,\Psi}$ is given by~\eqref{eq:auteq-a1-r}, 
while the twist~$\bT_{\Psi,\Psi^!}$ and its inverse are given by
\begin{align*}
\bT_{\Psi,\Psi^!}(\cF) &\cong 
\Cone\left( \bigoplus_{i=1}^r \Ext^\bullet({\rP}_i,\cF) \otimes {\rP}_i \xrightarrow{\ \ev\ } \cF\right),
\\
\bT_{\Psi,\Psi^!}^{-1}(\cF) &\cong 
\Cone\left(  \cF \xrightarrow{\ \mathrm{coev} \ } \bigoplus_{i=1}^r \Ext^\bullet(\cF, {\rP}_i)^{\vee} \otimes {\rP}_i \right)[-1] 
\end{align*}
where~$\ev$ and~$\mathrm{coev}$ stand for the evaluation and coevaluation morphisms.
\end{lemma}

\begin{proof}
First suppose that $\Psi$ satisfies conditions~\eqref{item:spherical-sequences-ext} and~\eqref{item:spherical-sequences-serre}. 
By~\eqref{item:spherical-sequences-ext}, $\Psi$ is necessarily conservative. 
It is straightforward to check that $\Psi$ has left and right adjoints given by 
\begin{equation*}
\Psi^{!}(\cF) = \bigoplus_{i=1}^r \Ext^\bullet({\rP}_i, \cF) \otimes \kk_i  
\qquad \text{and} \qquad 
\Psi^*(\cF) = \bigoplus_{i=1}^r \Ext^\bullet(\cF, {\rP}_i)^{\vee} \otimes \kk_i. 
\end{equation*} 
In particular, \eqref{item:spherical-sequences-ext} implies~$\Psi^!({\rP}_i) = \kk_i \oplus \kk_{\sigma(i)}[-n_i]$, 
hence 
\begin{equation*}
\bT_{\Psi^!,\Psi}(\kk_i) \cong \kk_{\sigma(i)}[-n_i-1]; 
\end{equation*}
therefore this twist is an autoequivalence, so property~\eqref{item:twist-cc} from Proposition~\ref{proposition:spherical-criterion} 
is satisfied.
Composing~$\bT_{\Psi^!,\Psi}$ with~$\Psi^*[1]$, and using Serre duality and~\eqref{item:spherical-sequences-serre} we obtain
\begin{equation*}
\bT_{\Psi^!,\Psi}(\Psi^*(\cF)[1]) 
\cong \bigoplus_{i=1}^r \Ext^\bullet(\cF, {\rP}_i)^{\vee} \otimes \kk_{\sigma(i)}[-n_i]
\cong \bigoplus_{i=1}^r \Ext^\bullet({\rP}_{\sigma(i)}[-n_i],\cF) \otimes \kk_{\sigma(i)}[-n_i].
\end{equation*}
Changing the summation order appropriately we see that the right side is equal to~$\Psi^!(\cF)$.
Thus, property~\eqref{item:iso-2} from Proposition~\ref{proposition:spherical-criterion} is satisfied.
Applying Proposition~\ref{proposition:spherical-criterion} we conclude that the functor~$\Psi$ is spherical.
The formulas for~$\bT_{\Psi,\Psi^!}$ and $\bT_{\Psi, \Psi^!}^{-1} \cong \bT_{\Psi^*, \Psi}$ (see~\eqref{eq:twists-inversion}) 
are just rephrasings of the defining triangles~\eqref{eq:twists-triangles}. 

Conversely, suppose $\Psi \colon \Db(\kk)^{\oplus r} \to \cD$ is a conservative spherical functor. 
Since~$\bT_{\Psi^!,\Psi}$ is an autoequivalence of~$\bD(\kk)^{\oplus r}$ it is given by~\eqref{eq:auteq-a1-r} for some~$\sigma$ and~$n_i$.
In particular, if~${\rP}_i = \Psi(\kk_i)$ the triangle~\eqref{eq:twists-triangles} applied to~$\kk_j$ gives
\begin{equation*}
\kk_{\sigma(j)}[-n_j-1] \to \kk_j \to \Psi^!({\rP}_j),
\end{equation*}
The second morphism is the unit of adjunction, so if it is zero then~$\Psi(\kk_j) = 0$ which contradicts the conservativity assumption.
Therefore, the second morphism is a monomorphism, hence the first morphism must be zero, hence
\begin{equation*}
\Psi^!({\rP}_j) \cong \kk_j \oplus \kk_{\sigma(j)}[-n_j].
\end{equation*}
Applying~$\Ext^\bullet(\kk_i,-)$ to this isomorphism and using adjunction we deduce the condition~\eqref{item:spherical-sequences-ext}. 

For condition~\eqref{item:spherical-sequences-serre}, 
just note that
\begin{multline*}
\bS_\cD({\rP}_i) 
\cong \bS_\cD(\Psi(\bS_{\cC}^{-1}(\kk_i))) 
\cong \Psi^{!!}(\kk_i) 
\\
\cong \Psi(\bT_{\Psi^!,\Psi}^{-1}(\kk_i[-1])) \cong 
\Psi(\kk_{\sigma^{-1}(i)}[n_{\sigma^{-1}(i)}]) 
= {\rP}_{\sigma^{-1}(i)}[n_{\sigma^{-1}(i)}],
\end{multline*}
where the first isomorphism uses~$\bS_{\cC} \cong \id$,  
the second is~\eqref{eq:double-right}, 
the third is~\eqref{eq:psi-shriek-shriek},
the fourth is the assumption on~$\bT_{\Psi^!,\Psi}$, and the last is the definition of~${\rP}_{\sigma^{-1}(i)}$.
\end{proof}

Note that in the case~$r = 1$ the functor~$\Psi$ is given by a single object~${\rP} \in \cD$ such that
\begin{equation*}
\Ext^\bullet({\rP},{\rP}) \cong \kk \oplus \kk[-n]
\qquad\text{and}\qquad
\bS_\cD({\rP}) \cong {\rP}[n].
\end{equation*}
Such objects are called {\sf $n$-spherical}.
In general, we call a sequence of objects $({\rP}_1,\dots,{\rP}_r)$ a \mbox{{\sf $\sigma$-spherical collection}} 
if it satisfies conditions~\eqref{item:spherical-sequences-ext} and~\eqref{item:spherical-sequences-serre} of Lemma~\ref{lemma:spherical-collection} 
for some~$\sigma \in \fS_r$ and integers~$(n_1,\dots,n_r)$. 
Further, we write 
\begin{equation*}
\bT_{{\rP}_1, \dots, {\rP}_r} \coloneqq \bT_{\Psi, \Psi^!} \colon \cD \to \cD 
\end{equation*} 
for the associated spherical twist. 

We also note some examples of spherical functors of geometric nature.
The following results have been proved in~\cite{K19} for schemes (see also~\cite[\S2.2]{Add16}), 
but the same proof works for stacks.
We state these results for categories of perfect complexes but they work as well for bounded categories of coherent sheaves.

\begin{lemma}[{\cite[Example~3.1 and Proposition~3.4]{K19}}]
\label{lemma:divisorial-embedding} 
Let $M$ be a Deligne--Mumford stack, and 
let~$i \colon X \to M$ be the embedding of a Cartier divisor.
Then~$i_* \colon \Dp(X) \to \Dp(M)$ and~$i^* \colon \Dp(M) \to \Dp(X)$ are spherical functors with spherical twists
\begin{equation*}
\bT_{i_*,i^*} \cong - \otimes \cO_M(-X),
\qquad 
\bT_{i^*,i_*} \cong - \otimes i^*\cO_M(-X)[2].
\end{equation*}
\end{lemma}

\begin{lemma}[{\cite[Example~3.2 and Proposition~3.4]{K19}}]
\label{lemma:double-covering}
Let $M$ be a Deligne--Mumford stack, 
and let~$f \colon X \to M$ be a flat double covering branched over a Cartier divisor~$B \subset M$.
Then~$f_* \colon \Dp(X) \to \Dp(M)$ and~$f^* \colon \Dp(M) \to \Dp(X)$ are spherical functors with spherical twists
\begin{equation*}
\bT_{f_*,f^*} \cong - \otimes \cO_M(-B/2)[1],
\qquad 
\bT_{f^*,f_*} \cong \tau \circ (- \otimes f^*\cO_M(-B/2)[1]),
\end{equation*}
where $\cO_M(-B/2)$ is the line bundle on~$M$ such that~$f_*\cO_X \cong \cO_M \oplus \cO_M(-B/2)$ 
and~$\tau$ is the autoequivalence of~$\Dp(X)$ induced by the covering involution. 
\end{lemma}

\section{Lefschetz collections and residual categories}
\label{sec:lefschetz}

In this section we discuss a class of
semiorthogonal collections --- called rectangular Lefschetz collections --- 
given by powers of an autoequivalence applied to an admissible subcategory. 
After briefly reviewing in~\S\ref{subsection-admissible-subcats} some results about admissible subcategories,  
in~\S\ref{subsection-spherical-compatible-lc} we recall the definition of Lefschetz collections and their associated residual categories, 
and introduce a notion of compatibility with an auxiliary autoequivalence.
Finally, in \S\ref{subsection-serre-compatible-lc} we prove some nice properties of 
Lefschetz collections that are compatible with the Serre functor.

\subsection{Admissible subcategories}
\label{subsection-admissible-subcats}

Recall that a (strictly) full triangulated subcategory~$\cB \subset \cC$ is {\sf admissible} if its embedding functor has both left and right adjoints.
An admissible subcategory gives rise to two semiorthogonal decompositions
\begin{equation}
\label{eq:sods}
\cC = \langle \cB, {}^\perp\cB \rangle
\qquad\text{and}\qquad
\cC = \langle \cB^\perp, \cB \rangle
\end{equation}
where~${}^\perp\cB$ and~$\cB^\perp$ are the left and right orthogonals to~$\cB$ in~$\cC$.

The projection functors to~${}^\perp\cB$ and~$\cB^\perp$ with respect to the decompositions~\eqref{eq:sods} are known 
as the {\sf right and left mutation functors} through~$\cB$ and denoted by~$\bR_\cB$ and~$\bL_\cB$, respectively.
By definition for any~$C \in \cC$ there are exact triangles
(the decomposition triangles for~\eqref{eq:sods})
\begin{equation}
\label{eq:mutation-triangles}
\bR_\cB(C) \to C \to B,
\qquad 
B' \to C \to \bL_\cB(C),
\end{equation}
where $B,B' \in \cB$, $\bR_\cB(C) \in {}^\perp\cB$, and~$\bL_\cB(C) \in \cB^\perp$. 
These are called the {\sf mutation triangles} of~$C$, and they determine $\bR_\cB(C)$ and $\bL_\cB(C)$ 
as the unique objects in ${}^\perp\cB$ and $\cB^\perp$ fitting into such triangles.

Note that the mutation functors vanish on~$\cB$ and act as the identity on the appropriate orthogonals of~$\cB$, 
while their restrictions to the other orthogonals of~$\cB$ are mutually inverse:
\begin{equation*}
\bR_\cB\vert_{{}^\perp\cB} \cong \id,
\qquad 
\bL_\cB\vert_{\cB^\perp} \cong \id,
\qquad 
\bR_\cB\vert_{\cB^\perp} \cong (\bL_\cB\vert_{{}^\perp\cB})^{-1}.
\end{equation*}

If the category~$\cC$ has a Serre functor, it also defines equivalences of the orthogonals of~$\cB$.

\begin{lemma}[{\cite[Proposition~3.6]{BK}}]
\label{lemma:serre-orthogonals}
If~$\cC$ has a Serre functor~$\bS_\cC$, the functors~$\bS_\cC$ and~$\bS_\cC^{-1}$ provide mutually inverse equivalences
$\bS_\cC \colon {}^\perp\cB \xrightarrow{\ \simeq\ } \cB^\perp$ and $\bS_\cC^{-1} \colon \cB^\perp \xrightarrow{\ \simeq\ } {}^\perp\cB$.
\end{lemma}

Combining the Serre functor of~$\cC$ with mutation functors one can express the Serre functors 
of the orthogonal complements~${}^\perp\cB, \cB^\perp \subset \cC$ of~$\cB$.
\begin{lemma}[{\cite[Proposition~3.7]{BK}}]
\label{lemma:serre-subcategory}
If~$\cC$ has a Serre functor $\bS_{\cC}$ and $\cB \subset \cC$ is an admissible subcategory, 
then the orthogonal categories~${}^\perp\cB, \cB^\perp \subset \cC$ also admit Serre functors, given by 
\begin{equation*}
\bS_{{}^\perp\cB} \cong \bR_\cB \circ \bS_\cC \vert_{{}^\perp\cB}
\qquad\text{and}\qquad
\bS_{\cB^\perp}^{-1} \cong \bL_\cB \circ \bS_\cC^{-1} \vert_{\cB^\perp}.
\end{equation*}
\end{lemma}

\subsection{Lefschetz collections and autoequivalence compatibility}
\label{subsection-spherical-compatible-lc}

Given an autoequivalence~$\bal_\cC \in \Aut(\cC)$ we say that
an admissible subcategory~$\cB \subset \cC$ induces a {\sf rectangular Lefschetz collection of length~$m$ with respect to~$\bal_\cC$}  
if the collection of subcategories
\begin{equation}
\label{eq:rlc}
\cB, \bal_\cC(\cB), \dots, \bal_\cC^{m-1}(\cB)
\end{equation}
is semiorthogonal in~$\cC$.
As each of the subcategories in~\eqref{eq:rlc} is admissible in~$\cC$, the collection extends to a semiorthogonal decomposition
\begin{equation}
\label{eq:residual-def}
\cC = \langle \cR, \cB, \bal_\cC(\cB), \dots, \bal_\cC^{m-1}(\cB) \rangle,
\end{equation} 
and the additional component~$\cR$ of this decomposition is called the {\sf residual category}~\cite[Definition~2.7]{KS20}.

Now assume we are given another autoequivalence~\mbox{$\bT \colon \cC \to \cC$}.

\begin{definition}
\label{def:twist-compatibility}
We say an admissible subcategory~$\cB \subset \cC$ is {\sf $\bT$-compatible of degree~$d$} 
if there is an equality
\begin{equation}
\label{eq:twist-compatible}
\bT(\cB) = \bal_\cC^{-d}(\cB),
\end{equation}
of subcategories of~$\cC$. 
\end{definition}

Usually, we will apply this notion to one of the spherical twists~$\bT_{\Psi^!,\Psi}$ or~$\bT_{\Psi^*,\Psi}$ 
associated to a spherical functor~$\Psi \colon \cC \to \cD$.
In this case we will say that~$\cB$ is {\sf twist compatible}.

\begin{lemma}
\label{lemma:twist-compatibility}
Assume~$\cB \subset \cC$ induces a rectangular Lefschetz collection~\eqref{eq:rlc} of length~$m$.
If the spherical twist~$\bT_{\Psi^!,\Psi}$ commutes with~$\bal_\cC$ and~$\cB \subset \cC$ is $\bT_{\Psi^!,\Psi}$-twist compatible of degree~$d$, then
\begin{equation}
\label{eq:twist-compatible-2}
\bT_{\Psi^!,\Psi}(\bal_\cC^i(\cB)) = \bal_\cC^{i-d}(\cB)
\end{equation}
and
\begin{equation}
\label{eq:twist-compatible-residual}
\bT_{\Psi^!,\Psi}(\bal_\cC^i(\cR)) = \bal_\cC^{i-d}(\cR) 
\end{equation}
for all $i \in \ZZ$. 
In particular, there is an autoequivalence~$\bt_\cR \in \Aut(\cR)$ defined by
\begin{equation}
\label{eq:tau-cr}
\bt_\cR \coloneqq (\bT_{\Psi^!,\Psi} \circ \bal_\cC^d)\vert_\cR.
\end{equation} 
\end{lemma}

\begin{proof}
The equality~\eqref{eq:twist-compatible-2} follows from~\eqref{eq:twist-compatible} with~$\bT = \bT_{\Psi^!,\Psi}$ 
and commutativity of~$\bT_{\Psi^!,\Psi}$ and~$\bal_\cC$.
Applying the autoequivalence~$\bal_\cC^{d - i} \circ  \bT_{\Psi^!,\Psi} \circ \bal_\cC^{i}$ to~\eqref{eq:residual-def} 
and using~\eqref{eq:twist-compatible-2}, we deduce~\eqref{eq:twist-compatible-residual}.
The last part is evident.
\end{proof}

\subsection{Serre compatible Lefschetz collections}
\label{subsection-serre-compatible-lc}

In this subsection we discuss the special case of compatibility with respect to the autoequivalence given by the Serre functor.

\begin{definition}
\label{def:serre-compatibility}
Assume~$\cC$ has a Serre functor~$\bS_{\cC}$ and is equipped with an autoequivalence~$\bal_{\cC}$. 
We say an admissible subcategory~$\cB \subset \cC$ is {\sf Serre compatible of length~$m$} if 
it induces a rectangular Lefschetz collection of length~$m$ with respect to~$\bal_{\cC}$ 
and there is an equality 
\begin{equation}
\label{eq:serre-compatible}
\bS_\cC(\cB) = \bal_\cC^{-m}(\cB) 
\end{equation}
of subcategories of~$\cC$.
When we do not want to specify~$m$ we just say that~$\cB$ is {\sf Serre compatible}.
\end{definition}

Serre compatibility implies the following nice properties.
\begin{lemma}
\label{lemma:residual-moves}
If~$\cB \subset \cC$ is Serre compatible of length~$m$, then the following hold:
\begin{enumerate}
\item 
\label{item:serre-compatibility}
For any $i \in \ZZ$ there are equalities of subcategories 
\begin{equation}
\label{eq:serre-compatible-2}
\bS_\cC(\bal_\cC^i(\cB)) = \bal_\cC^{i-m}(\cB)
\end{equation} 
and 
\begin{equation}
\label{eq:serre-compatible-residual}
\bS_\cC(\bal_\cC^i(\cR)) = \bal_\cC^{i-m}(\cR).
\end{equation}
In particular, there is an autoequivalence~$\bs_\cR \in \Aut(\cR)$ defined by 
\begin{equation}
\label{eq:sigma-cr}
\bs_\cR \coloneqq (\bS_\cC \circ \bal_\cC^m)\vert_\cR. 
\end{equation}
\item 
For each~$0 \le i \le m$ there is a semiorthogonal decomposition
\begin{equation}
\label{eq:residual-moved}
\cC = \langle \cB, \dots, \bal_\cC^{i-1}(\cB), \bal_\cC^{i}(\cR), \bal_\cC^{i}(\cB), \dots, \bal_\cC^{m-1}(\cB) \rangle.
\end{equation} 
\item 
For each~$0 \le i \le m$ there is an equality 
\begin{equation}
\label{eq:subcategories}
\langle \cR, \cB, \bal_\cC(\cB), \dots, \bal_\cC^{i-1}(\cB) \rangle =
\langle \cB, \bal_\cC(\cB), \dots, \bal_\cC^{i-1}(\cB), \bal_\cC^{i}(\cR) \rangle
\end{equation} 
of subcategories of~$\cC$.
\item 
The residual category $\cR \subset \cC$ is admissible.
\end{enumerate} 
\end{lemma}

\begin{proof}
Part~\eqref{item:serre-compatibility} is proved by the argument of Lemma~\ref{lemma:twist-compatibility}.
Next, applying Lemma~\ref{lemma:serre-orthogonals} to~\eqref{eq:residual-def} and the admissible subcategory~$\bal_\cC^{m-1}(\cB) \subset \cC$
and taking~\eqref{eq:serre-compatible-2} into account, 
we obtain a semiorthogonal decomposition
\begin{equation*}
\cC = \langle \bal_\cC^{-1}(\cB), \cR, \cB, \dots, \bal_\cC^{m-2}(\cB) \rangle.
\end{equation*}
Now, applying the autoequivalence~$\bal_\cC$ 
we obtain
\begin{equation}
\label{eq:residual-moved-1}
\cC = \langle \cB, \bal_\cC(\cR), \bal_\cC(\cB), \dots, \bal_\cC^{m-1}(\cB) \rangle
\end{equation}
which gives~\eqref{eq:residual-moved} for~$i = 1$.
Iterating the argument, we obtain~\eqref{eq:residual-moved} for other~$i$.

Furthermore, \eqref{eq:subcategories} follows easily from~\eqref{eq:residual-moved} 
since both sides of the equality are the orthogonals to~$\langle \bal_\cC^{i}(\cB), \dots, \bal_\cC^{m-1}(\cB) \rangle$ in~$\cC$.
Finally, $\cR$ is left admissible by definition and it is right admissible by~\eqref{eq:residual-moved} for~$i = m$.
\end{proof}

Comparing~\eqref{eq:residual-moved-1} with~\eqref{eq:residual-def} it is easy to prove the following. 

\begin{proposition}[{\cite[Theorem~2.8]{KS20}}]
\label{prop:rotation}
If~$\cB \subset \cC$ is Serre compatible then the endofunctor 
\begin{align*}
\bO_\cB & \coloneqq \bL_\cB \circ \bal_\cC^{\hphantom{-1}} \colon \cC \to \cC
\\
\intertext{restricts to an autoequivalence of~$\cR$ with the inverse given by the  restriction of the endofunctor}
\bO'_\cB & \coloneqq \bal_\cC^{-1} \circ \bR_\cB \colon \cC \to \cC.
\end{align*}
\end{proposition}

The functors~$\bO_\cB$ are known as the {\sf rotation functors}, see also~\cite{KP17,K19}.
The following property of these functors is useful.

\begin{lemma}[{\cite[Lemma~3.13]{K19}}]
\label{lemma:rotation}
If~$\cB \subset \cC$ is Serre compatible of length~$m$, then for any~$0 \le i \le m$ we have
\begin{equation*}
\bO_\cB^i \cong \bL_{\langle \cB, \bal_\cC(\cB), \dots, \bal_\cC^{i-1}(\cB) \rangle} \circ \bal_\cC^{i}.
\end{equation*}
In particular, the functor~$\bO_\cB^i$ is zero on the subcategory 
$\langle \bal_\cC^{-i}(\cB), \bal_\cC^{1-i}(\cB), \dots, \bal_\cC^{-1}(\cB) \rangle \subset \cC$.
\end{lemma}

Rotation functors can also be used to describe the Serre functor of the residual category.

\begin{lemma}
\label{lemma:serre-cr}
If~$\cB \subset \cC$ is Serre compatible of length~$m$ 
the Serre functor of~$\cR$ can be written as
\begin{equation*}
\bS_\cR 
\cong \bs_\cR \circ (\bO_\cB\vert_\cR)^{-m},
\end{equation*}
and the autoequivalences~$\bs_\cR$ and~$\bO_\cB\vert_\cR$ commute.
\end{lemma}

\begin{proof}
This follows from~\cite[Theorem~2.8]{KS20}, but for the reader's convenience we provide a short argument: 
\begin{alignat*}{2}
\bS_{\cR}^{-1} &  \cong \bL_{\langle  \cB, \bal_\cC(\cB), \dots, \bal_\cC^{m-1}(\cB)  \rangle} \circ \bS_{\cC}^{-1} \vert_{\cR}  
& \qquad & \text{~~by Lemma~\ref{lemma:serre-subcategory}} 
\\ 
& \cong \bL_{\langle  \cB, \bal_\cC(\cB), \dots, \bal_\cC^{m-1}(\cB)  \rangle} \circ \bal_{\cC}^{m} \circ \bal_{\cC}^{-m} \circ \bS_{\cC}^{-1} \vert_{\cR} 
\\ 
& \cong \bO_{\cB}^m \circ (\bS_{\cC} \circ \bal_{\cC}^m)^{-1} \vert_{\cR} 
& & \text{~~by Lemma~\ref{lemma:rotation}} 
\\ 
& \cong (\bO_{\cB} \vert_{\cR})^m \circ \bs_{\cR}^{-1} , 
\end{alignat*} 
so passing to inverses gives the desired formula for $\bS_\cR$.
Finally, since by~\eqref{eq:serre-commutativity} the Serre functor~$\bS_\cR$ commutes with any autoequivalence of~$\cR$, 
in particular with~$\bO_\cB\vert_\cR$, we conclude from this that~$\bs_\cR \cong \bS_\cR \circ (\bO_\cB\vert_\cR)^m$ commutes with~$\bO_\cB\vert_\cR$.
\end{proof}

%%%%%%%%%%%%%%%%%%%%%%%%%%%%%%%%%%%%%%%%%%%%%%%%%%%%%%%%%%%%%%%%%%%%%%%%%%%
%%%%%%%%%%%%%%%%%%%%%%%%%%%%%%%%%%%%%%%%%%%%%%%%%%%%%%%%%%%%%%%%%%%%%%%%%%%

\section{The main theorem}
\label{sec:proof}

In this section we prove Theorem~\ref{thm:main}, or, to be more precise, its categorical reformulation 
indicated in Remark~\ref{rem:categorical-reformulation}. 
Recall the definitions of intertwining (Definition~\ref{def:intertwining}), Serre compatibility (Definition~\ref{def:serre-compatibility}),
and twist compatibility (Definition~\ref{def:twist-compatibility}).
Recall also the endofunctors~$\bs_\cR$ and~$\bt_\cR$ of the residual category defined in~\eqref{eq:sigma-cr} and~\eqref{eq:tau-cr} under appropriate assumptions.

\begin{theorem}
\label{thm:main-categorical}
Let~$\Psi \colon \cC \to \cD$ be an \textup(enhanced\textup) spherical functor 
between \textup(enhanced\textup) idempotent complete triangulated categories 
which have Serre functors~$\bS_\cC$ and~$\bS_\cD$.
Let~$\cB \subset \cC$ be an admissible subcategory that induces a rectangular Lefschetz collection and a semiorthogonal decomposition
\begin{equation}
\label{C-sod}
\cC = \langle \cR, \cB, \bal_\cC(\cB), \dots, \bal_\cC^{m-1}(\cB) \rangle
\end{equation}
with respect to an autoequivalence~$\bal_\cC \in \Aut(\cC)$.
Assume that
\begin{itemize}
\item 
the functor~$\Psi$ intertwines between~$\bal_\cC$ and an autoequivalence~$\bal_\cD \in \Aut(\cD)$;
\item 
the subcategory~$\cB \subset \cC$ is Serre compatible of length~$m$;
\item 
the subcategory~$\cB \subset \cC$ is $\bT_{\Psi^!,\Psi}$-twist compatible of degree~$d$ with~$1 \le d \le m$.
\end{itemize}
Then: 
\begin{enumerate}\renewcommand{\theenumi}{\roman{enumi}}
\item 
\label{item:main-categorical:bcb}
There is a subcategory~$\bcR \subset \cD$ defined by the semiorthogonal decomposition
\begin{equation}
\label{eq:cd-bcr}
\cD = 
\begin{cases}
\langle \bcR, \bcB, \bal_\cD(\bcB), \dots, \bal_\cD^{m-d-1}(\bcB) \rangle, & \text{if $d < m$}\\
\hphantom{\langle}\bcR, & \text{if $d = m$}
\end{cases}
\end{equation} 
where in the first case~$\bcB \coloneqq \Psi(\cB)$ and the functor~$\Psi\vert_\cB$ is fully faithful. 
\item 
\label{item:main-categorical:bcb-compatibility}
If $d < m$ the category~$\bcB$ is Serre compatible of length~$m - d$ and~$\bT_{\Psi,\Psi^!}$-compatible of degree~$d$. 
In any case the functors~$\bS_\cD \circ \bal_\cD^{m-d}$ and~$\bT_{\Psi,\Psi^!} \circ \bal_\cD^d$ preserve~\eqref{eq:cd-bcr} 
and induce autoequivalences of the residual category~$\bcR$
\begin{equation}
\label{eq:sigmax-rhox}
\bs_{\bcR} \coloneqq (\bS_\cD \circ \bal_\cD^{m-d})\vert_{\bcR}
\qquad\text{and}\qquad 
\bt_{\bcR} \coloneqq (\bT_{\Psi,\Psi^!} \circ \bal_\cD^d)\vert_{\bcR} . 
\end{equation} 
\item 
\label{item:main-categorical:psi-cr}
The functor~$\Psi$ takes~$\cR$ to~$\bcR$ and its restriction
\begin{equation*}
\Psi_\cR \colon \cR \to \bcR
\end{equation*}
is spherical.
\item 
\label{item:main-categorical:serre}
If~$c = \gcd(d,m)$ then the Serre functors~$\bS_\cR$ and~$\bS_{\bcR}$ have the factorizations
\begin{equation}
\label{eq:serre-bcr-power}
\bS_{\cR}^{d/c} \cong \bT_{\Psi_\cR^!,\Psi_\cR}^{m/c} \circ \bt_{\cR}^{-m/c} \circ \bs_{\cR}^{d/c}
\qquad\text{and}\qquad
\bS_{\bcR}^{d/c} \cong \bT_{\Psi_\cR,\Psi_\cR^!}^{(m-d)/c} \circ \bt_{\bcR}^{(d-m)/c} \circ \bs_{\bcR}^{d/c}
\end{equation}
and all the factors on the right hand sides of~\eqref{eq:serre-bcr-power} commute. 
\end{enumerate}
\end{theorem}

\begin{remark}
The isomorphisms in part~\eqref{item:main-categorical:serre} of Theorem~\ref{thm:main-categorical} 
can also be written in the following more symmetric form: 
\begin{equation*}
(\bS_{\cR} \circ \bs_{\cR}^{-1})^{d/c} \cong (\bT_{\Psi_\cR^!,\Psi_\cR} \circ \bt_{\cR}^{-1})^{m/c}
\qquad\text{and}\qquad
(\bS_{\bcR} \circ \bs_{\bcR}^{-1})^{d/c} \cong (\bT_{\Psi_\cR^!,\Psi_\cR} \circ \bt_{\bcR}^{-1})^{(m-d)/c}.
\end{equation*}
Note that the exponents in the left hand sides correspond to the degree of the spherical functor
while those in the right hand sides correspond to the lengths of the rectangular Lefschetz collections.
\end{remark}

The proof of Theorem~\ref{thm:main-categorical} takes the rest of this section.
In~\S\ref{subsec:lnduced-Lefschetz} we construct the semiorthogonal decomposition~\eqref{eq:cd-bcr},
prove its Serre and twist compatibilities, and construct the autoequivalences~\eqref{eq:sigmax-rhox}.
In~\S\ref{subsec:residual-spherical} we construct the functor~$\Psi_\cR$ and prove that it is also spherical.
Finally, in~\S\ref{subsec:proof} we prove the theorem and deduce some of its consequences.

Throughout this section we work under the assumptions of Theorem~\ref{thm:main-categorical}.
We mostly concentrate on the case~$d < m$; 
in fact, in the case~$d = m$ all the statements are tautological, 
except for the claim that $\Psi_{\cR}$ is spherical
(which follows from a result of Addington~\cite[Proposition~2.1]{Add16}) 
and the first formula in~\eqref{eq:serre-bcr-power} 
(which can be deduced from Corollary~\ref{cor:spherical-sod}).
However, we prove some intermediate statements for~$d = m$ as well, because they are interesting by themselves.

\subsection{Induced Lefschetz collection}
\label{subsec:lnduced-Lefschetz}

In this subsection we show that the target category~$\cD$ of the spherical functor~$\Psi$ 
has a natural rectangular Lefschetz collection of length~$m - d$ with respect to the autoequivalence~$\bal_\cD$, 
which is also Serre and twist compatible. 

\begin{lemma}
\label{prop:bcb}
If $d < m$ the functor~$\Psi$ is fully faithful on the subcategory~$\cB \subset \cC$.
Furthermore, if
\begin{equation}
\label{eq:cbx}
\bcB \coloneqq \Psi(\cB) \subset \cD
\end{equation} 
then~$\bcB$ is admissible in~$\cD$ and induces a rectangular Lefschetz collection of length~$m - d$
with respect to the autoequivalence~$\bal_\cD$.
In particular, for any~$1 \le d \le m$ 
there is a semiorthogonal decomposition~\eqref{eq:cd-bcr}, 
where~$\bcR$ is the residual category if $d < m$ and~$\bcR = \cD$ if $d = m$. 
\end{lemma}

\begin{proof}
Full faithfulness of~$\Psi$ on~$\cB$, admissibility of the subcategory~$\bcB \subset \cD$, 
and semiorthogonality of the categories~$\bal_\cD^i(\bcB)$, $0 \le i \le m - d - 1$, 
is proved in~\cite[Lemma~3.10]{K19}.
It follows that there is a semiorthogonal decomposition~\eqref{eq:cd-bcr}.
\end{proof}

Next we check compatibility of~\eqref{eq:cd-bcr} with the spherical twist and Serre functors.  
We start with an obvious observation. 

\begin{lemma}
\label{lemma:bt-bal}
The functor~$\bT_{\Psi,\Psi^!}$ commutes with~$\bal_\cD$ and the functor~$\bT_{\Psi^!,\Psi}$ commutes with~$\bal_\cC$.
\end{lemma}
\begin{proof}
Follows from Lemma~\ref{lemma:twists-commute} and the intertwining property of~$\Psi$.
\end{proof}

Recall the notion of twist compatibility introduced in Definition~\ref{def:twist-compatibility}. 

\begin{lemma}
\label{lemma:bcb-twist-compatible}
If $d < m$ the subcategory~$\bcB \subset \cD$ is~$\bT_{\Psi,\Psi^!}$-twist compatible of degree~$d$, 
and thus 
\begin{equation}
\label{eq:btx-cbx}
\bT_{\Psi,\Psi^!}(\bal_\cD^{i}(\bcB)) = \bal_\cD^{i-d}(\bcB),
\qquad 
\bT_{\Psi,\Psi^!}(\bal_\cD^{i}(\bcR)) = \bal_\cD^{i-d}(\bcR),
\end{equation} 
for all~$i \in \ZZ$. 
Further, for any~\mbox{$1 \le d \le m$} there is an autoequivalence~$\bt_\bcR \in \Aut(\bcR)$ defined by~\eqref{eq:sigmax-rhox}.
\end{lemma}

\begin{proof}
If $d < m$ we have
\begin{equation*}
\bT_{\Psi,\Psi^!}(\bcB) =
(\bT_{\Psi,\Psi^!} \circ \Psi)(\cB) =
(\Psi \circ \bT_{\Psi^!,\Psi})(\cB) =
(\Psi \circ \bal_\cC^{-d})(\cB) =
(\bal_\cD^{-d} \circ \Psi)(\cB) =
\bal_\cD^{-d}(\bcB), 
\end{equation*}
where the first and last equalities hold by definition~\eqref{eq:cbx} of~$\bcB$,
the second is given by~\eqref{eq:spherical-intertwining},
the third holds by $\bT_{\Psi^!,\Psi}$-twist compatibility of~$\cB$,
and the fourth by the intertwining assumption.
Therefore~$\bcB \subset \cD$ is $\bT_{\Psi, \Psi^!}$-twist compatible of degree~$d$. 
The other claims in case~$d < m$ then follow from 
Lemma~\ref{lemma:twist-compatibility} 
(whose hypotheses are met in view of Lemma~\ref{prop:bcb} and Lemma~\ref{lemma:bt-bal}). 

It remains only to observe that when~$d = m$, we have an equality~$\bcR = \cD$, so the fact that~$\bt_{\bcR}$ is an autoequivalence is obvious. 
\end{proof}

Recall the notion of Serre compatibility introduced in Definition~\ref{def:serre-compatibility}.

\begin{lemma}
\label{lemma:serre-x-crx}
If $d < m$ the subcategory~$\bcB \subset \cD$ is Serre compatible of length~$m - d$, and thus 
\begin{equation}
\label{eq:bsx}
\bS_\cD({\bal}_{\cD}^i(\bcB)) = \bal_\cD^{i+d-m}(\bcB),
\qquad 
\bS_\cD({\bal}_{\cD}^i(\bcR)) = \bal_\cD^{i+d-m}(\bcR),
\end{equation} 
for all $i \in \ZZ$. 
Further, for any~\mbox{$1 \le d \le m$} there is an autoequivalence~$\bs_{\bcR} \in \Aut(\bcR)$ defined by~\eqref{eq:sigmax-rhox}. 
\end{lemma}

\begin{proof}
If $d < m$ then 
by~\eqref{eq:psi-shriek-shriek}, the definition~\eqref{eq:cbx} of~$\bcB$, and~\eqref{eq:btx-cbx}, we have
\begin{equation*}
\Psi^{!!}(\cB) = \bT_{\Psi,\Psi^!}^{-1}(\Psi(\cB)) = \bT_{\Psi,\Psi^!}^{-1}(\bcB) = \bal_\cD^d(\bcB).
\end{equation*}
On the other hand, by~\eqref{eq:double-right}, Serre compatibility~\eqref{eq:serre-compatible} of~$\cB$, 
the intertwining property of~$\Psi$, and the definition of~$\bcB$,
we have
\begin{equation*}
\Psi^{!!}(\cB)
= \bS_\cD(\Psi(\bS_\cC^{-1}(\cB)))
= \bS_\cD(\Psi(\bal_\cC^m(\cB)))
= \bS_\cD(\bal_\cD^m(\Psi(\cB)))
= \bS_\cD(\bal_\cD^m(\bcB)). 
\end{equation*}
Comparing the two equalities, we obtain~$\bS_\cD(\bal_\cD^m(\bcB)) = \bal_\cD^d(\bcB)$. 
Using that the Serre functor~$\bS_{\cD}$ commutes with autoequivalences, we thus obtain the first equality in~\eqref{eq:bsx} for~$i = 0$. 
Together with Lemma~\ref{prop:bcb}, this proves that~$\bcB \subset \cD$ is Serre compatible of length~$m-d$. 
The other claims in case~$d < m$ then follow from Lemma~\ref{lemma:residual-moves}. 

It remains only to observe that when~$d = m$, we have~$\bcR = \cD$ and~\eqref{eq:sigmax-rhox} reads as~$\bs_\bcR \coloneqq \bS_\cD$, 
which is evidently an autoequivalence. 
\end{proof}

The following description of the residual category $\bcR$ is sometimes useful. 

\begin{lemma}
\label{lemma-bcR-characterization}
The subcategory $\bcR \subset \cD$ is characterized by the equality 
\begin{equation}
\label{eq:bcr-characterization}
\bcR = \{ G \in \cD \mid \Psi^!(G) \in 
\langle \bal_\cC^{-d}(\cR), \bal_\cC^{-d}(\cB), \dots, \bal_\cC^{-1}(\cB) \rangle
= \langle \bal_\cC^{-d}(\cB), \dots, \bal_\cC^{-1}(\cB), \cR \rangle 
\}.
\end{equation}
\end{lemma} 

\begin{proof}
In view of the Serre compatibility of~$\bcB \subset \cD$ proved in Lemma~\ref{lemma:serre-x-crx}, 
the equality of categories
\begin{equation*}
\langle \bal_\cC^{-d}(\cR), \bal_\cC^{-d}(\cB), \dots, \bal_\cC^{-1}(\cB) \rangle
= \langle \bal_\cC^{-d}(\cB), \dots, \bal_\cC^{-1}(\cB), \cR \rangle
\end{equation*}
follows from~\eqref{eq:subcategories}. 
The description~\eqref{eq:bcr-characterization} of~$\bcR$ follows from its definition as an orthogonal 
category, adjunction, and the decomposition of~$\cC$ obtained by applying~${\bal}_{\cC}^{-d}$ to~\eqref{C-sod}; 
see~\cite[Lemma~3.11]{K19} for a similar argument.
\end{proof} 

Recall from~\S\ref{subsection-serre-compatible-lc} that
if~$d < m$, then we have the rotation functor 
\begin{equation*}
\bO_{\bcB} \coloneqq \bL_{\bcB} \circ {\bal}_{\cD} \colon \cD \to \cD 
\end{equation*} 
associated to $\bcB$. 
By Proposition~\ref{prop:rotation} and the Serre compatibility proved in Lemma~\ref{lemma:serre-x-crx}, 
this endofunctor restricts to an autoequivalence of $\bcR$, with inverse given by the restriction of the 
endofunctor 
\begin{equation*}
\bO_{\bcB}' \coloneqq {\bal}_{\cD}^{-1} \circ \bR_{\bcB} \colon \cD \to \cD . 
\end{equation*} 
If $d = m$, then unlike when $d < m$, the functor~$\Psi_\cB \coloneqq \Psi\vert_\cB$ 
is not fully faithful and the subcategory~$\bcB \subset \cD$ is not defined; 
however,~$\Psi_\cB$ is spherical by Proposition~\ref{prop:spherical-sod} and we define
\begin{equation*}
\bO_\bcB \coloneqq \bT_{\Psi_\cB,\Psi_\cB^!} \circ \bal_\cD.
\end{equation*}
This is a natural extension of our definition in case $d < m$, as 
$\bT_{\Psi_\cB,\Psi_\cB^!}$ and $\bL_{\bcB}$ are given by analogous formulas 
(cf.~\cite[Remark~7.5]{KP17}). 
Note that in view of Serre compatibility, 
Lemma~\ref{lemma:serre-cr} gives the following formula for the Serre functor of $\bcR$. 

\begin{lemma}
\label{lemma:serre-bcR}
The Serre functor of $\bcR$ can be written as 
\begin{equation*}
\bS_{\bcR} \cong \bs_{\bcR} \circ (\bO_\bcB \vert_{\bcR})^{d-m} . 
\end{equation*}
\end{lemma}

Finally, we prove commutativity of the autoequivalences of~$\bcR$ introduced above.  
We also prove commutativity of the corresponding autoequivalences of~$\cR$, 
namely~$\bs_{\cR} \in \Aut(\cR)$ defined in Lemma~\ref{lemma:serre-cr}, $\bt_{\cR} \in \Aut(\cR)$ defined in~\eqref{eq:tau-cr}, 
and the rotation functor~$\bO_\cB$ defined in Proposition~\ref{prop:rotation}. 

\begin{lemma}
\label{lemma:rotation-tau-commute}
The autoequivalences  $\bs_{\cR}, \bt_{\cR}, \bO_\cB\vert_{\cR} \in \Aut(\cR)$ all commute, and 
the autoequivalences $\bs_{\bcR}, \bt_{\bcR}, \bO_\bcB\vert_{\bcR} \in \Aut(\bcR)$ all commute. 
\end{lemma}

\begin{proof}
First we consider the autoequivalences of~$\cR$. 
Lemma~\ref{lemma:serre-cr} gives the commutativity of~$\bs_{\cR}$ and~$\bO_\cB\vert_{\cR}$. 

For commutativity of~$\bt_{\cR}$ and~$\bO_\cB\vert_{\cR}$, let~$F \in \cR$ and consider the mutation triangle 
\begin{equation*}
B \to \bal_\cC(F) \to \bO_\cB(F),
\end{equation*}
where~$B \in \cB$ and~$\bO_\cB(F) \in \cR$ by Proposition~\ref{prop:rotation}. 
Applying the functor~$\bT_{\Psi^!,\Psi} \circ \bal_\cC^d$ to this triangle
and using the definition~\eqref{eq:tau-cr} of~$\bt_\cR$ we obtain an exact triangle 
\begin{equation*}
\bT_{\Psi^!,\Psi}(\bal_\cC^d(B)) \to \bT_{\Psi^!,\Psi}(\bal_\cC^{d+1}(F)) \to \bt_\cR(\bO_\cB(F)).
\end{equation*}
By twist compatibility of~$\cB$ and~\eqref{eq:twist-compatible-2} the first term is in~$\cB$.
Furthermore, since~$\bT_{\Psi^!,\Psi}$ commutes with~$\bal_\cC$ (Lemma~\ref{lemma:bt-bal}), 
the middle term can be rewritten as~$\bal_\cC(\bt_\cR(F))$.
Finally, the last term is in~$\cR$. 
Therefore this is a mutation triangle and we conclude that
\begin{equation*}
\bt_\cR(\bO_\cB(F)) \cong \bL_\cB(\bal_\cC(\bt_\cR(F))) \cong \bO_\cB(\bt_\cR(F)).
\end{equation*}
This proves the commutativity of $\bt_{\cR}$ and $\bO_\cB\vert_{\cR}$. 

Finally, note that~$\bs_\cR \cong \bS_{\cR} \circ  (\bO_\cB\vert_\cR)^{m}$ by Lemma~\ref{lemma:serre-cr}. 
But the Serre functor~$\bS_{\cR}$ commutes with any autoequivalence and 
we already showed above that~$\bO_\cB\vert_\cR$ commutes with~$\bt_{\cR}$, 
so~$\bs_\cR$ commutes with~$\bt_{\cR}$. 

For the autoequivalences of $\bcR$, if $d < m$ then commutativity follows by the same argument as above. 
If $d = m$, then by definition $\bcR = \cD$ and $\bs_{\bcR} = \bS_{\cD}$, so $\bs_{\bcR}$ commutes with any 
autoequivalence, and we only need to show that $\bt_{\bcR}$ and $\bO_{\bcB}\vert_{\bcR}$ commute. 
This follows from Corollary~\ref{cor:spherical-sod}. 
\end{proof}

\subsection{Residual spherical functor}
\label{subsec:residual-spherical}

In this subsection we prove that the spherical functor~\mbox{$\Psi \colon \cC \to \cD$} 
induces a spherical functor~$\Psi_\cR \colon \cR \to \bcR$.
This generalizes the property of residual categories observed in~\cite[Lemma~2.5]{BKS}.
We also relate the spherical twists of~$\cR$ and~$\bcR$ to other functors discussed in~\S\ref{subsec:lnduced-Lefschetz}. 

\begin{lemma}
\label{lemma:residual-psi} 
We have~$\Psi(\cR) \subset \bcR$.
\end{lemma}

\begin{proof}
Let~$F \in \cR$.
Consider the triangle
\begin{equation*}
\bT_{\Psi^!,\Psi}(F) \to F \to \Psi^!(\Psi(F)).
\end{equation*}
The second term is in~$\cR$ and by~\eqref{eq:twist-compatible-residual} the first term is in~$\bal_\cC^{-d}(\cR)$.
Therefore, both these terms are contained in the subcategory
\begin{equation*}
\langle \bal_\cC^{-d}(\cR), \bal_\cC^{-d}(\cB), \dots, \bal_\cC^{-1}(\cB) \rangle
= \langle \bal_\cC^{-d}(\cB), \dots, \bal_\cC^{-1}(\cB), \cR \rangle, 
\end{equation*}
where the equality follows from~\eqref{eq:subcategories}.
Therefore, the third term $\Psi^!(\Psi(F))$ is also contained in this subcategory.
Now looking at~\eqref{eq:bcr-characterization} we conclude that~$\Psi(F) \in \bcR$.
\end{proof}

Since~$\cR$ and~$\bcR$ are full subcategories of~$\cC$ and~$\cD$, it follows from Lemma~\ref{lemma:residual-psi} 
that there exists a functor~$\Psi_\cR \colon \cR \to \bcR$ such that
\begin{equation}
\label{eq:psi-psi}
\xi_{\cD} \circ \Psi_\cR = \Psi \circ \xi,
\end{equation} 
where $\xi \colon \cR \to \cC$ and~$\xi_\cD \colon \bcR \to \cD$ are the embedding functors.
Note that~$\Psi_\cR$ has both adjoints 
\begin{equation}
\label{eq:Psi-adjoints}
\Psi_\cR^* \cong \xi^* \circ \Psi^* \circ \xi_\cD
\qquad\text{and}\qquad 
\Psi_\cR^! \cong \xi^! \circ \Psi^! \circ \xi_\cD,
\end{equation}
where~$\xi^*$ and~$\xi^!$ are the adjoints of~$\xi$ (recall that~$\cR$ is admissible by Lemma~\ref{lemma:residual-moves}).

For later use, we note the following description of the kernel of $\Psi^!_{\cR}$. 
\begin{lemma}
\label{lemma-ker-Psi!}
We have $\ker(\Psi_{\cR}^!) = \{ G \in \cD \mid \Psi^!(G) \in \langle \bal_\cC^{-d}(\cB), \dots, \bal_\cC^{-1}(\cB) \rangle  \}$. 
\end{lemma} 

\begin{proof}
Follows from the characterization~\eqref{eq:bcr-characterization} of $\bcR \subset \cD$ 
and the formula~\eqref{eq:Psi-adjoints} for $\Psi_{\cR}^!$. 
\end{proof}

In the rest of this subsection we prove that~$\Psi_\cR$ is a spherical functor. 
To do so, by the criterion of Proposition~\ref{proposition:spherical-criterion}, it 
suffices to show the twist functors~$\bT_{\Psi_\cR^!,\Psi_\cR} \colon \cR \to \cR$
and~$\bT_{\Psi_\cR,\Psi_\cR^!} \colon \bcR \to \bcR$ are autoequivalences.
Recall the autoequivalence~$\bt_\cR$ defined in~\eqref{eq:tau-cr}.

\begin{proposition}
\label{prop:bt-cr}
There is an isomorphism of functors 
\begin{equation*}
\bT_{\Psi_\cR^!,\Psi_\cR} \cong (\bO_\cB\vert_\cR)^{-d} \circ \bt_\cR.
\end{equation*}
In particular, $\bT_{\Psi_\cR^!,\Psi_\cR}$ is an autoequivalence of~$\cR$.
\end{proposition}

\begin{proof}
As we already mentioned, the embedding functor~$\xi \colon \cR \to \cC$ has both adjoints;
the same argument applies to~$\xi_\cD \colon \bcR \to \cD$.
Composing the defining triangle of~$\bT_{\Psi^!,\Psi}$ with~$\xi^!$ on the left and~$\xi$ on the right, and using full faithfulness of~$\xi$, 
we obtain the triangle
\begin{equation*}
\xi^! \circ \bT_{\Psi^!,\Psi} \circ \xi \to \id_{\cR} \to \xi^! \circ \Psi^! \circ \Psi \circ \xi.
\end{equation*}
On the other hand, using~\eqref{eq:psi-psi} and full faithfulness of~$\xi_\cD$ we rewrite the last term as
\begin{equation*}
\xi^! \circ \Psi^! \circ \Psi \circ \xi \cong
\Psi_\cR^! \circ \xi_\cD^! \circ \xi_\cD \circ \Psi_\cR \cong 
\Psi_\cR^! \circ \Psi_\cR.
\end{equation*}
Moreover, with this identification the last morphism in the triangle is the unit of the adjunction between~$\Psi_\cR^!$ and~$\Psi_\cR$.
Therefore, the triangle coincides with the defining triangle of~$\bT_{\Psi_\cR^!,\Psi_\cR}$ and we deduce an isomorphism
\begin{equation*}
\bT_{\Psi_\cR^!,\Psi_\cR} \cong \xi^! \circ \bT_{\Psi^!,\Psi} \circ \xi.
\end{equation*}
Using~\eqref{eq:tau-cr} and Lemma~\ref{lemma:bt-bal} we rewrite
\begin{equation*}
\xi^! \circ \bT_{\Psi^!,\Psi} \circ \xi \cong
\xi^! \circ \bal_\cC^{-d} \circ \xi \circ \bt_\cR. 
\end{equation*}
Finally, we have isomorphisms 
\begin{equation*}
\xi^! \circ \bal_\cC^{-d} \circ \xi \cong 
\bR_{\langle \bal_\cC^{-d}(\cB), \dots, \bal_\cC^{-1}(\cB) \rangle}  \circ \bal_\cC^{-d} \circ \xi \cong 
\bal_\cC^{-d} \circ \bR_{\langle \cB, \dots, \bal_\cC^{d-1}(\cB) \rangle}   \circ \xi \cong 
(\bO_\cB\vert_\cR)^{-d}, 
\end{equation*}
where the first follows from the inclusion~$\bal_\cC^{-d}(\cR) \subset \langle \bal_\cC^{-d}(\cB), \dots, \bal_\cC^{-1}(\cB), \cR \rangle$ 
implied by the equality~\eqref{eq:subcategories}, 
the second from general properties of mutation functors, and the last from Lemma~\ref{lemma:rotation} by passing to inverse functors. 
All together, this proves the required isomorphism~$\bT_{\Psi_\cR^!,\Psi_\cR} \cong (\bO_\cB\vert_\cR)^{-d} \circ \bt_\cR$.
\end{proof}

To prove that the other twist~$\bT_{\Psi_\cR,\Psi_\cR^!}$ is an autoequivalence we need some preparation.
\begin{lemma}
\label{lemma-gammai} 
For~$1 \leq i \leq d$, there is a morphism of functors 
\begin{equation*}
\gamma^i \colon \Psi \circ \bO_{\cB}^i \to \bO_{\bcB}^i \circ \Psi 
\end{equation*} 
which restricts to an isomorphism on the subcategory~$\langle \bal_\cC^{d-i}(\cB), \dots, \bal_\cC^{d-1}(\cB) \rangle^{\perp} \subset \cC$. 
In particular, we have an isomorphism of functors~$\cR \to \bcR$
\begin{equation*}
\Psi_\cR \circ (\bO_\cB\vert_\cR)^d \cong (\bO_\bcB\vert_\bcR)^d \circ \Psi_\cR.
\end{equation*}
\end{lemma} 

\begin{proof}
This is essentially~\cite[Proposition~3.17]{K19} for~$\Phi = \Psi^!$ and~$\Phi^* = \Psi$; 
the setup there is less general than ours, but the same argument works. 
For the reader's convenience, we include the proof. 

First we consider the case~$i = 1$. 
Let~$\beta \colon \cB \to \cC$ be the inclusion of~$\cB$, and consider the functor~$\Psi_{\cB} \coloneqq \Psi \circ \beta$. 
Then by the definition of the rotation functors, we have exact triangles 
\begin{align}
\label{OB}
\beta \circ \beta^! \circ \bal_{\cC} & \to \bal_{\cC} \to \bO_{\cB} ,  \\ 
\label{OBD} 
\Psi_{\cB} \circ \Psi_{\cB}^! \circ \bal_{\cD} & \to \bal_{\cD} \to \bO_{\bcB}. 
\end{align}  
We have a commutative diagram 
\begin{equation*}
\xymatrix{
\Psi \circ \beta \circ \beta^! \circ \bal_\cC \ar[r] \ar[d] & \Psi \circ \bal_{\cC} \ar[d]^{\cong} \ar[r] & \Psi \circ \bO_{\cB} \ar@{-->}[d]^{\gamma} \\ 
\Psi_{\cB} \circ \Psi_{\cB}^! \circ \bal_{\cD} \circ \Psi \ar[r] & \bal_{\cD} \circ \Psi \ar[r] & \bO_{\bcB} \circ \Psi 
}
\end{equation*} 
constructed as follows. 
The rows are the composition of the triangles~\eqref{OB} and~\eqref{OBD} with~$\Psi$.  
The middle vertical isomorphism is given by the intertwining data for~$\Psi$ between~$\bal_{\cC}$ and~$\bal_{\cD}$. 
To describe the left vertical arrow, note that the factorization~$\Psi_{\cB} = \Psi \circ \beta$ 
combined with the intertwining property of~$\Psi^!$ between~$\bal_{\cD}$ and~$\bal_{\cC}$ (Lemma~\ref{lemma:twists-commute}) implies that 
\begin{equation*}
\Psi_{\cB} \circ \Psi_{\cB}^! \circ \bal_{\cD} \circ \Psi \cong 
\Psi \circ \beta \circ \beta^! \circ \bal_{\cC} \circ \Psi^! \circ \Psi. 
\end{equation*} 
Via this isomorphism, the left vertical arrow is induced by the unit map~$\id \to \Psi^! \circ \Psi$. 
It is not hard to check that the left square commutes and, therefore, it extends to a morphism of triangles, 
giving the dotted arrow~$\gamma$. 

By construction, the fiber of the left vertical morphism is the functor~$\Psi \circ \beta \circ \beta^! \circ \bal_{\cC} \circ \bT_{\Psi^! \Psi}$. 
By twist compatibility of~$\cB \subset \cC$ and~\eqref{eq:twist-compatible-2} 
we see that~$\bal_{\cC} \circ \bT_{\Psi^! \Psi}$ takes~$(\bal_\cC^{d-1}(\cB))^{\perp}$ 
to the subcategory~$\cB^{\perp} \subset \cC$, which is killed by~$\beta^!$. 
Therefore, the left vertical arrow, and hence also the arrow~$\gamma$, is 
an isomorphism on the subcategory~$(\bal_\cC^{d-1}(\cB))^{\perp}$. 
This completes the proof for~$i = 1$. 

Now suppose~$i > 1$. 
Let~$\gamma^i \colon \Psi \circ \bO^i_{\cB} \to \bO^i_{\bcB} \circ \Psi$ be the $i$-th iterate of~$\gamma$. 
By induction, we may assume that~$\gamma^{i-1}$ 
is an isomorphism on the subcategory~$\langle \bal_\cC^{d-i+1}(\cB), \dots, \bal_\cC^{d-1}(\cB) \rangle^{\perp} \subset \cC$. 
Let~$F \in \langle \bal_\cC^{d-i}(\cB), \dots, \bal_\cC^{d-1}(\cB) \rangle^{\perp}$. 
By Lemma~\ref{lemma:rotation} we have 
\begin{equation*}
\bO_\cB^{i-1}(F) \cong \bL_{\langle \cB, \bal_\cC(\cB), \dots, \bal_\cC^{i-2}(\cB) \rangle}(\bal_\cC^{i-1}(F)), 
\end{equation*}
by our choice of~$F$ we have~$\bal_\cC^{i-1}(F) \in (\bal_\cC^{d-1}(\cB))^{\perp}$, 
and by the decomposition~\eqref{C-sod} 
we have~$\langle \cB, \bal_\cC(\cB), \dots, \bal_\cC^{i-2}(\cB) \rangle \subset (\bal_\cC^{d-1}(\cB))^{\perp}$, 
so it follows that~$\bO_\cB^{i-1}(F) \in (\bal_\cC^{d-1}(\cB))^{\perp}$. 
The~$i = 1$ case thus applies to show that $\gamma$ 
induces an isomorphism~$\Psi \circ \bO_{\cB}^i \cong \bO_{\bcB} \circ \Psi \circ \bO_{\cB}^{i-1}$ 
on the subcategory~$\langle \bal_\cC^{d-i}(\cB), \dots, \bal_\cC^{d-1}(\cB) \rangle^{\perp}$. 
Then the inductive hypothesis for~$\gamma^{i-1}$ implies that~$\gamma^i$ is also an isomorphism on the same subcategory. 

The last claim follows by substituting~$i = d$ since~$\cR \subset \langle \cB, \bal_\cC(\cB), \dots, \bal_\cC^{d-1}(\cB) \rangle^{\perp}$.
\end{proof}

\begin{proposition}
\label{proposition-rotation-twist-triangle}
For each $0 \leq i \leq d$, there is an exact triangle of functors 
\begin{equation*}
\Psi \circ \bO_\cB^i \circ \Psi^! \to \bO_\bcB^i \to \bT_{\Psi,\Psi^!} \circ \bal_\cD^i
\end{equation*} 
\end{proposition}

\begin{proof}
This is essentially~\cite[Proposition~3.17]{K19} for~$\Phi = \Psi^!$ and~$\Phi^* = \Psi$;  
the setup there is less general than ours, but the same argument works. 
For the reader's convenience, we include the proof. 

Consider the morphism of functors  
\begin{equation*}
\delta^i \colon \Psi \circ \bO_{\cB}^{i} \circ \Psi^! \to \bO_{\bcB}^i \circ \Psi \circ \Psi^! \to \bO_{\bcB}^i, 
\end{equation*} 
where the first arrow is the composition of the morphism~$\gamma^i \colon \Psi \circ \bO_{\cB}^i \to \bO_{\bcB}^i \circ \Psi$ 
from Lemma~\ref{lemma-gammai} with~$\Psi^!$, 
and the second arrow is the composition of the counit map~$\Psi \circ \Psi^! \to \id$ with~$\bO_{\bcB}^i$. 
We will prove by induction on~$i$ that the cone of this morphism is~$\bT_{\Psi,\Psi^!} \circ \bal_\cD^i$. 
The base case~$i = 0$ is just the defining exact triangle for~$\bT_{\Psi, \Psi^!}$. 

Now suppose~$i > 0$. 
Denoting by~$\beta \colon \cB \to \cC$ the inclusion, 
we have a commutative diagram with exact rows 
\begin{equation*}
\xymatrix{
\Psi \circ \bO_{\cB}^{i-1} \circ \beta \circ \beta^! \circ \bal_{\cC} \circ \Psi^! \ar[r] \ar[d] & 
\Psi \circ \bO_{\cB}^{i-1} \circ \bal_{\cC} \circ \Psi^! \ar[r] \ar[d] & 
\Psi \circ \bO_{\cB}^{i-1} \circ \bO_{\cB} \circ \Psi^! \ar[d]^{\delta^i} \\ 
  \bO_{\bcB}^{i-1} \circ \Psi \circ \beta \circ \beta^! \circ  \Psi^! \circ \bal_{\cD}  \ar[r] & 
  \bO_{\bcB}^{i-1} \circ \bal_{\cD} \ar[r] & 
  \bO_{\bcB}^{i-1} \circ   \bO_{\bcB}
}
\end{equation*} 
constructed as follows. 
The top row is the triangle~\eqref{OB} composed with~$\Psi \circ \bO_{\cB}^{i-1}$ on the left and~$\Psi^!$ on the right, 
the bottom row is the triangle~\eqref{OBD} composed with~$\bO_{\bcB}^{i-1}$ on the left,
and the morphism of triangles is constructed as in~\cite{K19}.
The left vertical arrow is induced by~$\gamma^{i-1}$ combined with the intertwining of~$\Psi^!$ 
between~$\bal_{\cD}$ and~$\bal_{\cC}$ (Lemma~\ref{lemma:twists-commute}), 
the middle vertical arrow is induced by~$\delta^{i-1}$ combined with the same intertwining, 
and the right vertical arrow is the map~$\delta^i$. 
As~$i \leq d$, the image~$\cB$ of~$\beta$ 
is contained in the subcategory~$\langle \bal_\cC^{d-i+1}(\cB), \dots, \bal_\cC^{d-1}(\cB) \rangle^{\perp} \subset \cC$. 
Therefore, by Lemma~\ref{lemma-gammai}, the left vertical arrow is an isomorphism, 
so the octahedral axiom implies the cone of~$\delta^i$ is isomorphic to the cone of the middle vertical map. 
But by the induction hypothesis, the cone of the middle vertical map is precisely~$\bT_{\Psi, \Psi^!} \circ \bal_{\cD}^{i}$. 
\end{proof}

\begin{lemma}
\label{lemma:bomd-phi} 
There is an isomorphism of functors
\begin{equation*}
(\bO_\cB^d \circ \Psi^!)\vert_{\bcR} \cong (\bO_\cB^d \circ \xi \circ \xi^! \circ \Psi^!)\vert_{\bcR}.
\end{equation*}
\end{lemma}

\begin{proof}
For any~$G \in \bcR$ we have $\Psi^!(G) \in \langle \bal_\cC^{-d}(\cB), \dots, \bal_\cC^{-1}(\cB), \cR \rangle$, see~\eqref{eq:bcr-characterization}.
Therefore, $\xi^!(\Psi^!(G))$ is the component of~$\Psi^!(G)$ in~$\cR$, hence we have a exact triangle
\begin{equation*}
\xi(\xi^!(\Psi^!(G))) \to \Psi^!(G) \to F,
\end{equation*}
where $F \in \langle \bal_\cC^{-d}(\cB), \dots, \bal_\cC^{-1}(\cB) \rangle$.
It remains to note that the category~$\langle \bal_\cC^{-d}(\cB), \dots, \bal_\cC^{-1}(\cB) \rangle$ is annihilated by~$\bO_\cB^d$ by Lemma~\ref{lemma:rotation}, 
hence $\bO_\cB^d(\xi(\xi^!(\Psi^!(G)))) \cong \bO_\cB^d(\Psi^!(G))$.
\end{proof}

The following relation is crucial.
Note that the statement is analogous to Proposition~\ref{prop:bt-cr}, but the proof is very different. 

\begin{proposition}
\label{prop:bt-crx}
We have an isomorphism of functors
\begin{equation*}
\bT_{\Psi_\cR,\Psi_\cR^!} \cong (\bO_\bcB\vert_\bcR)^{-d} \circ \bt_{\bcR}.
\end{equation*}
In particular, $\bT_{\Psi_\cR,\Psi_\cR^!}$ is an autoequivalence of~$\bcR$.
\end{proposition}

\begin{proof}
First, assume~$d < m$.
Using Lemma~\ref{lemma:bomd-phi} and 
Lemma~\ref{lemma-gammai}
we obtain
\begin{equation*}
(\Psi \circ \bO_\cB^d \circ \Psi^!)\vert_{\bcR} \cong 
(\Psi \circ \bO_\cB^d \circ \xi \circ \xi^! \circ \Psi^!)\vert_{\bcR} \cong 
(\bO_{\bcB}^d \circ \Psi \circ \xi \circ \xi^! \circ \Psi^!)\vert_{\bcR} \cong 
(\bO_\bcB\vert_\bcR)^d \circ \Psi_\cR \circ \Psi_\cR^!.
\end{equation*}
Combining this isomorphism with the triangle 
\begin{equation*}
\Psi \circ \bO_\cB^d \circ \Psi^! \to \bO_\bcB^d \to \bT_{\Psi,\Psi^!} \circ \bal_\cD^d
\end{equation*}
from Proposition~\ref{proposition-rotation-twist-triangle} for~$i=d$,
restricting it to~$\bcR$, and composing it with~$(\bO_\bcB\vert_\bcR)^{-d}$ on the left, we obtain
\begin{equation*}
\Psi_\cR \circ \Psi_\cR^! \to \id_{\bcR} \to (\bO_\bcB\vert_\bcR)^{-d} \circ (\bT_{\Psi,\Psi^!} \circ \bal_\cD^d)\vert_\bcR,
\end{equation*}
which means that
\begin{equation*}
\bT_{\Psi_\cR,\Psi_\cR^!} \cong (\bO_\bcB\vert_\bcR)^{-d} \circ (\bT_{\Psi,\Psi^!} \circ \bal_\cD^d)\vert_\bcR.
\end{equation*}
It remains to note that the composition of the last two factors equals~$\bt_{\bcR}$ by definition, see~\eqref{eq:sigmax-rhox}.

If $d = m$ we have 
\begin{equation*}
(\bO_\bcB\vert_\bcR)^{-m} \circ \bt_\bcR = 
(\bT_{\Psi_\cB,\Psi_\cB^!} \circ \bal_\cD)^{-m} \circ \bT_{\Psi,\Psi^!} \circ \bal_\cD^m
\end{equation*}
and this is isomorphic to~$\bT_{\Psi_\cR,\Psi_\cR^!}$ by Corollary~\ref{cor:spherical-sod}.
\end{proof}

Using the criterion of Proposition~\ref{proposition:spherical-criterion}, we finally deduce 
from Propositions~\ref{prop:bt-cr} and~\ref{prop:bt-crx} the main result of this subsection. 

\begin{corollary}
\label{cor:psi-cr-spherical}
The functor~$\Psi_\cR$ is spherical.
\end{corollary}

\subsection{Proof of the main theorem and some corollaries}
\label{subsec:proof}

In this subsection we complete the proof of Theorem~\ref{thm:main-categorical}, 
deduce from it Theorem~\ref{thm:main} from the Introduction and a similar result for Gorenstein Deligne--Mumford stacks, 
and discuss how Theorem~\ref{thm:main-categorical} simplifies 
when the residual category~$\cR$ (or~$\cR_M$) is generated by a completely orthogonal exceptional sequence.

\begin{proof}[Proof of Theorem~\textup{\ref{thm:main-categorical}}]
Part~\eqref{item:main-categorical:bcb} is proved in Lemma~\ref{prop:bcb}.
Part~\eqref{item:main-categorical:bcb-compatibility} is a combination 
of Lemma~\ref{lemma:bcb-twist-compatible} and Lemma~\ref{lemma:serre-x-crx}.
Part~\eqref{item:main-categorical:psi-cr} is proved in Lemma~\ref{lemma:residual-psi} and Corollary~\ref{cor:psi-cr-spherical}.
So it remains to prove part~\eqref{item:main-categorical:serre}.

Using Lemma~\ref{lemma:serre-bcR} and Lemma~\ref{lemma:rotation-tau-commute} we obtain
\begin{equation*}
\bS_{\bcR}^{d/c} \cong ((\bO_\bcB\vert_\bcR)^{d-m} \circ \bs_{\bcR})^{d/c} \cong (\bO_\bcB\vert_\bcR)^{d(d-m)/c} \circ \bs_{\bcR}^{d/c}.
\end{equation*}
On the other hand, using Proposition~\ref{prop:bt-crx} and Lemma~\ref{lemma:rotation-tau-commute} we obtain
\begin{equation*}
\bT_{\Psi_\cR,\Psi_\cR^!}^{(d-m)/c} \cong (\bO_\bcB\vert_\bcR)^{-d(d-m)/c} \circ \bt_{\bcR}^{(d-m)/c}.
\end{equation*}
Multiplying these two isomorphisms and using commutativity of the functors~$\bO_\bcB\vert_\bcR$, $\bt_{\bcR}$, and~$\bs_{\bcR}$ 
(Lemma~\ref{lemma:rotation-tau-commute}), we deduce the second isomorphism in~\eqref{eq:serre-bcr-power}.

Similarly, using Lemma~\ref{lemma:serre-cr} and Proposition~\ref{prop:bt-cr}
instead of Lemma~\ref{lemma:serre-bcR} and Proposition~\ref{prop:bt-crx} 
we obtain the first isomorphism in~\eqref{eq:serre-bcr-power}.
This completes the proof of the theorem.
\end{proof}

Now we can deduce Theorem~\ref{thm:main} from the Introduction. 

\begin{proof}[Proof of Theorem~\textup{\ref{thm:main}}]
Just take
\begin{equation*}
\cC = \Db(M), 
\qquad 
\cD = \Db(X),
\qquad 
\bal_\cC = - \otimes \cL_M,
\qquad 
\bal_\cD = - \otimes \cL_X
\end{equation*}
and apply Theorem~\ref{thm:main-categorical}.
Note that~$\cR_M = \cR$ and~$\cR_X = \bcR$.
\end{proof}

Moreover, an extension of Theorem~\ref{thm:main} to Gorenstein Deligne--Mumford stacks also follows; 
we refer to \cite{nironi} for details on Grothendieck duality in the setting of stacks. 

\begin{corollary}
\label{cor:gorenstein}
Assume~$M$ and~$X$ are Gorenstein proper Deligne--Mumford stacks. 
Then the statement of Theorem~\textup{\ref{thm:main}} with bounded coherent categories~$\Db(M)$ and~$\Db(X)$ 
replaced by the perfect derived categories~$\Dp(M)$ and~$\Dp(X)$ holds true. 
\end{corollary}

\begin{proof}
This time we take
\begin{equation*}
\cC = \Dp(M), 
\qquad 
\cD = \Dp(X),
\qquad 
\bal_\cC = - \otimes \cL_M,
\qquad 
\bal_\cD = - \otimes \cL_X.
\end{equation*}
To apply Theorem~\ref{thm:main-categorical} we only need to observe that 
the perfect derived categories of Gorenstein Deligne--Mumford stacks~$M$ and~$X$ satisfy Serre duality
with~$\bS_M = - \otimes \omega_M[\dim M]$ and~\mbox{$\bS_X = - \otimes \omega_X[\dim X]$}.
\end{proof}

Note that Theorem~\ref{thm:main} is a special case of Corollary~\ref{cor:gorenstein}
since when~$M$ and~$X$ are smooth one has~$\Db(M) = \Dp(M)$ and~$\Db(X) = \Dp(X)$.

The following special case of Corollary~\ref{cor:gorenstein} 
where the functor~$\Psi$ is specialized to one of the functors listed in Lemmas~\ref{lemma:divisorial-embedding} and~\ref{lemma:double-covering} 
is often useful in examples, so we spell it out explicitly. 

\begin{corollary}
\label{corollary-divisor-double-cover}
Let~$M$ be an $n$-dimensional Gorenstein proper Deligne--Mumford stack. 
Let~$\cL_M$ be a line bundle on~$M$ such that~$\omega_M = \cL_M^{-m}$ for some integer~$m \geq 1$, 
and assume there is a semiorthogonal decomposition of the form 
\begin{equation*}
\Dp(M) = \langle \cR_M, \cB_M, \cB_M \otimes \cL_M, \dots, \cB_M \otimes \cL_M^{m-1} \rangle. 
\end{equation*} 
Let $f \colon X \to M$ be either the embedding of a divisor defined by a section of~$\cL_M^{d}$, 
or a flat double cover branched over a divisor defined by a section of~$\cL_M^{2d}$, where~$1 \leq d \leq m$. 
There is a subcategory~$\cR_X \subset \Dp(X)$ defined by the semiorthogonal decomposition
\begin{equation*}
\Dp(X) = 
\begin{cases}
\langle \cR_X, \cB_X, \cB_X \otimes \cL_X, \dots, \cB_X \otimes \cL_X^{m-d-1} \rangle , & \text{if $d < m$}\\
\hphantom{\langle}\cR_X, & \text{if $d = m$}
\end{cases}
\end{equation*} 
where in the first row 
$\cB_X \coloneqq f^*(\cB_M)$, $\cL_X = f^*\cL_M$, 
and the functor~$f^*\vert_{{\cB_M}}$ is fully faithful.
Furthermore, the functor~$\Psi \coloneqq f^* \colon \Dp(M) \to \Dp(X)$ 
restricts to a spherical functor~$\Psi_\cR \colon {\cR_M} \to \cR_X$. 
Finally, setting~$c = \gcd(d,m)$, we have:  
\begin{enumerate}
\item 
If~$f$ is a divisorial embedding, then 
\begin{equation*}
\bS_{{\cR_M}}^{d/c} \cong \bT_{\Psi_\cR^!,\Psi_\cR}^{m/c} \circ \left[\frac{dn}{c} \right]
\qquad\text{and}\qquad
\bS_{\cR_X}^{d/c} \cong \bT_{\Psi_\cR,\Psi_\cR^!}^{(m-d)/c} \circ \left[ \frac{d(n+1) - 2m}{c} \right] .  
\end{equation*} 
\item 
If~$f$ is a double cover, then 
\begin{equation*}
\bS_{{\cR_M}}^{d/c} \cong \bT_{\Psi_\cR^!,\Psi_\cR}^{m/c} \circ \left[ \frac{dn - m}{c} \right] 
\qquad\text{and}\qquad
\bS_{\cR_X}^{d/c} \cong \bT_{\Psi_\cR,\Psi_\cR^!}^{(m-d)/c} \circ \tau^{(m-d)/c} \circ 
\left[ \frac{d(n+1) - m}{c} \right]  , 
\end{equation*} 
where $\tau$ is the autoequivalence of~$\Dp(X)$ induced by the covering involution. 
Moreover, the functors~$\tau$ and~$\bT_{\Psi_{\cR},\Psi_{\cR}^!}$ commute.
\end{enumerate} 
\end{corollary} 

\begin{proof}
Note that by construction the pullback functor~$f^*$ intertwines 
between the tensor product functors~\mbox{$- \otimes \cL_M$} and~$- \otimes \cL_X$, 
and the subcategory $\cB_M \subset \Dp(M)$ is Serre compatible of length~$m$
as~\mbox{$\bS_{M} = - \otimes \cL_M^{-m} [n]$} by the assumption $\omega_M = \cL_M^{-m}$. 
Moreover, Lemmas~\ref{lemma:divisorial-embedding} and~\ref{lemma:double-covering} show 
that~$f^*$ is spherical and~\mbox{${\cB_M} \subset \Dp(M)$} is~$\bT_{f_*, f^*}$-twist compatible of degree~$d$. 
Therefore, we may apply Corollary~\ref{cor:gorenstein}. 

Note that $\bS_{X} = - \otimes \cL_{X}^{-m+d}[n -1]$ if $f$ is a divisorial embedding, 
and $\bS_{X} = - \otimes \cL_{X}^{-m+d}[n]$ if $f$ is a double covering 
because~$\omega_X \cong \cL_X^{-m+d}$ in both cases while~$\dim(X) = n-1$ for the embedding and~$\dim(X) = n$ for the covering.
Using this and Lemmas~\ref{lemma:divisorial-embedding} and~\ref{lemma:double-covering} again, 
we find that the autoequivalences  defined by~\eqref{eq:tau-cr}, \eqref{eq:sigma-cr}, and~\eqref{eq:sigmax-rhox}
are given by 
\begin{align*}
\bt_{{\cR_M}} &\cong \id_{{\cR_M}},
&
\bs_{{\cR_M}} &\cong [n],
&
\bt_{\cR_X} &\cong [2],
&
\bs_{\cR_X} &\cong [n - 1] 
\\
\intertext{if~$f$ is a divisorial embedding, and by}
\bt_{{\cR_M}} &\cong [1],
& 
\bs_{{\cR_M}} &\cong [n],
&
\bt_{\cR_X} &\cong \tau \circ [1],
& 
\bs_{\cR_X} &\cong [n ]  
\end{align*}
if~$f$ is a double covering. 
Combining these formulas with Corollary~\ref{cor:gorenstein} gives the result. 
\end{proof} 

Now we discuss how Theorem~\ref{thm:main-categorical} simplifies under additional assumptions on the residual category~$\cR$.
First, as we already mentioned in Remark~\ref{rem:cr-zero}, 
if~$\cR = 0$, the spherical twist~$\bT_{\Psi_\cR,\Psi_\cR^!}$ is isomorphic to identity, 
hence the formula~\eqref{eq:serre-bcr-power} reduces to
\begin{equation*}
\bS_{\bcR}^{d/c} \cong \bt_{\bcR}^{(d-m)/c} \circ \bs_{\bcR}^{d/c}
\end{equation*}
as in~\cite{K19}.
The next case is when~$\cR$ is generated by an exceptional object.
Recall the notion of a spherical object from~\S\ref{subsec:spherical-examples}.

\begin{corollary}
\label{cor:cr-exceptional}
Under assumptions of Theorem~\textup{\ref{thm:main-categorical}} 
assume that the residual category~$\cR$ is generated by an exceptional object~$\cE \in \cR$.
Then the object
\begin{equation*}
{\rP} = \Psi(\cE) \in \bcR
\end{equation*}
is spherical and
\begin{equation*}
\bS_{\bcR}^{d/c} \cong \bT_{{\rP}}^{(m-d)/c} \circ \bt_{\bcR}^{(d-m)/c} \circ \bs_{\bcR}^{d/c},
\end{equation*}
where~$\bT_{{\rP}}$ is the spherical twist with respect to~${\rP}$.
\end{corollary}

\begin{proof}
Follows from Lemma~\ref{lemma:spherical-collection}.
\end{proof}

A similar result holds when~$\cR$ is generated by a completely orthogonal exceptional collection 
(cf.~\cite[Theorem~2.6]{BKS}).
Recall the notion of a spherical collection from~\S\ref{subsec:spherical-examples}.

\begin{corollary}
\label{cor:cr-exceptional-collection}
Under assumptions of Theorem~\textup{\ref{thm:main-categorical}} assume that the residual category~$\cR$ 
is generated by a completely orthogonal exceptional collection~${\cE}_1,\dots, {\cE}_{r} \in \cR$.
Then the collection
\begin{equation*}
({\rP}_1, \dots, {\rP}_{r}) = (\Psi({\cE}_1), \dots, \Psi({\cE}_{r})) \in \bcR
\end{equation*}
is a spherical collection and
\begin{equation*}
\bS_{\bcR}^{d/c} \cong \bT_{{\rP}_1, \dots, {\rP}_{r}}^{(m-d)/c} \circ \bt_{\bcR}^{(d-m)/c} \circ \bs_{\bcR}^{d/c},
\end{equation*}
where~$\bT_{{\rP}_1, \dots, {\rP}_{r}}$ is the spherical twist 
with respect to the spherical collection~$({\rP}_1, \dots, {\rP}_{r})$.
\end{corollary}

The permutation~$\sigma \in \fS_r$ and the tuple~$(n_i)$ associated with the spherical collection~$({\rP}_i)$
are determined by the action of the rotation functor~$\bO_\cB\vert_\cR$ and autoequivalence~$\bt_\cR$ on the objects~$\cE_i$, 
see Proposition~\ref{prop:bt-cr}.

\begin{proof}
Follows from Lemma~\ref{lemma:spherical-collection}.
\end{proof}

We will see in the next section that these particular cases of Theorem~\textup{\ref{thm:main-categorical}} 
show up in geometrically interesting situations.

\section{Examples}
\label{sec:examples}

In this section we discuss some examples where the results proved in~\S\ref{sec:proof} apply. 
In~\S\ref{subsec:ci} we discuss the Serre functors of residual categories of Fano complete intersections in weighted projective spaces 
and Fano branched double covers of complete intersections in weighted projective spaces. 
In~\S\ref{subsec:rescatci-nontrivial} we specialize to the case of complete intersections in projective space, 
and prove nontriviality results for residual categories and the spherical functors between them; 
these results will be used in~\S\ref{sec:serre-dims}. 
In~\S\ref{subsec:ci-2-d} we further specialize to the case of a complete intersection of type~$(2,d)$ in~$\P^n$, 
and when~$d = n-2$ and~$n \geq 5$ is odd we give an example of refined residual categories.
Finally, in~\S\ref{subsec:simple-cr} we list some examples of varieties~$M$ with simple residual category~${\cR_M}$,  
in which case applying Theorem~\ref{thm:main} and other similar results 
gives more explicit consequences. 

\subsection{Weighted complete intersections}
\label{subsec:ci}

For a collection~$\rw = (w_0,w_1,\dots,w_n)$ of positive integers we denote by~$\P(\rw)$ the weighted projective space
\begin{equation*}
\P(\rw) \coloneqq ( \AA^{n+1} \setminus 0 ) /_\rw \Gm
\end{equation*}
where the right side is understood as the \emph{quotient stack} for the $\Gm$-action 
with weight $w_i$ on the $i$-th coordinate. 
Note that~$\P(\rw)$ is a Deligne--Mumford stack.
Note also that the usual projective space~$\P^n = \P((1,1,\dots,1))$ (with~$n+1$ entries in the right side) is a particular case of this general construction.

The derived category of~$\P(\rw)$ can be identified with the~$\Gm$-equivariant derived category of~$\AA^{n+1} \setminus 0$. 
In particular, for each $d \in \ZZ$ we have the invertible sheaf~$\cO_{\P(\rw)}(d)$ on~$\P(\rw)$ 
(associated with the equivariant invertible sheaf on~$\AA^{n+1} \setminus 0$ which is the trivial line bundle 
with equivariant structure given by the character~$t^d$ of~$\Gm$), and 
the dualizing sheaf is the line bundle 
\begin{equation*}
\omega_{\P(\rw)} \cong \cO_{\P(\rw)}(-|\rw|) \qquad \text{where} \qquad
|\rw| = \sum_{i=0}^n w_i. 
\end{equation*} 

Now let~$X = Y_1 \cap Y_2 \cap \dots \cap Y_k \subset \P(\rw)$ be a complete intersection of type~$(d_1,d_2,\dots,d_k)$, 
i.e.~$\dim X = n-k$ and each~$Y_i \subset \P(\rw)$ is a degree~$d_i$ hypersurface, 
such that the bound 
\begin{equation}
\label{fano-condition}
\sum_{i=1}^k d_i \leq |\rw| 
\end{equation}
holds. 
We do not assume $X$ to be smooth, but it is necessarily Gorenstein. 
The dualizing sheaf of $X$ is the line bundle 
\begin{equation}
\label{ci-index}
\omega_X \cong \cO_X(-\ind(X)), \qquad \text{where} \quad 
\ind(X) = |\rw| - \sum_{i=1}^{k} d_i
\end{equation}
is the {\sf Fano index} of~$X$.
Note that~$\ind(X) \geq 0$ by the assumption~\eqref{fano-condition}.
Moreover, there is a subcategory~$\cR_X \subset \Dp(X)$ defined by the semiorthogonal decomposition
\begin{equation}
\label{ci-RX}
\Dp(X) = 
\begin{cases}
\langle \cR_X, \cO_X, \dots, \cO_X(\ind(X) - 1) \rangle  , & \text{if~$\ind(X) > 0$}\\
\hphantom{\langle}\cR_X, & \text{if~$\ind(X) = 0$}
\end{cases}
\end{equation} 
where in the first case the structure sheaf~$\cO_X$ is exceptional.

Furthermore, let~$M = Y_1 \cap Y_2 \cap \dots \cap Y_{k-1} \subset \P(\rw)$  
be the complete intersection of type~\mbox{$(d_1, \dots, d_{k-1})$} obtained by omitting the $k$-th hypersurface defining~$X$. 
Note that for~$M$ the Fano condition~$\ind(M) > 0$ holds.
As for~$X$, we have a semiorthogonal decomposition 
\begin{equation}
\label{ci-M}
\Dp(M) = \langle \cR_{M}, \cO_{M}, \dots, \cO_{M}(\ind(M) - 1) \rangle . 
\end{equation} 
Applying Corollary~\ref{corollary-divisor-double-cover} in this situation 
with~$\cB_M = \langle \cO_M \rangle$, so that ${\cB_X} = \langle \cO_X \rangle$,
gives the following. 

\begin{corollary}
\label{cor:serre-ci-pn}
In the above setup, 
let $c = \gcd(d_k, \ind(M))$. 
If $i \colon X \to M$ denotes the embedding
and the residual categories~$\cR_M \subset \Dp(M)$ and~$\cR_X \subset \Dp(X)$ are defined by~\eqref{ci-M} and~\eqref{ci-RX},
then the functor~$\Psi_{\cR} = i^* \vert_{\cR_{M}} \colon \cR_{M} \to \cR_X$ is spherical and there are isomorphisms 
\begin{align*}
\bS_{\cR_{M}}^{d_k/c} &\cong 
\bT_{\Psi_{\cR}^!,\Psi_{\cR}}^{\,\ind(M)/c} \circ \left[ \frac{d_k \dim(M)}{c} \right],\\
\bS_{\cR_X}^{d_k/c} &\cong 
\bT_{\Psi_{\cR},\Psi_{\cR}^!}^{\,\ind(X)/c} \circ \left [ \frac{d_k \dim(X) - 2 \ind(X)}{c} \right ], 
\end{align*}
where explicitly 
\begin{equation*}
\dim(M) = n+1 - k,\quad 
\ind(M) = |\rw| - \sum_{i=1}^{k-1} d_i,\quad  
\dim(X) = n-k,
\quad\text{and}\quad 
\ind(X) = |\rw| - \sum_{i=1}^k d_i. 
\end{equation*}
\end{corollary}

\begin{remark}
Given a complete intersection~$X \subset \P(\rw)$, there are generally many different ways 
to write~$X = M \cap Y$ where~$M \subset \P(\rw)$ is a complete intersection of one dimension bigger 
and~$Y \subset \P(\rw)$ is a hypersurface. 
For each such presentation of $X$, Corollary~\ref{cor:serre-ci-pn} gives an expression for the Serre functor of~$\cR_X$ 
in terms of the associated spherical functor~\mbox{$\Psi_{\cR} \colon \cR_M \to \cR_X$} and the degree of~$Y$. 
This observation shows that we can express any power of~$\bS_{\cR_X}$ 
with exponent of the form~$\sum a_i\frac{d_i}{\gcd(d_i, m)}$, 
where $m = |\rw| - \sum_{j=1}^k d_j$ and $a_i \in \ZZ$, 
as a product of spherical twists and shifts.
\end{remark} 

Similarly, we obtain the following result for a double cover of a complete intersection. 

\begin{corollary}
\label{cor:serre-ci-dc} 
As before let~$M \subset \P(\rw)$ be a complete intersection of type $(d_1, \dots, d_{k-1})$, 
and let~\mbox{$f \colon X \to M$} be a flat double cover branched over a {Cartier} divisor of degree~$2d_k$ 
such that the bound~\eqref{fano-condition} holds. 
Then there is a semiorthogonal decomposition~\eqref{ci-RX} of $\Dp(X)$ with residual category~$\cR_X$. 
Setting 
$c = \gcd(d_k, \ind(M))$, the functor 
$\Psi_{\cR} = f^*\vert_{\cR_M} \colon \cR_M \to \cR_X$ is spherical
and there are isomorphisms 
\begin{align*}
\bS_{\cR_M}^{d_k/c} & \cong 
\bT_{\Psi_{\cR}^!,\Psi_{\cR}}^{\,\ind(M)/c} \circ \left[ \frac{d_k \dim(M) - \ind(M)}{c} \right] ,  \\ 
\bS_{\cR_X}^{d_k/c} & \cong 
\bT_{\Psi_{\cR},\Psi_{\cR}^!}^{\,\ind(X)/c} \circ \tau^{\ind(X)/c} \circ \left[ \frac{d_k \dim(X) - \ind(X)}{c} \right] ,
\end{align*}
where explicitly 
\begin{equation*}
\dim(M) = n + 1 - k,\quad 
\ind(M) = |\rw| - \sum_{i=1}^{k-1} d_i,\quad  
\dim(X) = n + 1 -k,
\quad
\ind(X) = |\rw| - \sum_{i=1}^k d_i, 
\end{equation*}
and~$\tau$ is the autoequivalence of~$\Dp(X)$ induced by the covering involution. 
Moreover, the functors~$\tau$ and~$\bT_{\Psi_{\cR},\Psi_{\cR}^!}$ commute.
\end{corollary} 

Note that $X$ as in Corollary~\ref{cor:serre-ci-dc} can be obtained 
by considering the cone $\widehat{M}$ over~$M$ in a bigger weighted projective space with the additional weight~$w_{n+1} = d_k$, 
and taking~$X$ to be a divisor of degree~$2d_k$ in this~$\widehat{M}$ (not passing through the vertex of the cone); 
thus Corollary~\ref{cor:serre-ci-pn} also applies in this situation, 
but the formula for the Serre functor given in Corollary~\ref{cor:serre-ci-dc} is often more precise. 
For example, if $X \to M = \P^3$ is a double cover branched over a quartic hypersurface, then 
Corollary~\ref{cor:serre-ci-dc} gives $\bS_{\cR_X} \cong \tau \circ [2]$  
whereas Corollary~\ref{cor:serre-ci-pn} gives $\bS_{\cR_X}^2 \cong [4]$. 

\subsection{Residual categories and functors for complete intersections} 
\label{subsec:rescatci-nontrivial} 

In this subsection, we specialize the above discussion to the case of complete intersections in projective space
(in particular, we pass from Deligne--Mumford stacks to varieties), 
and prove some conservativity and non-conservativity properties of the spherical functors between their residual categories. 
These results will be needed in \S\ref{sec:serre-dims} for our computation of Serre dimensions, 
but are also independently interesting, as they show the residual categories and their spherical functors are far from trivial. 
We fix the following setup. 

\begin{setup}
\label{setup-ci} 
Let~$X \subset \P^n$ be a Fano complete intersection of type~$(d_1, \dots, d_{k})$. 
Choose a presentation $X = M \cap Y \subset \P^n$, where~$M \subset \P^n$ is a complete intersection of type~$(d_1, \dots, d_{k-1})$ 
and~$Y \subset \P^n$ is a hypersurface of degree~$d_k$, so that 
\begin{align*}
\dim(X) &= n - k, &
\dim(M) &= n - k + 1, &
\dim(Y) &= n - 1, 
\\
\ind(X) &= n + 1 - \sum_{i=1}^k d_i, &
\ind(M) &= n + 1 - \sum_{i=1}^{k-1} d_i, &
\ind(Y) &= n + 1 - d_k, 
\end{align*}
and we have a fiber square 
\begin{equation}
\label{XYM-diagram} 
\vcenter{
\xymatrix{ 
X \ar[r]^{j} \ar[d]_{i} & Y \ar[d]^{i_Y} \\ 
M \ar[r]^{j_M} & \P^n .
} }
\end{equation} 
The residual categories $\cR_X$ and $\cR_M$ are defined by~\eqref{ci-RX} and~\eqref{ci-M}, and similarly $\cR_Y$ is defined by the semiorthogonal decomposition 
\begin{equation*} 
\Dp(Y)  = \langle \cR_Y, \cO_Y, \dots, \cO_Y(\ind(Y) - 1) \rangle .  
\end{equation*} 
For now we do not assume that any of the varieties~$X$, $M$, or~$Y$ is smooth.
\end{setup} 

\subsubsection{Nonvanishing of residual categories} 
If $d_i = 1$ for all of the degrees of the complete intersection~$X \subset \P^n$, i.e. $X \cong \P^{n-k}$, then the residual category $\cR_X$ 
vanishes, but nonvanishing holds as soon as some $d_i > 1$. 
We deduce this from the following observation. 

\begin{lemma}
\label{lemma:rank-k0}
Let~$Z \subset \P^n$ be a projective variety.
Then~$\rank(\rK_0(\Dp(Z))) \ge \dim({Z}) + 1$.
\end{lemma}

\begin{proof}
Let~$Z_p \subset Z$ be the complete intersection of~$Z$ with a general linear subspace of codimension~$p$.
We check that the classes in~$\rK_0(\Dp(Z))$ of the structure sheaves~$\cO_{Z_p}$, $0 \le p \le \dim({Z})$, are linearly independent.
Indeed, this follows from the fact that the Euler form looks like
\begin{equation*}
\chi(\cO_{Z_p},\cO_{Z_q}) \coloneqq 
\sum_s (-1)^s \dim \Ext^s(\cO_{Z_p}, \cO_{Z_q}) =
\begin{cases}
0 & \text{if $p + q > \dim(Z)$} , \\
\pm \deg(Z) & \text{if $p + q = \dim(Z)$, }
\end{cases}
\end{equation*}
hence the matrix of the bilinear form~$\chi$ on the set of classes~$[\cO_{Z_p}]$ is non-degenerate,
and so the classes are linearly independent.
\end{proof}

\begin{corollary}
\label{cor:RX-nonzero}
If $d_i > 1$ for some $i$, then $\cR_X \neq  0$. 
\end{corollary}

\begin{proof}
The semiorthogonal decomposition~\eqref{ci-RX} implies 
\begin{equation*}
\rank(\rK_0(\Dp(X))) = \rank(\rK_0(\cR_X)) + \ind(X). 
\end{equation*} 
Therefore, Lemma~\ref{lemma:rank-k0} gives 
\begin{equation*}
\rank(\rK_0(\cR_X)) \geq \dim(X) + 1 - \ind(X)  = 
\sum_{i=1}^k d_i - k = \sum_{i=1}^k (d_i - 1),
\end{equation*}
which is positive if $d_i > 1$ for some $i$. 
\end{proof}

\subsubsection{Ind-conservativity of $\Psi_{\cR}$}

In what follows we set~$\Psi \coloneqq i^* \colon \Dp(M) \to \Dp(X)$, so that~$\Psi_\cR = i^*\vert_{\cR_M}$.
In this paragraph we will show the functor $\Psi_{\cR} \colon \cR_M \to \cR_X$ is 
ind-conservative in the sense of Definition~\ref{definition-ind-conservative} (so in particular conservative). 
See Appendix~\ref{appendix-ind} for some reminders about ind-completions of categories; 
in particular, by the discussion there, it follows that the ind-completions~$\Ind(\cR_M)$ and~$\Ind(\cR_X)$ 
are equal to the orthogonals
\begin{alignat*}{2}
\Ind(\cR_X)  & = \langle \cO_X, \dots, \cO_X(\ind(X) - 1) \rangle^\perp  && \subset \Dqc(X), \\ 
\Ind(\cR_M)  & = \langle \cO_M, \dots, \cO_M(\ind(M) - 1) \rangle^\perp  && \subset \Dqc(M),
\end{alignat*} 
in the unbounded derived categories of quasicoherent sheaves on~$X$ and~$M$, respectively. 

The key observation for ind-conservativity of $\Psi_{\cR}$ is the following. 
\begin{lemma}
\label{lemma-pullback-kernel}
Let $Z$ be a projective variety and let $i \colon D \hookrightarrow Z$ be the inclusion of an ample divisor. 
Then the pullback functor~$i^* \colon \Dqc(Z) \to \Dqc(D)$ has trivial kernel on the subcategory 
\begin{equation*} 
\langle \cO_Z \rangle^{\perp} = \{ E \in \Dqc(Z) \mid \Ext^\bullet(\cO_Z, E) = 0 \} \subset \Dqc(Z), 
\end{equation*} 
i.e. $\ker(i^*\vert_{\langle \cO_Z \rangle^{\perp}} \colon \langle \cO_Z \rangle^{\perp} \to \Dqc(D)) = 0$. 
\end{lemma} 

\begin{proof}
Let~$u \colon U \hookrightarrow Z$
denote the inclusion of the open complement of~$D \subset Z$. 
If~$E \in \Dqc(Z)$ is such that~$i^*E = 0$, then the canonical map~$E \to u_*u^*E$ is an isomorphism. 
Indeed, the claim is local, so we may assume~$Z = \Spec(A)$ and~$D$ is defined by a function~$f \in A$. 
As~$i^*E = 0$, tensoring the sequence 
\begin{equation*}
0 \xrightarrow \quad A \xrightarrow{ \ f \ } A \xrightarrow\quad A/(f) \xrightarrow\quad 0
\end{equation*}
with~$E$ shows the multiplication map~$E \xrightarrow{\ f \ } E$ is an isomorphism, and hence the map
\begin{equation*}
E \to u_*u^*E = E \otimes_{A} A_f 
\end{equation*}
is an isomorphism. 
If further~$E \in \langle \cO_Z \rangle^{\perp}$, then by adjunction we have~$u^*E \in \langle \cO_U \rangle^{\perp}$. 
But the open~$U \subset Z$ is affine because~$D \subset Z$ is ample, so this implies~$u^*E = 0$. 
\end{proof} 

\begin{proposition}
\label{proposition-ker-Psi-Ind}
The functor $\Psi_{\cR} \colon \cR_M \to \cR_X$ is ind-conservative. 
\end{proposition} 

\begin{proof}
Follows from Lemma~\ref{lemma-pullback-kernel}, because~$\Ind(\cR_M) \subset \langle \cO_M \rangle^{\perp} \subset \Dqc(M)$ by definition.  
\end{proof} 

\subsubsection{Non-conservativity of $\Psi_{\cR}^!$}

Finally, in this paragraph we strengthen Corollary~\ref{cor:RX-nonzero} by showing that the adjoint 
functor~$\Psi_{\cR}^! \colon \cR_X \to \cR_M$ is not conservative. 
The key observation is that nonzero objects of the residual component~$\cR_Y$ of the hypersurface~$Y$ 
restrict to nonzero objects of~$\ker(\Psi_{\cR}^!)$. 

\begin{lemma}
\label{lemma-Y-to-X}
Pullback along the inclusion~$j \colon X \to Y$ induces a conservative functor 
\begin{equation*} 
j^*\vert_{\cR_Y} \colon \cR_Y \to \cR_{X}
\end{equation*}  
whose image is contained in~$\ker(\Psi^!_{\cR} \colon \cR_X \to \cR_M)$. 
\end{lemma} 

\begin{proof}
Recall from diagram~\eqref{XYM-diagram} our notation $i_Y \colon Y \to \P^n$ and $j_M \colon M \to \P^n$ for the embeddings. 
Note that we have
\begin{alignat*}{4}
 \cR_Y  & = \{ E \in \Dp(Y) && \mid i_{Y*}E  & & \in \langle \cO_{\P^n}(-d_k), \dots, \cO_{\P^n}(-1) \rangle \} & \quad & 
 \text{by Lemma~\ref{lemma-bcR-characterization}, and } \\ 
\ker(\Psi_{\cR}^!) & = \{ E \in \Dp(X) && \mid i_*E & & \in \langle \cO_{M}(-d_k), \dots, \cO_M(-1) \rangle \} & &  
\text{by Lemma~\ref{lemma-ker-Psi!}}, 
\end{alignat*} 
so by base change~$i_*j^*E \cong j_M^*i_{Y*}E$ (which holds because the square~\eqref{XYM-diagram} is $\Tor$-independent)
it follows that~$j^* \vert_{\cR_Y}$ indeed has image contained in~$\ker(\Psi^!_{\cR})$.

To show that~$j^*\vert_{{\cR_Y}}$ is conservative, we factor the embedding~$j \colon X \hookrightarrow Y$ as a sequence of hypersurface sections
and apply Lemma~\ref{lemma-pullback-kernel} to each of these. 
\end{proof} 

\begin{proposition}
\label{proposition-kerPsiM!-nonzero}
In Setup~\textup{\ref{setup-ci}}, if~$d_k > 1$ then~$\ker(\Psi_{\cR}^!) \neq 0$. 
\end{proposition} 

\begin{proof}
Corollary~\ref{cor:RX-nonzero} applied to $Y$ shows that $\cR_Y \neq 0$, 
so the claim follows from Lemma~\ref{lemma-Y-to-X}. 
\end{proof} 
  
\begin{remark}
Conservativity of the functor~$\Psi_\cR$ and non-conservativity of its adjoint~$\Psi_\cR^!$ show
that in some sense the category~$\cR_X$ has ``more objects'' than~$\cR_M$, and hence it is ``more complicated''.
Thus, complexity of the residual category of a complete intersection grows with codimension.
\end{remark}

\subsection{Fano divisor in a quadric}
\label{subsec:ci-2-d}

As we explained in~\S\ref{subsection-intro-refined-categories} in some cases one can further decompose the residual category.
Since in many cases the residual category comes with a spherical functor from a simpler residual category, 
the following general result is useful for this purpose. 

\begin{lemma}
\label{lemma-refined-residual} 
Let~$\Psi \colon \cC \to \cD$ be a spherical functor. 
If~$\eps \colon \cE \hookrightarrow \cC$ is the inclusion of an admissible subcategory, 
then the composition~$\Psi \circ \eps \colon \cE \to \cD$ is fully faithful with admissible image 
if and only if~$\eps^! \circ \bT_{\Psi^!, \Psi} \circ \eps = 0$, if and only if~$\cE \subset {^\perp}\bT_{\Psi^!, \Psi}(\cE)$. 
\end{lemma}

\begin{proof}
Note that $\Psi$ has both adjoints, as this is part of our definition of a spherical functor. 
The functor~$\Psi \circ \eps$ thus has both adjoints because~$\Psi$ and~$\eps$ do, 
so it is enough to determine when the composition~$\eps^! \circ \Psi^! \circ \Psi \circ \eps$ is isomorphic to the identity.
Composing the defining triangle~\eqref{eq:twists-triangles} with~$\eps^!$ and~$\eps$ we obtain
\begin{equation*}
\eps^! \circ \bT_{\Psi^!,\Psi} \circ \eps \xrightarrow{\qquad} 
\eps^! \circ \eps \xrightarrow{\ \eps^! \circ \eta_{\Psi^!,\Psi} \circ \eps\ } 
\eps^! \circ \Psi^! \circ \Psi \circ \eps.
\end{equation*}
The middle term is the identity and the second arrow is the unit of adjunction for~$\Psi \circ \eps$;
it is an isomorphism if and only if the first term is zero.
This proves the first statement.
The second statement follows easily from adjunction.
\end{proof}

For example, if~$\Psi_{\cR} \colon \cR \to \bcR$ is a spherical functor between residual categories, 
then Lemma~\ref{lemma-refined-residual} gives a criterion for producing an admissible subcategory in~$\bcR$ from one in~$\cR$, 
which can then be split off to define a refined residual category. 
In this subsection we consider the case where~$\cR$ is the residual category of a smooth quadric hypersurface; 
consequently, we assume the characteristic of the base field~$\kk$ is not equal to~2.

Consider a smooth Fano divisor~$X \subset Q \subset \P^n$ of degree~$d$ in a quadric hypersurface~$Q \subset \P^n$.
In other words, $X$ is a complete intersection of type~$(2,d)$ and $d \le n - 2$.
This fits into the setup of~\S\ref{subsec:ci} with~$M = Q$.
So, we take $\cB_M = \langle \cO_Q \rangle$ and consider the corresponding residual categories in~$Q$ and~$X$.
The first is easy to describe explicitly.
Note that the assumption of smoothness of~$X$ implies that either~$Q$ is non-degenerate or has corank~$1$.
For simplicity we focus on the case where~$Q$ is non-degenerate, i.e. smooth; 
see Remark~\ref{remark-corank-1} below for the corank~$1$ case.

\begin{lemma}[{\cite{K08}}]
\label{lemma:residual-quadric}
Assume the characteristic of the base field $\kk$ is not equal to~$2$. 
If~$Q \subset \P^n$ is a smooth quadric hypersurface
and the subcategory~$\cR_{Q} \subset \Db(Q)$ is defined by the semiorthogonal decomposition
\begin{equation}
\label{eq:sod-quadric}
\Db(Q) = 
\langle \cR_{Q}, \cO_Q, \cO_Q(1), \dots, \cO_Q(n-2) \rangle , 
\end{equation}
then~$\cR_{Q} \cong \Db(\Cl_0(q))$ where~$\Cl_0(q)$ is the even part of the Clifford algebra of the corresponding quadratic form~$q$. 
Moreover, if the base field~$\kk$ is algebraically closed, 
there is a Morita equivalence
\begin{equation*}
\Cl_0(q) \sim 
\begin{cases}
\kk & \text{if $n$ is even},\\
\kk \oplus \kk & \text{if $n$ is odd}.
\end{cases}
\end{equation*}
\end{lemma}

The last part of the lemma means that, if the base field is algebraically closed, 
depending on parity of~$n$ the category~$\cR_{Q}$ 
is generated by one or two completely orthogonal exceptional objects 
--- the corresponding objects in~$\Db(Q)$ are the spinor bundle~$\cS$ or the spinor bundles~$\cS_+$ and~$\cS_-$, see~\cite{Ott} --- 
hence Corollary~\ref{cor:cr-exceptional} or~\ref{cor:cr-exceptional-collection} applies.
Recall the notion of a~$\sigma$-exceptional pair (collection of length~2) from~\S\ref{subsec:spherical-examples}
and of the corresponding spherical twist.
We denote by~$\sigma_0 \in \fS_2$ the transposition.

\begin{corollary}
\label{cor:quadric-divisor}
If the base field~$\kk$ is algebraically closed of characteristic not equal to~$2$ 
and~$X \subset Q \subset \P^n$ is a smooth divisor of degree~$d \le n - 2$ in a smooth quadric hypersurface 
then there is a semiorthogonal decomposition
\begin{equation*}
\Db(X) = \langle \cR_X, \cO_X, \dots, \cO_X(n-d-2) \rangle, 
\end{equation*}
and if~$c = \gcd(d,n-1)$ then
\begin{itemize}
\item 
if $n$ is even, the object~$\cS\vert_X$ is a spherical object in~$\cR_X$ and
\begin{equation*}
\bS_{\cR_X}^{d/c} \cong \bT_{\cS\vert_X}^{(n-1-d)/c} \circ \left[ \frac{(d-2)n + 2}{c} \right ]; 
\end{equation*}
\item 
if $n$ is odd, the objects~$\cS_+\vert_X$ and~$\cS_-\vert_X$ form a $\sigma_0^d$-spherical pair in~$\cR_X$ and
\begin{equation*}
\bS_{\cR_X}^{d/c} \cong \bT_{\cS_+\vert_X, \cS_-\vert_X}^{(n-1-d)/c} \circ \left[ \frac{(d-2)n + 2}{c} \right ] .
\end{equation*}
\end{itemize}
\end{corollary}

\begin{remark}
\label{remark-corank-1}
A similar result holds if~$Q$ has corank~1.
In this case we still have~\eqref{eq:sod-quadric} and an equivalence~$\cR_{Q} \cong \Db(\Cl_0(q))$,
but this time the description of~$\Cl_0(q)$ is more complicated:
\begin{equation*}
\Cl_0(q) \sim 
\begin{cases}
\big(\xymatrix@1{\kk \ar@/^/[r]^{\eps_1} & \kk \ar@/^/[l]^{\eps_2}}\big) / (\eps_1\eps_2 = 0 = \eps_2\eps_1), 
& \text{if $n$ is even},\\
\kk[\eps]/\eps^2, 
& \text{if $n$ is odd},
\end{cases}
\end{equation*}
where in the former case we consider the path algebra of a quiver with relations.
A similar semiorthogonal decomposition and equivalence hold for~$\Dp$ instead of~$\Db$.
Using these one can obtain a result similar to Corollary~\ref{cor:quadric-divisor} in this case. 
\end{remark}

Below we discuss some special cases of Corollary~\ref{cor:quadric-divisor}.
The first two are trivial but instructive.

\begin{remark}
Assume $d = 1$, so that~$c = 1$ and~$X$ is itself a smooth quadric hypersurface (in a hyperplane of~$\P^n$).
Then the residual category~$\cR_X$ also can be described by Lemma~\ref{lemma:residual-quadric}.
Thus, if~$n$ is even then~$\cR_X$ is generated by the two spinor bundles~$\cS'_\pm$;
moreover, in this case~$\cS\vert_X \cong \cS'_+ \oplus \cS'_-$, 
and the corresponding spherical twist of~$\cR_X$ is the composition of the transposition with the shift~$[1]$.
Similarly, if~$n$ is odd then~$\cR_X$ is generated by the single spinor bundle~$\cS'$;
moreover, in this case~$\cS_+\vert_X \cong \cS_-\vert_X \cong \cS'$, 
and the corresponding spherical twist is the shift~$[1]$.
In both cases the formula of Corollary~\ref{cor:quadric-divisor} gives an isomorphism~$\bS_{\cR_X} \cong \id$.
\end{remark}

\begin{remark}
Assume $d = 2$, so that~$X$ is an intersection of two quadrics.
Thus, if~$n$ is even then~$\cR_X$ is equivalent to the derived category of a square root stack~$C$ over~$\P^1$, 
the object~\mbox{$\cS\vert_X \in \cR_X$} corresponds to the structure sheaf of a (non-stacky) point~$c \in C$,
and the corresponding spherical twist is isomorphic to the twist by~$\cO_C(c)$; 
hence the formula of Corollary~\ref{cor:quadric-divisor} gives an isomorphism~$\bS_{\cR_X}^2 \cong {- \otimes {}}\cO_C((n-3)c)[2]$.
Similarly, if~$n$ is odd then~$\cR_X$ is equivalent to the derived category of a hyperelliptic curve $\pi \colon C \to \P^1$,
the objects~\mbox{$\cS_\pm\vert_X \in \cR_X$} corresponds to the structure sheaves of two points~$c_\pm \in C$ 
over a non-branching point~$c_0 \in \P^1$,
and the corresponding spherical twist is isomorphic to the tensor product by~$\cO_C(c_+ + c_-) \cong \pi^*\cO(1)$; 
hence the formula of Corollary~\ref{cor:quadric-divisor} gives an isomorphism~$\bS_{\cR_X} \cong - \otimes  \pi^*\cO((n-3)/2)[1]$.
\end{remark}

Finally, we consider a special case where the semiorthogonal decomposition of~$\Db(X)$ can be refined, 
and we describe the Serre functor of the refined component. 

\begin{proposition}
\label{proposition-2-3-intersection}
Assume that the base field~$\kk$ is algebraically closed of characteristic not equal to~$2$, 
$n \ge 5$ is odd, and \mbox{$X \subset Q \subset \P^n$} is a smooth divisor of degree~$d = n - 2$ 
in a smooth quadric hypersurface~$Q$.
Then there is a semiorthogonal decomposition
\begin{equation}
\label{equation-sod-AX}
\Db(X) = \langle \cA_X, \cS_+\vert_X, \cO_X \rangle.
\end{equation}
Moreover, $\Ext^\bullet(\cS_+\vert_X, \cS_-\vert_X) \cong \kk[3-n]$ and if the object~$\cK \in \cA_X$ is defined by the triangle
\begin{equation}
\label{eq:ck}
\cK \to \cS_+\vert_X \to \cS_-\vert_X[n-3]
\end{equation}
with non-trivial second arrow, then~$\cK$ is~$(2n-7)$-spherical and
\begin{equation*}
\bS_{\cA_X}^{n-2} \cong \bT_\cK^{(3-n)/2}  \circ  [(n-2)^2 - 2].
\end{equation*}
\end{proposition}

\begin{remark}
Using the other spinor bundle gives another semiorthogonal decomposition 
\begin{equation*}
\Db(X) = \langle \cA'_X, \cS_{-}\vert_X, \cO_X \rangle, 
\end{equation*} 
and an analogous result holds for~$\cA'_{X}$. 
In fact, there is an equivalence~$\cA'_X \simeq \cA_X$, which can be proved by an argument 
analogous to the one used in~\cite[Proposition~3.15]{K21}.
\end{remark}

\begin{proof}
We have an exact sequence
\begin{equation}
\label{eq:spinor-restricted}
0 \to \cS_\pm(-d) \to \cS_\pm \to i_*(\cS_\pm\vert_X) \to 0.
\end{equation}
Combining it with the exact sequences~\cite[Theorem~2.8(ii)]{Ott}
\begin{equation}
\label{eq:spinor-sequences}
0 \to \cS_\mp(-i) \to \cO_Q(-i)^{\oplus N} \to \cS_\pm(1-i) \to 0,
\end{equation}
where $N = 2^{(n-1)/2}$, semiorthogonality of~\eqref{eq:sod-quadric} and complete orthogonality of~$\cS_+$ and~$\cS_-$,
we conclude that each of the bundles~$\cS_+\vert_X$ and~$\cS_-\vert_X$ is exceptional. 
(Alternatively, the exceptionality of $\cS_+\vert_X$ and~$\cS_-\vert_X$ can be checked using Lemma~\ref{lemma-refined-residual}.)
In particular, this gives the claimed semiorthogonal decomposition~\eqref{equation-sod-AX}. 

A similar argument 
allows us to compute~$\Ext^\bullet(\cS_\pm\vert_X, \cS_\mp\vert_X)$.
Then, using the defining triangle of~$\cK$ it is easy to check that
\begin{equation*}
\Ext^\bullet(\cK,\cK) \cong \kk \oplus \kk[7 - 2n].
\end{equation*}
Furthermore, the semiorthogonal decomposition of~$\Db(X)$ 
combined with Lemma~\ref{lemma:serre-subcategory} and the isomorphism~$\omega_X \cong \cO_X(-1)$ 
implies that~for~$\cF \in \cA_X \subset \cR_X$ we have
\begin{equation*}
\bS_{\cA_X}^{-1}(\cF) \cong \bL_{\cS_+\vert_X} (\bS_{\cR_X}^{-1}(\cF))
\qquad\text{and}\qquad
\bS_{\cR_X}^{-1}(\cF) \cong \bL_{\cO_X} (\cF \otimes \cO_X(1)[2-n]).
\end{equation*}
Now using the exact sequences~\eqref{eq:spinor-sequences} it is easy to see that
\begin{equation}
\label{eq:serre-crx-spinors}
\bS_{\cR_X}^{-1}(\cS_\pm\vert_X) \cong \bL_{\cO_X} (\cS_\pm(1)\vert_X[2-n]) \cong \cS_\mp\vert_X[3-n]
\end{equation} 
and hence
\begin{equation*}
\bS_{\cR_X}^{-1}(\cK) \cong \cK'[3-n],
\end{equation*}
where~$\cK' \in \cR_X$ is defined by the triangle
\begin{equation*}
\cK' \to \cS_-\vert_X \to \cS_+\vert_X[n-3]
\end{equation*}
analogous to~\eqref{eq:ck}.
Since~$\bL_{\cS_+\vert_X}$ kills~$\cS_+\vert_X$, it follows that
\begin{equation*}
\bL_{\cS_+\vert_X}(\cK') \cong 
\bL_{\cS_+\vert_X}(\cS_-\vert_X) \cong 
\Cone(\cS_+\vert_X[3-n] \to \cS_-\vert_X) \cong \cK[4-n].
\end{equation*}
Combining these isomorphisms we deduce that~$\bS_{\cA_X}^{-1}(\cK) \cong \cK[7-2n]$, hence~$\cK$ is $(2n-7)$-spherical.

Now we are ready to prove the formula for~$\bS_{\cA_X}^{n - 2}$.
We have
\begin{align*}
\bS_{\cA_X}^{2-n} 
&= (\bL_{\cS_+\vert_X} \circ \bS_{\cR_X}^{-1})^{n-2} 
\\
& \cong \bL_{\cS_+\vert_X} \circ \bL_{\bS_{\cR_X}^{-1}(\cS_+\vert_X)} \circ \dots 
\circ \bL_{\bS_{\cR_X}^{3-n}(\cS_+\vert_X)} \circ \bS_{\cR_X}^{2-n} 
\\
& \cong \bL_{\cS_+\vert_X} \circ \bL_{\cS_-\vert_X} \circ \dots \circ \bL_{\cS_+\vert_X} 
\circ \bT_{\cS_+\vert_X,\cS_-\vert_X}^{-1} \circ  [-2 - (n-4)n]
\\
& \cong (\bL_{\cS_+\vert_X} \circ \bL_{\cS_-\vert_X})^{(n-3)/2} \circ \bL_{\cS_+\vert_X} 
\circ \bT_{\cS_+\vert_X,\cS_-\vert_X}^{-1}  \circ  [-2 - (n-4)n]; 
\end{align*}
indeed, the first isomorphism follows from the fact that the conjugation of a mutation functor by an autoequivalence is a mutation functor,
the second isomorphism follows from~\eqref{eq:serre-crx-spinors} and~Corollary~\ref{cor:quadric-divisor},
and the last isomorphism is obtained just by gathering some of the factors in pairs.

Now let~$\cF \in \cA_X$
and set~$V = \Ext^\bullet(\cS_-\vert_X,\cF)$. 
Recall from Lemma~\ref{lemma:spherical-collection} that~$\bT_{\cS_+\vert_X,\cS_-\vert_X}^{-1}(\cF)$ can be written as
\begin{align*}
&
\Cone\Big(\cF \to \Ext^\bullet(\cF,\cS_+\vert_X)^\vee \otimes \cS_+\vert_X \oplus \Ext^\bullet(\cF,\cS_-\vert_X)^\vee \otimes \cS_-\vert_X\Big)[-1]
\\
\cong& \Cone\Big(\cF \to \Ext^\bullet(\cS_-\vert_X,\cF) \otimes \cS_+\vert_X[n-3] \oplus \Ext^\bullet(\cS_+\vert_X, \cF) \otimes \cS_-\vert_X[n-3]\Big)[-1]
\\
\cong& \Cone\Big(\cF \to V \otimes \cS_+\vert_X[n-3] \Big)[-1],
\end{align*}
where in the second line we used~\eqref{eq:serre-crx-spinors} and Serre duality in~$\cR_X$,
and in the last line we used the containment~$\cF \in \cA_X$ and the definition of~$V$.
Furthermore, $\bL_{\cS_+\vert_X}$ kills~$\cS_+\vert_X$ and does not change~$\cF$, hence
\begin{equation*}
\bL_{\cS_+\vert_X} (\bT_{\cS_+\vert_X,\cS_-\vert_X}^{-1}(\cF)) \cong \cF.
\end{equation*}
So, to complete the proof of the proposition it remains to show that
\begin{equation*}
(\bL_{\cS_+\vert_X} \circ \bL_{\cS_-\vert_X})\vert_{\cA_X} \cong \bT_\cK.
\end{equation*}
For this, note that
\begin{equation*}
\bL_{\cS_-\vert_X}(\cF) \cong \Cone(V \otimes \cS_-\vert_X \to \cF),
\end{equation*}
while
\begin{equation*}
\bL_{\cS_+\vert_X}(\cS_-\vert_X) \cong \cK[4 - n]
\qquad\text{and}\qquad
\bL_{\cS_+\vert_X}(\cF) \cong \cF,
\end{equation*}
hence 
\begin{equation*}
\bL_{\cS_+\vert_X}(\bL_{\cS_-\vert_X}(\cF)) \cong \Cone(V \otimes \cK[4-n] \to \cF). 
\end{equation*}
Finally, observe from the defining triangle of~$\cK$ that~$\Ext^\bullet(\cK,\cF) \cong V[4-n]$, 
so that the right side above is identified with~$\bT_\cK(\cF)$. 
\end{proof}

\subsection{Other examples}
\label{subsec:simple-cr}

As we saw above, for complete intersections of type $(2,d)$ the 
residual category has a particularly simple Serre functor, ultimately 
due to the simple nature of the residual category of a quadric. 
There are a number of other varieties~$M$ with relatively ``simple'' residual category~$\cR_M$: 
in~\S\ref{sssec:gr}--\ref{sssec:sporadic} we list examples of~$M$ where~$\cR_M$ is generated by a completely orthogonal exceptional collection
(see~\cite[Conjecture~1.12]{KS20} which predicts when this occurs in terms of the small quantum cohomology ring of~$M$),
and in~\S\ref{sssec:dynkin-an} and~\S\ref{sssec:dynkin-other} we list examples where~$\cR_M$ is the derived category of representations of a Dynkin quiver
(see~\cite[Conjecture~1.3]{KS21} for a related conjecture).
In all these cases one can take~$f \colon X \to M$ to be a divisorial embedding of degree~$1 \le d \le m$ 
or a double covering with branch divisor of degree~$2d$ where again~$1 \le d \le m$,
and obtain a ``simple'' formula for~$\bS_{\cR_X}$. 

\subsubsection{Grassmannians}
\label{sssec:gr}

Let~$M = \Gr(k,m)$. 
In this case the required semiorthogonal decomposition is constructed in~\cite[Theorem~4.3]{Fon}.
If~$\gcd(k,m) = 1$, the residual category~{$\cR_M$} vanishes~\cite[Proposition~4.8]{Fon} 
and if~$\gcd(k,m) > 1$,
it is expected that~{$\cR_M$} is generated by a completely orthogonal exceptional collection~\cite[Conjecture~3.10]{KS20}; 
this conjecture has been proved in the case~$(k,m) = (p,pr)$ for prime~$p$ modulo another conjecture, 
and for~$p \in \{2,3\}$ unconditionally, see~\cite[Theorem~9.5]{CKMPS} and~\cite[Theorem~3.13 and Proposition~A.1]{KS20}.

\subsubsection{Products of projective spaces}
\label{sssec:ppp}

Let~$M = \P^{m-1} \times \P^{m-1}$. 
Of course, in this case there is a rectangular Lefschetz decomposition, but if we impose an extra constraint,
restricting ourselves to semiorthogonal decompositions which are \emph{invariant} under the~$\fS_2$-action on~$M$ by transposition of factors
(such decompositions descend to the~$\fS_2$-equivariant category), the situation becomes more interesting. 

An~$\fS_2$-invariant Lefschetz decomposition of~$\Db(M)$ has been constructed in~\cite{Ren} (see also~\cite[Theorem~3.2]{Mir}).
If~$m$ is odd the residual category~$\cR_M$ of this decomposition is zero and if~$m$ is even 
it is generated by a completely orthogonal exceptional collection of length~$2m$, see~\cite[Example~1.4]{KS21}.

One can also find in~\cite{Mir} a more general construction
of~$\fS_k$-invariant Lefschetz decompositions for powers of projective spaces. 
For instance, for~$(\P^1)^k$ there is such a decomposition, where for~$k$ odd the residual category is zero 
and for~$k$ even it is generated by a completely orthogonal exceptional collection \cite[Theorem~4.1]{Mir}. 
For~$(\P^{m-1})^3$ there is also such a decomposition, where for~$\gcd(3,m) = 1$ the residual category is zero 
and for~$m = 3$ (and presumably for any~$m$ divisible by~$3$) it is generated 
by a completely orthogonal exceptional collection~\cite[Theorems~4.2 and~4.6]{Mir}. 

\subsubsection{Sporadic cases}
\label{sssec:sporadic}

Some other homogeneous varieties of simple algebraic groups are known to have Serre compatible Lefschetz collections 
with residual categories generated by completely orthogonal exceptional collections:
\begin{enumerate}
\item 
\label{item:igr3-8}
$\IGr(3,8)$, see~\cite[Theorem~8.6]{Gus};
\item 
\label{item:e6-p1}
$\rE_6/\rP_1$, see~\cite{FM} and~\cite[Theorem~3.9]{BKS}.
\end{enumerate}

\subsubsection{Dynkin quivers of type~$\rA$}
\label{sssec:dynkin-an}

There are several examples of varieties~$M$ with a Serre compatible Lefschetz decomposition 
such that its residual category is equivalent to the derived category of representations of a Dynkin quiver of type~$\rA$.
Examples of this sort can be obtained by~\cite[Theorem~2.6]{BKS} from some of the examples of the previous paragraphs
as smooth hyperplane sections.

\begin{enumerate}
\item 
\label{item:igr2-2k}
$\IGr(2,2k)$, a smooth hyperplane section of~$\Gr(2,2k)$ has ${\cR_M} \cong \Db(\rA_{k-1})$, \cite[Theorem~9.6]{CKMPS};
\item 
\label{item:fl-1-2k-1-2k}
$\Fl(1,2k-1;2k)$, a smooth hyperplane section of~$\P^{2k-1} \times \P^{2k-1}$ has ${\cR_M} \cong \Db(\rA_{2k-1})$, \cite[Theorem~2.1]{KS21};
\item 
\label{item:f4-p4}
$\rF_4/\rP_4$, a smooth hyperplane section of~$\rE_6/\rP_1$ has ${\cR_M} \cong \Db(\rA_2)$, \cite[Theorem~1.4]{BKS}.
\end{enumerate}

\subsubsection{Other Dynkin quivers}
\label{sssec:dynkin-other}

It is expected (\cite[Conjecture~1.8]{KS21}) that for any simple algebraic group~$\rG$ (except for Dynkin type~$\rA_{2k}$) 
the \emph{coadjoint} homogeneous variety has a Serre compatible Lefschetz decomposition 
such that its residual category is equivalent to the derived category of representations of the Dynkin quiver 
of the subdiagram of the Dynkin diagram of~$\rG$ corresponding to short roots.
For type~$\rA_{2k-1}$ this was already listed in~\S\ref{sssec:dynkin-an}\eqref{item:fl-1-2k-1-2k},
for types~$\rB_k$ and~$\rG_2$ in Lemma~\ref{lemma:residual-quadric} (where the coadjoint variety is a smooth quadric of odd dimension),
for type~$\rC_k$ in~\S\ref{sssec:dynkin-an}\eqref{item:igr2-2k},
and for type~$\rF_4$ in~\S\ref{sssec:dynkin-an}\eqref{item:f4-p4}.
For type~$\rD_k$ the conjecture has been proved in~\cite[Theorem~3.1]{KS21}:
\begin{enumerate}
\item 
$\OGr(2,2k)$, the isotropic Grassmannian of $2$-dimensional subspaces has ${\cR_M} \cong \Db(\rD_k)$. 
\end{enumerate}
For the remaining types~$\rE_6$, $\rE_7$, and~$\rE_8$ the conjecture is still open.

\section{Serre dimensions} 
\label{sec:serre-dims}

The main goal of this section is to prove Theorem~\ref{theorem-ci-sdim} on the Serre dimensions of residual categories of complete intersections. 
In~\S\ref{subsection-dimensions} we recall the notions of upper and lower dimension of an endofunctor, 
and prove a result that relates these dimensions for the spherical twists on the target and source categories of a spherical functor. 
In~\S\ref{subsection-computation-sdim} we combine this with our formula for the Serre functors of residual categories 
to prove Theorem~\ref{theorem-ci-sdim}, as well as to compute the Serre dimensions of the refined residual categories 
from Proposition~\ref{proposition-2-3-intersection}. 
Finally, in~\S\ref{subsection-serre-invariant} we prove nonexistence of stability conditions invariant under an autoequivalence 
if the upper and lower dimensions of the autoequivalence are distinct, 
and in particular deduce Corollary~\ref{corollary-serre-invariant-stab} from the Introduction. 

\subsection{Dimensions of endofunctors}
\label{subsection-dimensions}

The notion of dimension of an endofunctor relies on the notion of a generator of a triangulated category.

\begin{definition}
\label{definition-generator}
We say that an object~$G$ of a triangulated category~$\cC$ is a {\sf generator} if the 
thick (i.e., idempotent complete) triangulated subcategory generated by~$G$ is equal to~$\cC$. 
\end{definition} 

\begin{example}
By~\cite{BVdB} if~$X$ is a quasi-compact and quasi-separated scheme, then the category~$\Dp(X)$ admits a generator. 
It follows that if~$\cC \subset \Dp(X)$ is a semiorthogonal component, then~$\cC$ also admits a generator. 
This will be our main source of examples, for~$X$ a projective variety and~$\cC$ a residual category. 
\end{example}

Elagin and Lunts~\cite{EL} introduced the following notions of dimension of an endofunctor. 

\begin{definition}
\label{definition-dimF} 
For any objects~$C_1, C_2 \in \cC$ such that~$\Ext^\bullet(C_1,C_2) \ne 0$, let 
\begin{equation*}
e_-(C_1, C_2) = \inf \{ i  \mid   \Ext^i(C_1, C_2) \neq 0 \} \quad \text{and} \quad 
e_+(C_1, C_2) = \sup \{i  \mid  \Ext^i(C_1, C_2) \neq 0 \}. 
\end{equation*} 
Let~$\cC$ be a triangulated category that admits a generator~$G$, 
and let~$F \colon \cC \to \cC$ be a non-nilpotent endofunctor. 
Then the {\sf upper} and {\sf lower~$F$-dimension of~$\cC$} are the 
numbers in~$\RR \cup \{ \pm \infty \}$ defined by 
\begin{equation*}
\uFdim{F} = \limsup_{m \to \infty} \frac{-e_-(G, F^m(G))}{m} 
\quad \text{and} \quad 
\lFdim{F} = \liminf_{m \to \infty} \frac{-e_+(G, F^m(G))}{m}  . 
\end{equation*} 
If we would like to emphasize notationally the category on which~$F$ acts we will write~$\uFdim{\cC,F}$ and~$\lFdim{\cC,F}$ 
for the upper and lower $F$-dimensions. 
\end{definition}

The following result collects some of the main properties of $F$-dimensions;  
the first, in particular, shows that the $F$-dimensions indeed only depend on~$F$ and not the choice of~$G$. 

\begin{proposition}
\label{proposition-properties-Fdim} 
Let~$\cC$ be a triangulated category that admits a generator, 
and let~$F \colon \cC \to \cC$ be a non-nilpotent endofunctor. 
\begin{enumerate}
\item 
\label{Fdim-well-defined}
For any generators $G, G' \in \cC$, we have 
\begin{equation*}
\uFdim{F} = \limsup_{m \to \infty} \frac{-e_-(G, F^m(G'))}{m} 
\quad \text{and} \quad 
\lFdim{F} = \liminf_{m \to \infty} \frac{-e_+(G, F^m(G'))}{m}  . 
\end{equation*} 

\item 
\label{Fdim-inequality}
We have $\lFdim{F} \leq \uFdim{F}$. 

\item 
\label{Fdim-shift} 
If $n \in \bZ$, then 
$\uFdim{F[n]} = \uFdim{F} + n$ and~$\lFdim{F[n]} = \lFdim{F} + n$. 

\item 
\label{Fdim-power} 
If~$p$ is a positive integer, then~$\uFdim{F^p} = p \cdot \uFdim{F}$ and~$\lFdim{F^p} = p \cdot \lFdim{F}$. 

\item 
\label{Fdim-inverse} 
If~$\cC$ admits a Serre functor and~$F$ is an autoequivalence, 
then~$\uFdim{F^{-1}} = - \lFdim{F}$ and~$\lFdim{F^{-1}} = - \uFdim{F}$. 
\end{enumerate}
\end{proposition} 

\begin{proof}
\eqref{Fdim-well-defined} is~\cite[Lemma 6.3]{EL}, 
\eqref{Fdim-inequality} is~\cite[Lemma 6.5]{EL}, 
and~\eqref{Fdim-shift} and~\eqref{Fdim-power} follow easily from the definitions, 
and~\eqref{Fdim-inverse} is \cite[Proposition 6.7]{EL}. 
\end{proof} 

There is a refinement of the upper and lower dimensions of an endofunctor, known as 
its \emph{entropy}, which was introduced in~\cite{entropy}. 
We briefly recall this notion, for use in our proof of Theorem~\ref{theorem-sdim-bounds} below. 
Entropy is defined for an endofunctor of any category which admits a generator, 
but in~\cite{entropy} the general definition is proved to be equivalent to a simpler one in the case of smooth and proper categories; 
as it is all we shall need, we restrict to the smooth and proper case and present the simpler definition.

For the definition and basic properties of smooth and proper categories we refer to~\cite[\S4]{NCHPD}. 
In particular, we will use the fact that 
any semiorthogonal component of the derived category of a smooth projective variety is smooth and proper, 
and that any smooth and proper category admits a generator~\cite[Lemma~2.6]{toen} 
and a Serre functor~\cite[Lemma~4.19]{NCHPD}. 

\begin{definition}
\label{definition-entropy}
Let~$\cC$ be a smooth and proper category, and let~$F \colon \cC \to \cC$ be {a non-nilpotent} endofunctor. 
The {\sf entropy} of~$F$ is the function~$\RR \to [-\infty, +\infty)$ 
defined by 
\begin{equation*}
h_t(F) = \lim_{N \to \infty} \frac{1}{N} \log \left( \sum_{{n = -\infty}}^{{\infty}} \dim \Ext^n(G, F^N(G))e^{-nt}  \right) 
\end{equation*} 
where~$G$ is any generator of~$\cC$, and~$t$ is the argument of the function,
and convergence is pointwise.
\end{definition} 

The fact that~$h_t(F)$ indeed takes values in~$[-\infty, +\infty)$ 
and does not depend on the choice of generator~$G$ 
is proved in~\cite[Lemma~2.5 and Theorem~2.6]{entropy}.

\begin{proposition}[{\cite[Proposition~6.13]{EL}}]
\label{proposition-entropy-sdim}
Let~$\cC$ be a smooth and proper category, and let~$F \colon \cC \to \cC$ be a non-nilpotent endofunctor. 
Then for any generators~$G, G' \in \cC$, we have 
\begin{equation*}
\uFdim{F} = \lim_{t \to +\infty} \frac{h_t(F)}{t},
\qquad 
\lFdim{F} = \lim_{t \to -\infty} \frac{h_t(F)}{t},
\end{equation*}
and both of these limits are finite.
\end{proposition} 

In order to compute Serre dimensions of residual categories using Corollary~\ref{corollary-complete-intersection}, 
we will need the following consequence of~\cite{Kim} 
relating the dimensions of the spherical twists on the target and source of a spherical functor. 

Recall the notion of ind-conservativity of a functor (Definition~\ref{definition-ind-conservative}).

\begin{theorem}
\label{theorem-sdim-bounds}
Let~$\Psi \colon \cC \to \cD$ be a spherical functor between triangulated categories which
admit generators. 
\begin{enumerate}
\item \label{Psi!-generator}
If~$\Psi$ is ind-conservative,
then 
\begin{equation*}
\lFdim{\cD,\bT_{\Psi, \Psi^!}} \leq 
\lFdim{\cC,\bT_{\Psi^!, \Psi}} + 2 \leq 
\uFdim{\cC,\bT_{\Psi^!, \Psi}} + 2 \leq 
\uFdim{\cD,\bT_{\Psi, \Psi^!}}.
\end{equation*} 

\item \label{kerPsiPsi!}
If~$\Psi^!$ is not conservative, i.e. if~$\ker(\Psi^!) \neq 0$,
then 
\begin{equation*}
\lFdim{\cD,\bT_{\Psi, \Psi^!}} \leq 0 \leq  \uFdim{\cD,\bT_{\Psi, \Psi^!}}. 
\end{equation*} 

\item \label{smoothproper}
If~$\Psi$ is ind-conservative 
and~$\cC$ and~$\cD$ are smooth and proper, then 
\begin{align*} 
\uFdim{\cD,\bT_{\Psi, \Psi^!}} & \leq   \max \{ 0, \uFdim{\cC, \bT_{\Psi^!, \Psi}} + 2 \}, 
\\
\lFdim{\cD,\bT_{\Psi, \Psi^!}} & \geq \,\min \{ 0, \lFdim{\cC, \bT_{\Psi^!, \Psi}} + 2 \}.
\end{align*}
\end{enumerate}
\end{theorem} 

\begin{proof}
\eqref{Psi!-generator} 
Let~$G$ be a generator of~$\cD$ and let~$G'$ be a generator of~$\cC$. 
Then for any~$i \in \bZ$ and~$m \geq 1$ we have 
\begin{align*}
\Ext^i(G, \bT_{\Psi, \Psi^!}^m(G \oplus \Psi(G'))) & \supset 
\Ext^i(G, \bT_{\Psi, \Psi^!}^m (\Psi(G')) ) \\ 
& \cong \Ext^i(G, \Psi ((\bT_{\Psi^!, \Psi}[2])^m(G')) ) \\ 
& \cong \Ext^i(\Psi^*(G), (\bT_{\Psi^!, \Psi}[2])^m(G')), 
\end{align*} 
where the first isomorphism follows from the isomorphism~\eqref{eq:spherical-intertwining} of Proposition~\ref{prop:spherical-intertwining}. 
Note that~$G \oplus \Psi(G')$ is a generator for~$\cD$, 
and~$\Psi^*(G)$ is also a generator for~$\cC$ by the ind-conservativity of~$\Psi$ combined with Lemma~\ref{lemma-generator-adjoint}.
Now the claim~\eqref{Psi!-generator} follows directly from the above inclusion
and Proposition~\ref{proposition-properties-Fdim}\eqref{Fdim-well-defined}--\eqref{Fdim-power}.

\smallskip \noindent 
\eqref{kerPsiPsi!} 
By the definition of~$\bT_{\Psi, \Psi^!}$, the condition~$E \in \ker(\Psi^!)$ 
implies that~$E \cong \bT_{\Psi, \Psi^!}(E)$. 
Let~$E$ be a nonzero such object, and let~$G$ be a generator for~$\cD$. 
Then for any~$i \in \bZ$ and~$m \geq 1$, we have 
\begin{equation*}
\Ext^i(G, \bT_{\Psi, \Psi^!}^m(G \oplus E)) \supset \Ext^i(G, E) . 
\end{equation*} 
Note that since~$G$ is a generator and~~$E$ is nonzero, the right side is nonzero for some~$i$. 
The above inclusion then directly implies the claim~\eqref{kerPsiPsi!}. 

\smallskip \noindent 
\eqref{smoothproper} 
As mentioned in the proof of~\eqref{Psi!-generator}, 
ind-conservativity of $\Psi$ implies that $\Psi^*$ takes a generator of $\cD$ to a generator of $\cC$, 
and thus by Proposition~\ref{proposition:spherical-criterion}\eqref{item:iso-1} the same is true for $\Psi^!$. 
Therefore, by~\cite[Theorem~1.6]{Kim} 
we have the following bound on entropy: 
\begin{equation*}
h_t(\bT_{\Psi, \Psi^!}) \leq \max\{ 0, h_t(\bT_{\Psi^!, \Psi}[2]) \}.
\end{equation*} 
If~$\cC$ and~$\cD$ are smooth and proper, then together with Proposition~\ref{proposition-entropy-sdim} this 
implies the claim~\eqref{smoothproper}. 
\end{proof} 

\begin{remark}
\label{remark-Kim}
In~\cite[Theorem~1.7]{Kim}, Kim proves a result at the 
level of entropy functions parallel to parts~\eqref{Psi!-generator} and~\eqref{kerPsiPsi!} of Theorem~\ref{theorem-sdim-bounds}. 
We provided a direct argument above for the convenience of the reader, but 
also because our assumption in~\eqref{Psi!-generator} is weaker than Kim's; further, 
to use Kim's result to deduce a statement about dimensions would require a smooth and properness assumption 
(as in part~\eqref{smoothproper} of Theorem~\ref{theorem-sdim-bounds}), but our direct argument works without such an assumption. 
It would be interesting to show that part~\eqref{smoothproper} of Theorem~\ref{theorem-sdim-bounds} holds without any smooth and properness assumption; 
in particular, this would obviate the need for the smoothly attainable assumption in Theorem~\ref{theorem-ci-sdim}.
\end{remark} 

Under some mild assumptions, Theorem~\ref{theorem-sdim-bounds} gives the following computation of the dimensions of the spherical twist associated to a spherical object. 
Using Proposition~\ref{proposition-entropy-sdim}, this also follows from the main result of~\cite{ouchi}, but for illustration we deduce the result from Theorem~\ref{theorem-sdim-bounds}.

\begin{corollary}
\label{corollary-dim-TP}
Let $\cD$ be a smooth and proper category, and let ${\rP} \in \cD$ be a $d$-spherical object where 
$d \geq 1$. 
If $d > 1$ assume that ${\rP}^{\perp} \neq 0$. 
Then 
\begin{equation*}
\uFdim{\bT_{{\rP}}} = 0 \qquad \text{and} \qquad 
\lFdim{\bT_{{\rP}}} = 1- d. 
\end{equation*} 
\end{corollary} 

\begin{proof} 
Let~$\Psi \colon \Db(\kk) \to \cD$ be the spherical functor associated to~${\rP}$, given by~$\Psi(\kk) = {\rP}$. 
Note that by definition~$\bT_{{\rP}} = \bT_{\Psi, \Psi^!}$, 
while~$\bT_{\Psi^!, \Psi} = [-d-1]$ by Lemma~\ref{lemma:spherical-collection}. 
In particular 
\begin{equation*}
\uFdim{\bT_{\Psi^!, \Psi}} = \lFdim{\bT_{\Psi^!, \Psi}} = -d-1. 
\end{equation*} 
The functor~$\Psi$ is ind-conservative by Lemma~\ref{lemma-ind-conservative-Dbk}.
On the other hand, $\Psi^! = \Ext^\bullet({\rP}, -)$ by (the proof of) Lemma~\ref{lemma:spherical-collection}, 
so if~$d > 1$ then~$\Psi^!$ is not conservative by the assumption~${\rP}^{\perp} \neq 0$. 
Thus if~$d > 1$ the hypotheses of all parts of Theorem~\ref{theorem-sdim-bounds} hold, 
and if~$d = 1$ then the hypotheses of parts~\eqref{Psi!-generator} and~\eqref{smoothproper} hold. 
In either case, it is easy to see that the inequalities of the theorem, together with 
the dimensions of~$\bT_{\Psi^!, \Psi}$ mentioned above, imply the desired formulas 
for the dimensions of~$\bT_{{\rP}}$. 
\end{proof}

\subsection{Computation of Serre dimensions} 
\label{subsection-computation-sdim} 

In this subsection we prove Theorem~\ref{theorem-ci-sdim}
by induction on the codimension of the complete intersection~$X \subset \P^n$, 
combining {Corollary~\ref{corollary-complete-intersection}}
and Theorem~\ref{theorem-sdim-bounds} for the induction step. 

First, we recall the precise definition of Serre dimensions, a special case of Definition~\ref{definition-dimF}.

\begin{definition}
\label{def:sdim}
Let~$\cC$ be a proper 
triangulated category which admits a generator and a Serre functor~$\bS_{\cC}$. 
Then the {\sf upper} and {\sf lower Serre dimensions} are defined by 
\begin{equation*}
\usdim(\cC) = \uFdim{{\cC}, \bS_{\cC}} 
\quad \text{and} \quad 
\lsdim(\cC) = \lFdim{{\cC}, \bS_{\cC}}
\end{equation*} 
\end{definition}

Now let~$X \subset \P^n$ be a smooth complete intersection of type~$(d_1, d_2, \dots, d_k)$
and assume the degrees are ordered so that 
\begin{equation}
\label{eq:d-ordering}
d_1 \geq d_2 \geq \cdots \geq d_k \geq 2. 
\end{equation} 
Note that up to replacing~$\bP^n$ with a linear subspace, there is no loss of generality in assuming~\mbox{$d_k \geq 2$}. 
Recall the notion of smooth attainability, Definition~\ref{def:smoothly-attainable}. 
We note that in characteristic zero this condition is automatic.

\begin{lemma}
\label{lemma:m-smooth}
If the base field has characteristic zero and~$X \subset \P^n$ is a smooth complete intersection of type~$(d_1,d_2,\dots,d_k)$, then~$X$ is smoothly attainable.
\end{lemma}

\begin{proof}
We construct a chain of smooth complete intersections~$X_l \subset \P^n$ by induction on~$l$.
In the base case~$l = 0$ there is nothing to construct, so assume~\mbox{$1 \le l \le k$}.
By induction assumption there is a smooth complete intersection~$X_{l-1} \subset \P^n$ of type~$(d_1,\dots,d_{l-1})$ containing~$X$.
Then~$X \subset X_{l-1}$ is a complete intersection of divisors of degrees~$d_l \ge \dots \ge d_k$.
In particular, if~$I \subset \cO_{X_{l-1}}$ is the ideal sheaf of~$X$ then the sheaf~$I(d_l)$ is globally generated.
Therefore, by Bertini's Theorem a general divisor~$X_l$ in~$X_{l-1}$ of degree~$d_l$ containing~$X$ is smooth away from~$X$.
Now let us show that a general such divisor is smooth along~$X$.
Indeed, the twisted conormal bundle of~$X$ in~$X_{l-1}$ is
\begin{equation*}
I/I^2(d_l) \cong \cO_X \oplus \cO_X(d_l - d_{l+1}) \oplus \cdots \oplus \cO_X(d_l - d_k)
\end{equation*}
and the morphism $H^0(X_{l-1},I(d_l)) \to H^0(X,I/I^2(d_l))$ is surjective.
Therefore, for general~$X_l$ the image of the corresponding section of~$I/I^2(d_l)$ has non-zero component in the first summand~$\cO_X$,
hence it vanishes nowhere, and hence~$X_l$ is smooth along~$X$.
\end{proof}

Now we are ready to prove the theorem.

\begin{proof}[Proof of Theorem~\textup{\ref{theorem-ci-sdim}}]
We may assume the degrees~$(d_1, d_2, \dots, d_k)$ of the complete intersection~$X \subset \P^n$ 
are in decreasing order and greater than~$1$, i.e.~\eqref{eq:d-ordering} holds. 
As in Setup~\ref{setup-ci}, we fix a presentation~$X = M \cap Y$ where~$M \subset \P^n$ is a complete intersection of 
type~$(d_1, \dots, d_{k-1})$ and~$Y \subset \P^n$ is a hypersurface of degree~$d_k$; 
moreover, by smooth attainability of $X$ we may assume that~$M$ is smooth.

We prove the result by induction on the codimension~$k$ of~$X \subset \P^n$. 
For the base case~$k = 1$, we have~$M = \P^n$ and~$\cR_M = 0$. 
In this situation, Corollary~\ref{cor:serre-ci-pn} simplifies to give 
\begin{equation*}
\bS_{\cR_X}^{d_1/c} \cong \left[ \frac{d_1\dim(X) - 2 \ind(X)}{c} \right],  
\end{equation*} 
and thus parts~\eqref{Fdim-shift} and~\eqref{Fdim-power} of 
Proposition~\ref{proposition-properties-Fdim} give the desired equality 
\begin{equation*} 
\usdim(\cR_X) = \lsdim(\cR_X) = \dim(X) - 2\frac{\ind(X)}{d_1} . 
\end{equation*} 
In fact, as we mentioned in the Introduction, 
this fractional Calabi--Yau property for hypersurfaces in~$\P^n$ already follows from~\cite{K04} or~\cite{K19}. 

Now assume the result is proved for~$k-1$, i.e., for~$M$. 
Corollary~\ref{corollary-complete-intersection} combined with Proposition~\ref{proposition-properties-Fdim} gives
\begin{align}
\label{uT-RM} 
\usdim(\cR_M) & = \frac{\ind(M)}{d_k} \uFdim{\cR_M, \bT_{\Psi_{\cR}^!,\Psi_{\cR}}} + \dim(M) ,   \\ 
\label{lT-RM} 
\lsdim(\cR_M) & = \frac{\ind(M)}{d_k} \lFdim{\cR_M, \bT_{\Psi_{\cR}^!,\Psi_{\cR}}} + \dim(M) , \\  
\label{uT-RX} 
\usdim(\cR_X) & = \frac{\ind(X)}{d_k} \uFdim{\cR_X, \bT_{\Psi_{\cR},\Psi_{\cR}^!}} + \dim(X) - 2 \frac{\ind(X)}{d_k}, \\ 
\label{lT-RX} 
\lsdim(\cR_X) & = \frac{\ind(X)}{d_k} \lFdim{\cR_X, \bT_{\Psi_{\cR},\Psi_{\cR}^!}} + \dim(X) - 2 \frac{\ind(X)}{d_k}. 
\end{align}
By the choice of ordering~\eqref{eq:d-ordering} the maximal and minimal degrees for~$M$ are~$d_1$ and~$d_{k-1}$, respectively.
Therefore, by the induction hypothesis, we have 
\begin{align*}
\usdim(\cR_M) & = \dim(M) - 2 \frac{\ind(M)}{d_{1}}  , \\ 
\lsdim(\cR_M) & = \dim(M) - 2 \frac{\ind(M)}{d_{k-1}}. 
\end{align*} 
Plugging this into~\eqref{uT-RM} and~\eqref{lT-RM} and taking~\eqref{eq:d-ordering} into account, we find
\begin{align*}
\uFdim{\cR_M, \bT_{\Psi_{\cR}^!,\Psi_{\cR}}} & = -2\frac{d_k}{d_1}       \ge -2, \\ 
\lFdim{\cR_M, \bT_{\Psi_{\cR}^!,\Psi_{\cR}}} & = -2\frac{d_k}{d_{k-1}}   \ge -2. 
\end{align*} 
Since both~$X$ and~$M$ are smooth and proper, we conclude from Theorem~\ref{theorem-sdim-bounds}\eqref{smoothproper} that
\begin{align*}
\uFdim{\cR_X, \bT_{\Psi_{\cR},\Psi_{\cR}^!}} & \le \uFdim{\cR_M, \bT_{\Psi_{\cR}^!,\Psi_{\cR}}} + 2, \\
\lFdim{\cR_X, \bT_{\Psi_{\cR},\Psi_{\cR}^!}} & \ge 0.\\
\intertext{On the other hand, by Proposition~\ref{proposition-ker-Psi-Ind} and Theorem~\ref{theorem-sdim-bounds}\eqref{Psi!-generator} we have}
\uFdim{\cR_X, \bT_{\Psi_{\cR},\Psi_{\cR}^!}} & \ge \uFdim{\cR_M, \bT_{\Psi_{\cR}^!,\Psi_{\cR}}} + 2,\\
\intertext{and by Proposition~\ref{proposition-kerPsiM!-nonzero} and Theorem~\ref{theorem-sdim-bounds}\eqref{kerPsiPsi!} we have}
\lFdim{\cR_X, \bT_{\Psi_{\cR},\Psi_{\cR}^!}} & \le 0.
\end{align*} 
Therefore, all the inequalities above are equalities, so that we have 
\begin{equation*}
\uFdim{\cR_X, \bT_{\Psi_{\cR},\Psi_{\cR}^!}} = -2\frac{d_k}{d_1} + 2 
\qquad\text{and}\qquad 
\lFdim{\cR_X, \bT_{\Psi_{\cR},\Psi_{\cR}^!}} = 0. 
\end{equation*}
Plugging this into~\eqref{uT-RX} and~\eqref{lT-RX} gives the desired formulas for $\usdim(\cR_X)$ and $\lsdim(\cR_X)$. 
This completes the induction. 
\end{proof} 

One simple consequence of this theorem mentioned in the Introduction is non-geometricity of most residual categories 
of smooth Fano complete intersections.

\begin{corollary}
\label{corollary-geometricity}
Let~$X \subset \P^n$ be a type~$(d_1, d_2, \dots, d_{k})$ smoothly attainable Fano complete intersection, with all~$d_i > 1$.
If~$\cR_X \simeq \Db({Z})$ for a variety~${Z}$, then all~$d_i$ are equal to an integer~$d$ which divides~$2(n+1)$ 
and~$\dim {(Z)} = n+k - 2(n+1)/d$. 
\end{corollary} 
\begin{proof}
This follows from a combination of Theorem~\ref{theorem-ci-sdim} with the formula~\eqref{SdimX} for the Serre dimensions of~$\Db({Z})$.
\end{proof}

A similar technique can be used to compute Serre dimensions of residual categories in other examples, 
like complete intersections in weighted projective spaces or in the varieties listed in \S\ref{subsec:simple-cr}; 
we leave it to the reader to formulate such results. 
In cases where we have a refined residual category, like for the Fano divisors in a quadric 
considered in Proposition~\ref{proposition-2-3-intersection}, we can also apply a similar argument to compute Serre dimensions. 
As an example we prove the following result. 

\begin{proposition}
\label{proposition-2-3-serredim}
Assume the base field is algebraically closed of characteristic not equal to~$2$, 
$n \ge 5$ is odd, 
and~$X \subset Q \subset \P^n$ is a smooth divisor of degree~$n - 2$ in a smooth quadric hypersurface~$Q$. 
If~$\cA_X \subset \Db(X)$ is the refined residual category defined by the decomposition 
\begin{equation*} 
\Db(X) = \langle \cA_X, \cS_+\vert_X, \cO_X \rangle, 
\end{equation*} 
then its Serre dimensions are given by 
\begin{equation*}
\usdim(\cA_X) = 2n - 7
\qquad \text{and} 
\qquad 
\lsdim(\cA_X) = \frac{(n-2)^2 - 2}{n-2} .
\end{equation*}  
\end{proposition} 

\begin{proof} 
Proposition~\ref{proposition-2-3-intersection} combined with Proposition~\ref{proposition-properties-Fdim} gives
\begin{align}
\label{udimAX} (n-2) \usdim(\cA_X) & = \frac{3-n}{2} \lFdim{\bT_{\cK}} + (n-2)^2 - 2, \\ 
\label{ldimAX} (n-2) \lsdim(\cA_X) & = \frac{3-n}{2} \uFdim{\bT_{\cK}} + (n-2)^2 - 2, 
\end{align} 
where $\cK$ is the $(2n-7)$-spherical object defined by the triangle~\eqref{eq:ck}. 

We claim that 
\begin{equation} 
\label{dimTK}
\uFdim{\bT_{\cK}} = 0 \qquad \text{and}  \qquad \lFdim{\bT_{\cK}} = 8 - 2n. 
\end{equation} 
By Corollary~\ref{corollary-dim-TP}, it suffices to show there exists~$0 \neq E \in \cA_X$ such that~$\Ext^\bullet(\cK, E) = 0$. 
By Proposition~\ref{proposition-kerPsiM!-nonzero} there exists~$0 \neq E \in \cR_X$ such that~$\Psi_{\cR}^!(E) = 0$, 
where~\mbox{$\Psi_{\cR} \colon \cR_Q \to \cR_X$} is the 
functor between the residual categories induced by~$\Psi = i^* \colon \Db(Q) \to \Db(X)$, 
where~$i \colon X \hookrightarrow Q$ is the embedding. 
As~$\cR_Q = \langle \cS_{+}, \cS_{-} \rangle$, by adjunction~$\Psi_{\cR}^!(E) = 0$ 
implies~$E \in \langle  \cS_{+} \vert_X, \cS_{-} \vert_X \rangle^{\perp}$. 
But~$\cR_X = \langle \cA_X,  \cS_+\vert_X \rangle$, so this implies~$E \in \cA_X$, 
and further~$\Ext^\bullet(\cK, E) = 0$ by the triangle~\eqref{eq:ck}. 

Finally, plugging~\eqref{dimTK} into~\eqref{udimAX} and~\eqref{ldimAX} 
gives the desired formulas for the Serre dimensions~$\usdim(\cA_X)$ and~$\lsdim(\cA_X)$. 
\end{proof}

\subsection{Serre invariant stability conditions} 
\label{subsection-serre-invariant}

In this subsection we relate Serre dimensions (and more generally dimensions of autoequivalences) to stability conditions, 
and in particular prove Corollary~\ref{corollary-serre-invariant-stab}. 
We start with a brief review of stability conditions; for further background, 
see the original paper~\cite{bridgeland} 
or the recent surveys~\cite{macri-schmidt, BM21}. 

\begin{definition}
\label{def:slicing}
Let~$\cC$ be a triangulated category.
A {\sf slicing}~$\cP$ of~$\cC$ consists of full additive subcategories~$\cP(\phi)$ for~$\phi \in \RR$
(objects in~$\cP(\phi)$ are called {\sf semistable of phase~$\phi$}), 
such that: 
\begin{enumerate}
\item 
\label{item:slicing:shift}
$\cP(\phi + 1) = \cP(\phi) [1]$ for all $\phi \in \RR$, 
\item 
\label{item:slicing:hom}
$\Hom(F_1, F_2) = 0$ if~$F_1 \in \cP(\phi_1)$, $F_2 \in \cP(\phi_2)$, and~$\phi_1 > \phi_2$, 
\item  
\label{item:slicing:hn}
every~$0 \ne F \in \cC$ admits a {\sf Harder--Narasimhan (HN) filtration}, i.e. there exists a sequence of morphisms
\begin{equation*}
0 = F_0 \to F_1 \to \cdots \to F_{n} = F 
\end{equation*} 
such that the~$\Cone(F_{i-1} \to F_i) \in \cP(\phi_i)  \setminus 0$ for some~$\phi_1 > \phi_2 > \cdots > \phi_n$. 
\end{enumerate} 
\end{definition}

The HN filtration of~$F$ is uniquely determined, and we write~$\phi^+(F) \coloneqq \phi_1$ and~$\phi^-(F) \coloneqq \phi_n$ 
for the largest and smallest phases of the HN factors. 
Note that for objects $F, G \in \cC$, we have 
\begin{equation}
\label{eq:phases-hom}
\phi^-(F) > \phi^+(G)
\implies 
\Hom(F, G) = 0
\end{equation}
by Definition~\ref{def:slicing}\eqref{item:slicing:hom}.

Now fix a group homomorphism $v \colon \rK_0(\cC) \to \Lambda$ from the Grothendieck group of $\cC$ 
to a finite rank free abelian group $\Lambda$. 

\begin{definition}
\label{definition-prestability}
A {\sf pre-stability condition} on $\cC$ with respect to $v$ consists of a pair $\sigma = (\cP, Z)$, where $\cP$ is a 
slicing of $\cC$ and $Z \colon \Lambda \to \CC$ is a group homomorphism satisfying the following compatibility condition: 
\begin{equation*}
0 \neq F \in \cP(\phi)  \implies Z(v(F)) \in \RR_{>0} \cdot e^{i\pi\phi} . 
\end{equation*} 
A {\sf stability condition} on $\cC$ is a pre-stability condition $\sigma$ which satisfies the {\sf support property}, i.e. for a fixed norm $\| - \|$ on $\Lambda \otimes_{\ZZ} \RR$, there exists a constant $C > 0$ such that 
\begin{equation*}
F \in \cP(\phi) \implies \| v(F) \| \leq C | Z(v(F)) |. 
\end{equation*} 
\end{definition} 

We denote by~$\Stab_{\Lambda}(\cC) \subset \PreStab_\Lambda(\cC)$ 
the sets of all stability and prestability conditions on~$\cC$ with respect to~$v$.

There are two canonical groups acting on~$\PreStab_\Lambda(\cC)$. 
First consider the group of pairs 
\begin{equation*} 
\Aut_{\Lambda}(\cC) \coloneqq \{ (\Phi, a) \in  \Aut(\cC) \times \Aut(\Lambda) ~ | ~ v \circ \Phi^{\rK_0} = a \circ v \} 
\end{equation*} 
where~$\Phi^{\rK_0} \in \Aut(\rK_0(\cC))$ 
denotes the induced automorphism of the Grothendieck group. 
The group~$\Aut_{\Lambda}(\cC)$ acts on the left on~$\PreStab_\Lambda(\cC)$, via 
\begin{equation*}
(\Phi, a) \cdot (\cP, Z) \coloneqq (\Phi \circ \cP, Z \circ a^{-1}), 
\end{equation*} 
where $\Phi \circ \cP$ is the slicing $\Phi(\cP(\phi)), \phi \in \RR$. 

Second consider the universal cover $\tGL^+_2(\RR) \to \GL^+_2(\RR)$ of the group of $2 \times 2$ real matrices with positive determinant. Explicitly, this cover can be described as  
\begin{equation*}
 \tGL^+_2(\RR)  = \left \{ (A, f) ~  \middle\vert ~ 
 \begin{array}{l} 
 A \in \GL^+_2(\RR), f \colon \RR \to \RR \text{ is an increasing function} \\ 
 \text{such that for all $\phi \in \RR$ we have $f(\phi+1) = f(\phi) + 1$} \\
 \text{and $A \cdot e^{i\pi \phi} \in \RR_{>0} \cdot e^{i \pi f(\phi)}$} 
 \end{array}  \right \}.
\end{equation*} 
The group $\tGL^+_2(\RR)$ acts on the right on $\PreStab_\Lambda(\cC)$, via 
\begin{equation*} 
(\cP, Z) \cdot (A, f) \coloneqq (\cP \circ f, A^{-1} \circ Z), 
\end{equation*} 
where~$\cP \circ f$ is the slicing~$\cP(f(\phi)), \phi \in \RR$. 
Note that the actions of~$\Aut_{\Lambda}(\cC)$ and~$\tGL^+_2(\RR)$ on~$\PreStab_\Lambda(\cC)$ commute 
and preserve the subset~$\Stab_{\Lambda}(\cC)$. 

We are interested in autoequivalences that act on~$\PreStab_{\Lambda}(\cC)$ trivially modulo the action of~$\tGL^+_2(\RR)$. 
To be more precise, we introduce the following definition.

\begin{definition}
\label{def:serre-invariant-stab}
Let $\Phi \in \Aut(\cC)$ be an autoequivalence of~$\cC$.   
A pre-stability condition \mbox{$\sigma \in \PreStab_{\Lambda}(\cC)$} is called {\sf $\Phi$-invariant} 
if there exists~$a \in \Aut(\Lambda)$ and~${(A,f)} \in \tGL^+_2(\RR)$ such that 
\begin{equation*}
(\Phi, a) \cdot \sigma = \sigma \cdot {(A,f)},
\end{equation*}
i.e. if the action of $(\Phi, a)$ preserves the $\tGL^+_2(\RR)$-orbit of $\sigma$.  
If $\cC$ admits a Serre functor $\bS_{\cC}$, we say $\sigma$ is {\sf Serre invariant} if it is $\bS_{\cC}$-invariant.   
\end{definition}

Our main observation is that for a given autoequivalence, the existence of an invariant pre-stability condition forces the equality 
of its upper and lower dimensions. 

\begin{proposition}
\label{proposition-Phi-invt}
Let $\cC$ be a smooth and proper category and let $\Phi \in \Aut(\cC)$ be an autoequivalence. 
If there exists a $\Phi$-invariant pre-stability condition $\sigma \in \PreStab_{\Lambda}(\cC)$, then 
\begin{equation*}
\uFdim{\Phi} = \lFdim{\Phi}. 
\end{equation*}
\end{proposition} 

\begin{proof}
Let~$G$ be a generator for the category~$\cC$ and fix some~$m > 0$. 
Note that if~$i \in \ZZ$ is such that~$\Ext^i(G, \Phi^m(G)) \neq 0$, 
then~\eqref{eq:phases-hom} implies that~$\phi_{\sigma}^{-}(G) \leq \phi_{\sigma}^+(\Phi^m(G)) + i$. 
Therefore, 
\begin{equation*}
\phi_{\sigma}^{-}(G) - \phi_{\sigma}^+(\Phi^m(G)) \leq e_{-}(G, \Phi^m(G)). 
\end{equation*} 
On the other hand, suppose~$i \in \ZZ$ is such that~$\Ext^i(\bS_{\cC}^{-1}(G), \Phi^m(G)) \neq 0$. 
Then by Serre duality it follows that~$\Hom(\Phi^m(G)[i], G) \neq 0$, 
and thus~$\phi_{\sigma}^-(\Phi^m(G)) + i \leq \phi_{\sigma}^+(G)$, again by~\eqref{eq:phases-hom}.
Therefore, 
\begin{equation*}
e_+(\bS_{\cC}^{-1}(G), \Phi^m(G)) \leq \phi_{\sigma}^+(G) - \phi_{\sigma}^-(\Phi^m(G)) . 
\end{equation*}  
Subtracting the first inequality from the second and dividing by $m$, we obtain 
\begin{equation*} 
\frac{-e_{-}(G, \Phi^m(G))}{m} - \frac{-e_+(\bS_{\cC}^{-1}(G), \Phi^m(G))}{m} 
\leq \frac{\phi_{\sigma}^+(\Phi^m(G)) - \phi_{\sigma}^-(\Phi^m(G))}{m} + \frac{\phi_{\sigma}^+(G) - \phi_{\sigma}^{-}(G)}{m}.
\end{equation*} 
We claim that $\Phi$-invariance of~$\sigma$ implies
\begin{equation}
\label{eq:lim-phases}
\lim_{m \to \infty} \frac{\phi_{\sigma}^+(\Phi^m(G)) - \phi_{\sigma}^-(\Phi^m(G))}{m} = 0. 
\end{equation} 
This claim implies the result, as by Proposition~\ref{proposition-entropy-sdim} taking the limit~$m \to \infty$ in the above inequality 
(note that~$\bS_{\cC}^{-1}(G)$ is a generator for~$\cC$) then shows $\uFdim{\Phi} \leq \lFdim{\Phi}$, 
while the reverse inequality holds by Proposition~\ref{proposition-properties-Fdim}\eqref{Fdim-inequality}. 

Now we prove~\eqref{eq:lim-phases}.
Let~$a \in \Aut(\Lambda)$ and~${(A,f)} \in \tGL^{+}_2(\RR)$ such that~$(\Phi, a) \cdot \sigma = \sigma \cdot {(A,f)}$. 
We denote by~$f^{\circ m}$ the composition~$f \circ f \circ \dots \circ f$ of~$m$ factors.
Then 
\begin{equation*}
\phi_{\sigma}^+(\Phi^m({G})) = f^{\circ m}(\phi_{\sigma}^{+}({G})) 
\quad \text{and} \quad 
\phi_{\sigma}^-(\Phi^m({G})) = f^{\circ m}(\phi_{\sigma}^-({G})). 
\end{equation*} 
Therefore, to prove the claim it suffices to show that for any $\phi_1, \phi_2 \in \RR$, we have 
\begin{equation}
\label{limit-f}
\lim_{m \to \infty} \frac{f^{\circ m}(\phi_1) - f^{\circ m}(\phi_2)}{m} = 0 . 
\end{equation} 
We will show that this holds for any increasing function~$f$ such that~$f(\phi+1) = f(\phi) + 1$ for all~$\phi \in \RR$.

Indeed, we can find an integer~$b$ such that~$\phi_2 + b \le \phi_1 \le \phi_2 + b + 1$.
Then it follows that~$f^{\circ m}(\phi_2) + b \le f^{\circ m}(\phi_1) \le f^{\circ m}(\phi_2) + b + 1$,
hence
\begin{equation*}
\frac{b}m \le \frac{f^{\circ m}(\phi_1) - f^{\circ m}(\phi_2)}{m} \le \frac{b+1}m
\end{equation*}
and passing to the limit we obtain~\eqref{limit-f}.
\end{proof} 

As a consequence, we deduce the nonexistence of Serre invariant pre-stability conditions 
on the residual categories of Fano complete intersections, as announced in the Introduction. 

\begin{proof}[Proof of Corollary~\textup{\ref{corollary-serre-invariant-stab}}]
Combine Theorem~\ref{theorem-ci-sdim} and Proposition~\ref{proposition-Phi-invt}. 
\end{proof} 
 
We also obtain a similar result for the refined residual components considered in Proposition~\ref{proposition-2-3-intersection}
and Proposition~\ref{proposition-2-3-serredim}. 

\begin{proposition}
\label{proposition-2-3-serreinvtstab}
Assume the base field is algebraically closed of characteristic not equal to~$2$, 
$n \ge 5$ is odd, 
and~$X \subset Q \subset \P^n$ is a smooth divisor of degree~$n - 2$ in a smooth quadric hypersurface~$Q$. 
If~$\cA_X \subset \Db(X)$ is the refined residual category defined by the decomposition 
\begin{equation*} 
\Db(X) = \langle \cA_X, \cS_+\vert_X, \cO_X \rangle, 
\end{equation*} 
then there does not exist a Serre invariant pre-stability condition on $\cA_X$. 
\end{proposition} 

\begin{proof}
Combine Proposition~\ref{proposition-2-3-serredim} and Proposition~\ref{proposition-Phi-invt}. 
\end{proof} 

It is an interesting question whether any stability conditions whatsoever exist on~$\cA_X$, especially in case~$n = 5$. 
Indeed, with the exception of this case, in~\cite{BLMS} stability conditions were constructed 
on the residual categories of all prime Fano threefolds. 

%%%%%%%%%%%%%%%%%%%%%%%%%%%%%%%%%%

\appendix

\section{Spherical functors and semiorthogonal decompositions}
\label{sec:more}

In this appendix we 
prove the following result on the factorization of spherical twists in the presence of a semiorthogonal decomposition, 
which is a mild generalization of results from~\cite{Add16} and~\cite{HLS}.

\begin{proposition}
\label{prop:spherical-sod}
Let~$\Psi \colon \cC \to \cD$ be a spherical functor.
Assume~$\cC = \langle \cC_1, \dots, \cC_r \rangle$ is a semiorthogonal decomposition with admissible components 
such that either of the functors
\begin{equation*}
\bS_\cC \circ \bT_{\Psi^*,\Psi}
\qquad\text{or}\qquad 
\bS_\cC^{-1} \circ \bT_{\Psi^!,\Psi}
\end{equation*}
preserves all the components~$\cC_i$ of~$\cC$.
Then the functors~$\Psi_i \coloneqq \Psi\vert_{\cC_i} \colon \cC_i \to \cD$ are spherical and 
\begin{equation}
\label{eq:twist-factorization}
\bT_{\Psi,\Psi^*} \cong \bT_{\Psi_r,\Psi_r^*} \circ \dots \circ \bT_{\Psi_1,\Psi_1^*}
\qquad\text{and}\qquad
\bT_{\Psi,\Psi^!} \cong \bT_{\Psi_1,\Psi_1^!} \circ \dots \circ \bT_{\Psi_r,\Psi_r^!}.
\end{equation}
\end{proposition}

\begin{proof}
Before starting the proof, we note that the functor~$\bS_\cC^{-1} \circ \bT_{\Psi^!,\Psi}$ 
is inverse to~$\bS_\cC \circ \bT_{\Psi^*,\Psi}$
(this follows from~\eqref{eq:twists-inversion} and~\eqref{eq:serre-commutativity}),
so if one preserves all the components~$\cC_i$, so does the other.
Similarly, the isomorphism~\eqref{eq:twist-factorization} for~$\bT_{\Psi, \Psi^!}$ 
is equivalent to the one for~$\bT_{\Psi,\Psi^*}$ by~\eqref{eq:twists-inversion}. 

Now consider the case $r = 2$. 
In this situation, the result is~\cite[Theorem~4.14]{HLS}, 
whose hypotheses are satisfied by our assumption that $\bS_\cC^{-1} \circ \bT_{\Psi^!,\Psi}$ preserves the components of $\cC$. 
However, for the convenience of the reader, we will sketch an alternative direct proof. 
The fact that~$\Psi_i$ are spherical is proved in the same way as in~\cite{Add16}, so we concentrate on proving the factorization of the spherical twist $\bT_{\Psi, \Psi^*}$. 

Let~$\gamma_i \colon \cC_{i} \to \cC$ be the embedding functors.
Then~$\Psi_i = \Psi \circ \gamma_i$, $\Psi_i^* \cong \gamma_i^* \circ \Psi^*$, and we have an exact triangle
(the decomposition triangle for~$\cC = \langle \cC_1, \cC_2 \rangle$)
\begin{equation}
\label{eq:triangle-gamma}
\gamma_2 \circ  \gamma_2^! \to \id_\cC \to \gamma_1 \circ \gamma_1^*.
\end{equation} 
It follows that we have a commutative diagram
\begin{equation}
\label{eq:diagram-big}
\vcenter{\xymatrix@C=7em{
\bT_{\Psi_1,\Psi_1^*} \ar[r] \ar[d] &
\id_\cD \ar[r]^-{\eta_{\Psi_1,\Psi_1^*}} \ar[d]_{\eta_{\Psi,\Psi^*}} &
\Psi_1 \circ \Psi_1^* \ar@{=}[d]
\\
\Psi \circ \gamma_2 \circ  \gamma_2^! \circ \Psi^* \ar[r] & 
\Psi \circ \Psi^* \ar[r]^-{\Psi\eta_{\gamma_1,\gamma_1^*}\Psi^*} & 
\Psi \circ \gamma_1 \circ \gamma_1^* \circ \Psi^*.
}}
\end{equation}
One can show that 
\begin{equation}
\label{eq:isomorphism-big}
\Psi \circ \gamma_2 \circ  \gamma_2^! \circ \Psi^* \cong \Psi_2 \circ \Psi_2^* \circ \bT_{\Psi_1,\Psi_1^*}
\end{equation}
and that the left vertical arrow in~\eqref{eq:diagram-big} is identified with the 
unit of adjunction for $\Psi_2$.
It follows that its cone is~$\bT_{\Psi_2,\Psi_2^*} \circ \bT_{\Psi_1,\Psi_1^*}[1]$, 
and by the octahedral axiom this is isomorphic to the cone of the middle vertical arrow. 
This proves the factorization $\bT_{\Psi, \Psi^*} \cong \bT_{\Psi_2,\Psi_2^*} \circ \bT_{\Psi_1,\Psi_1^*}$.

Now assume~$r > 2$.
The assumptions of the proposition are fulfilled for the semiorthogonal decomposition~$\cC = \langle \cC_r^\perp, \cC_r \rangle$, 
hence the functors~$\bar\Psi \coloneqq \Psi\vert_{\cC_r^\perp}$ and~$\Psi_r \coloneqq \Psi\vert_{\cC_r}$ are spherical and
\begin{equation*}
\bT_{\Psi,\Psi^*} \cong \bT_{\Psi_r,\Psi_r^*} \circ \bT_{\bar\Psi,\bar\Psi^*}
\qquad\text{and}\qquad
\bT_{\Psi,\Psi^!} \cong \bT_{\bar\Psi,\bar\Psi^!} \circ \bT_{\Psi_r,\Psi_r^!}.
\end{equation*}
Now we claim that the functor~$\bar\Psi$ on~$\cC_r^\perp = \langle \cC_1, \dots, \cC_{r-1} \rangle$ 
satisfies the assumptions of the proposition.

Indeed, let~$\bar\gamma$ be the embedding functor of~$\cC_r^\perp$. 
Composing the defining triangle of~$\bT_{\Psi^!,\Psi}$ with~$\bar\gamma^!$ and~$\bar\gamma$ we obtain an exact triangle
\begin{equation*}
\bar\gamma^! \circ \bT_{\Psi^!, \Psi} \circ \bar\gamma \xrightarrow{\qquad}
\bar\gamma^! \circ \bar\gamma \xrightarrow{\ \bar\gamma^! \eta_{\Psi^!,\Psi} \bar\gamma\ }
\bar\gamma^! \circ \Psi^! \circ \Psi \circ \bar\gamma.
\end{equation*}
The middle term is isomorphic to~$\id_{\cC_r^\perp}$, the last term is isomorphic to~$\bar\Psi^! \circ \bar\Psi$,
and the second arrow is the unit of adjunction; therefore
\begin{equation*}
\bar\gamma^! \circ \bT_{\Psi^!, \Psi} \circ \bar\gamma \cong \bT_{\bar\Psi^!,\bar\Psi}.
\end{equation*}
Using this and~\eqref{eq:right-left} to relate the right and left adjoint functors of~$\bar\gamma$ we obtain 
\begin{equation*}
\bS_{\cC_r^\perp}^{-1} \circ \bT_{\bar\Psi^!,\bar\Psi} \cong
\bS_{\cC_r^\perp}^{-1} \circ \bar\gamma^! \circ \bT_{\Psi^!, \Psi} \circ \bar\gamma \cong 
\bar\gamma^* \circ (\bS_\cC^{-1} \circ \bT_{\Psi^!, \Psi}) \circ \bar\gamma.
\end{equation*}
Now the functor~$\bS_\cC^{-1} \circ \bT_{\Psi^!, \Psi}$ preserves the subcategory $\cC_i$ for~$1 \le i \le r-1$ by assumption, 
and the functor~$\bar\gamma^*$ acts as the identity on it, hence the claim.

Now by induction we conclude that~$\Psi_i \cong \bar\Psi\vert_{\cC_i}$ are spherical for~$1 \le i \le r-1$
and deduce the required formulas for~$\bT_{\Psi,\Psi^*}$ and~$\bT_{\Psi,\Psi^!}$.
\end{proof}

We also formulate a consequence of this proposition for the case of a rectangular Lefschetz collection with residual category.
\begin{corollary}
\label{cor:spherical-sod}
Let~$\Psi \colon \cC \to \cD$ be a spherical functor intertwining between autoequivalences~$\bal_\cC \in \Aut(\cC)$ and~$\bal_\cD \in \Aut(\cD)$.
Assume~$\cC$ has a semiorthogonal decomposition of the form~\eqref{eq:residual-def}
which is Serre compatible of length~$m$ and~$\bT_{\Psi,\Psi^!}$-twist compatible of the same degree~$m$.
Then the functors~$\Psi_\cR \coloneqq \Psi\vert_\cR$ and~$\Psi_{\bal_\cC^i(\cB)} \coloneqq \Psi\vert_{\bal_\cC^i(\cB)}$
are spherical,
there is an isomorphism
\begin{equation*}
\bT_{\Psi,\Psi^!} \circ \bal_{\cD}^m \cong 
\bT_{\Psi_\cR,\Psi_\cR^!} \circ (\bT_{\Psi_\cB,\Psi_\cB^!} \circ \bal_{\cD})^m , 
\end{equation*}
and the autoequivalences~$\bT_{\Psi_\cR,\Psi_\cR^!}$ and~$\bT_{\Psi_\cB,\Psi_\cB^!} \circ \bal_{\cD}$ of~$\cD$ commute.
\end{corollary}

\begin{proof}
First, Serre and twist compatibility of~\eqref{eq:residual-def} implies 
the functor~$\bS_\cC^{-1} \circ \bT_{\Psi^!,\Psi}$ preserves the components~$\cR$ and~$\bal_\cC^i(\cB)$ of~$\cC$.
Therefore, the first statement follows from Proposition~\ref{prop:spherical-sod}.
Moreover, it follows that
\begin{equation}
\label{eq:iso-twists}
\bT_{\Psi,\Psi^!} \cong \bT_{\Psi_\cR,\Psi_\cR^!} \circ \bT_{\Psi_\cB,\Psi_\cB^!} \circ 
\bT_{\Psi_{\bal_\cC(\cB)},\Psi_{\bal_\cC(\cB)}^!} \circ \dots \circ \bT_{\Psi_{\bal_\cC^{m-1}(\cB)},\Psi_{\bal_\cC^{m-1}(\cB)}^!}.
\end{equation}
Further, the composition of the spherical functor~$\Psi_{\bal_\cC^i(\cB)}$ with the equivalence~\mbox{$\bal_\cC^i \colon \cB \to \bal_\cC^i(\cB)$}
can be rewritten as~\mbox{$\Psi \circ \bal_\cC^i \vert_{\cB} \cong  \bal_{\cD}^i \circ \Psi_{\cB}$}
by the intertwining property of~$\Psi$,
hence the corresponding spherical twist can be written as~$\bal_{\cD}^i \circ \bT_{\Psi_\cB,\Psi_\cB^!} \circ \bal_{\cD}^{-i}$,
hence the right side of~\eqref{eq:iso-twists} can be rewritten as
\begin{equation*}
\bT_{\Psi_\cR,\Psi_\cR^!} \circ 
\bT_{\Psi_\cB,\Psi_\cB^!} \circ 
(\bal_{\cD} \circ \bT_{\Psi_\cB,\Psi_\cB^!} \circ \bal_{\cD}^{-1}) \circ \dots \circ 
(\bal_{\cD}^{m-1} \circ \bT_{\Psi_\cB,\Psi_\cB^!} \circ \bal_{\cD}^{1-m}).
\end{equation*}
Canceling the factors~$\bal_\cD^i$ and composing with~$\bal_\cD^m$ on the right we deduce the second statement of the corollary.

Finally, we note that~$\langle \cR, \cB \rangle = \langle \cB, \bal_\cC(\cR) \rangle$ by Lemma~\ref{lemma:residual-moves}, 
hence again by Proposition~\ref{prop:spherical-sod} and the above computation we have
\begin{equation*}
\bT_{\Psi_\cR,\Psi_\cR^!} \circ \bT_{\Psi_\cB,\Psi_\cB^!} \cong 
\bT_{\Psi_\cB,\Psi_\cB^!} \circ 
(\bal_{\cD} \circ \bT_{\Psi_\cR,\Psi_\cR^!} \circ \bal_{\cD}^{-1}) , 
\end{equation*}
and the commutativity of~$\bT_{\Psi_\cR,\Psi_\cR^!}$ and~$\bT_{\Psi_\cB,\Psi_\cB^!} \circ \bal_{\cD}$ follows.
\end{proof}

%%%%%%%%%%%%%%%%%%%%%%%%%%%%%%

\section{Ind-completion} 
\label{appendix-ind} 

In this appendix, we briefly review the notion of ind-completion of a category, 
and formulate in terms of it a criterion for a functor to take a generator to a generator. 

Let~$\cC$ be a small enhanced idempotent complete triangulated category. 
Then we can consider the {\sf ind-comple\-tion}~$\Ind(\cC)$ of~$\cC$, which is a triangulated category 
with infinite direct sums and equipped with a fully faithful functor~$\cC \hookrightarrow \Ind(\cC)$. 
When~$\cC$ is enhanced by a differential graded algebra~$A$, 
i.e., $\cC = \Dp(A)$ is the derived category of perfect differential graded modules over~$A$,
one has~$\Ind(\cC) = \rD(A)$, the derived category of all differential graded modules over~$A$.
When enhancements are taken in the $\infty$-categorical sense, 
an overview of this construction and its properties 
can be found in~\cite[\S2]{HH} or~\cite[\S2 and~\S3]{NCHPD}. 
When categories are assumed to arise as semiorthogonal components of derived categories of varieties, 
this construction is studied in~\cite[\S4.2]{Kuz-bc} (where it is denoted by~$\hat{\cC}$ instead of~$\Ind(\cC)$). 

We summarize below the main properties we need about ind-completion. 
The first three follow from the construction and formal properties of ind-completion from \cite[\S5.3.5]{HTT}, 
see especially \cite[Propositions 5.3.5.10 and 5.3.5.11]{HTT}.
\begin{enumerate}
\item 
\label{Ind-compact} 
The embedding~$\cC \hookrightarrow \Ind(\cC)$ realizes~$\cC$ 
as the subcategory of \emph{compact objects} of~$\Ind(\cC)$. 
\item 
\label{Ind-continuous}
Ind-completion is \emph{functorial} --- any enhanced functor~$\Psi \colon \cC \to \cD$ extends 
to a \emph{continuous} (i.e., commuting with arbitrary direct sums) functor on Ind-completions, 
abusively denoted by the same symbol~$\Psi \colon \Ind(\cC) \to \Ind(\cD)$. 
\item 
\label{ind-adjoints}
Ind-completion \emph{preserves adjointness} --- 
if~$\Phi \colon \cD \to \cC$ is a (right or left) adjoint functor to~\mbox{$\Psi \colon \cC \to \cD$}, 
then~$\Phi \colon \Ind(\cD) \to \Ind(\cC)$ is a (right or left) adjoint functor to~\mbox{$\Psi \colon \Ind(\cC) \to \Ind(\cD)$}. 
\item 
\label{Ind-sod} 
Ind-completion \emph{preserves semiorthogonal decompositions} --- 
if~$\cC = \langle \cA_1, \dots, \cA_n \rangle$ is a semiorthogonal decomposition, 
then there is an induced semiorthogonal decomposition~$\Ind(\cC) = \langle \Ind(\cA_1), \dots, \Ind(\cA_n) \rangle$,
see~\cite[Proposition~4.6]{KL15} and~\cite[Lemma~3.12]{NCHPD}. 
\item 
\label{Ind-Dp}
If~$X$ is a quasi-compact scheme 
with affine diagonal
then~$\Ind(\Dp(X)) = \Dqc(X)$, see \cite[Proposition~3.19]{bzfn} 
(or~\cite[Theorem~3.11]{KL15} for the case of separated schemes of finite type over a field).
\end{enumerate} 

Recall the notion of a generator from Definition~\ref{definition-generator}. 
The following lemma is the key to the utility of ind-completion for studying generators. 

\begin{lemma} 
\label{lemma-generator-ind} 
Let~$\cC$ be an enhanced triangulated category. 
Then an object~$G \in \cC$ is a generator for~$\cC$ if and only if 
for any~$E \in \Ind(\cC)$ the condition~$\Ext^\bullet(G, E) = 0$ implies~$E = 0$. 
\end{lemma} 

\begin{proof}
Using property~\eqref{Ind-compact} of ind-completion above, 
this amounts to \cite[\href{https://stacks.math.columbia.edu/tag/09SR}{Tag~09SR}]{SP}.  
(Beware that what we call a ``generator'' is called there a ``classical generator'', and the term 
``generator'' is reserved for a weaker notion~\cite[\href{https://stacks.math.columbia.edu/tag/09SJ}{Tag~09SJ}]{SP}.) 
\end{proof}

This leads to a simple criterion for an adjoint to preserve generators.

\begin{definition}
\label{definition-ind-conservative} 
We say that a functor~$\Psi \colon \cC \to \cD$ is {\sf ind-conservative} 
if its ind-comple\-tion~$\Psi \colon \Ind(\cC) \to \Ind(\cD)$ is conservative, i.e. 
if~$\ker(\Psi \colon \Ind(\cC) \to \Ind(\cD)) = 0$. 
\end{definition}

\begin{lemma}
\label{lemma-generator-adjoint}
Let~$\Psi \colon \cC \to \cD$ be a functor between enhanced triangulated categories 
which admits a left adjoint~$\Psi^* \colon \cD \to \cC$. 
Assume~$\cD$ admits a generator. 
Then~$\Psi$ is ind-conservative if and only if~$\Psi^*$ is generator preserving, i.e. 
for any generator~$G \in \cD$ its image~$\Psi^*(G) \in \cC$ is a generator for~$\cC$. 
\end{lemma} 

\begin{proof}
Let~$G \in \cD$ be a generator. 
Then by Lemma~\ref{lemma-generator-ind}, we see that the functor~$\Psi$ is ind-conservative if and only if 
the condition~$\Ext^\bullet(G, \Psi(E)) = 0$ implies~$E = 0$ for any~$E \in \Ind(\cC)$. 
As~$\Ext^\bullet(G, \Psi(E)) = \Ext^\bullet(\Psi^*(G), E)$, Lemma~\ref{lemma-generator-ind} 
again shows this is equivalent to~$\Psi^*(G)$ being a generator. 
\end{proof} 

Finally, observe the following simple property.

\begin{lemma}
\label{lemma-ind-conservative-Dbk}
Let $\cD$ be an enhanced category and let $\Psi \colon \Db(\kk) \to \cD$ be a conservative functor.
Then its ind-completion $\Psi \colon \Dqc(\kk) \to \Ind(\cD)$ is conservative. 
\end{lemma} 

\begin{proof}
By property~\eqref{Ind-continuous} above the ind-completion $\Psi \colon \Dqc(\kk) \to \Ind(\cD)$ commutes with arbitrary direct sums.
But~$\Psi(\kk)$ is nonzero by conservativity of~$\Psi \colon \Db(\kk) \to \cD$, 
and every object in~$\Dqc(\kk)$ is a direct sum of shifts of~$\kk$, 
so~$\Psi \colon \Dqc(\kk) \to \Ind(\cD)$ is conservative.  
\end{proof} 

%%%%%%%%%%%%%%%%%%%%%%%%%%%%%%%%%%%

\newcommand{\etalchar}[1]{$^{#1}$}

\end{document}